\documentclass[a4paper,centertags,oneside,12pt]{amsart}
\usepackage{hyperref}
\usepackage{mathrsfs}
\usepackage{appendix}
\usepackage{amssymb}
\usepackage{fancyhdr}
\usepackage{charter}
\usepackage{typearea}
\usepackage{pdfsync}
\usepackage{amsmath,amstext,amsthm,amscd}
\usepackage{mathrsfs}
\usepackage[a4paper,top=3cm,bottom=3cm,left=2cm,right=2cm]{geometry}

\usepackage{enumitem}
\setlist[1]{itemsep=5pt}
\newcommand{\comment}[1]{}
\usepackage{color}
\makeatletter
      \def\@setcopyright{}
      \def\serieslogo@{}
      \makeatother

\newcommand{\Complex}{\mathbb C}
\newcommand{\Real}{\mathbb R}
\newcommand{\N}{\mathbb N}
\newcommand{\Z}{\mathbb Z}

\newcommand{\CP}{\mathbb C\mathbb P}
\newcommand{\ddbar}{\overline\partial}
\newcommand{\pr}{\partial}
\newcommand{\ol}{\overline}
\newcommand{\Td}{\widetilde}
\newcommand{\norm}[1]{\left\Vert#1\right\Vert}
\newcommand{\abs}[1]{\left\vert#1\right\vert}

\newcommand{\set}[1]{\left\{#1\right\}}
\newcommand{\To}{\rightarrow}

\DeclareMathOperator{\supp}{supp}
\DeclareMathOperator{\dist}{dist}

\newcommand{\mU}{\mathcal{U}}
\newcommand{\mV}{\mathcal{V}}
 \def\cC{\mathscr{C}}

\theoremstyle{plain}
\newtheorem{thm}{Theorem}[section]

\newtheorem{lem}[thm]{Lemma}
\newtheorem{prop}[thm]{Proposition}

\theoremstyle{definition}
\newtheorem{defn}[thm]{Definition}

\theoremstyle{remark}
\newtheorem{rem}[thm]{Remark}

\numberwithin{equation}{section}

\begin{document}
\title[Szeg\H{o} kernel asymptotics and Kodaira embedding theorems]
{Szeg\H{o} kernel asymptotics and Kodaira embedding theorems of Levi-flat
CR manifolds}
\author{Chin-Yu Hsiao}
\address{Institute of Mathematics \\ Academia Sinica and National Center for
Theoretical Sciences\\ 6F,  Astronomy-Mathematics Building,
No.1, Sec.4, Roosevelt Road, Taipei 10617, Taiwan}

  \email{chsiao@math.sinica.edu.tw}

\author{George Marinescu}
\address{Univerisit\"at zu K\"oln, Mathematisches institut,
Weyertal 86-90, 50931 K\"oln, Germany 
\newline\mbox{\quad}\,Institute of Mathematics `Simion Stoilow', 
Romanian Academy, Bucharest, Romania}
\email{gmarines@math.uni-koeln.de}

\pagestyle{fancy}
\lhead{\itshape{Chin-Yu Hsiao \& George Marinescu}}
%\chead{}
\rhead{\itshape{Szeg\H{o} kernel asymptotics and Kodaira embedding theorems }}
\cfoot{\thepage}

\begin{abstract}
Let $X$ be an orientable compact Levi-flat CR manifold 
and let $L$ be a positive CR complex line bundle over $X$. 
We prove that certain microlocal conjugations of the associated Szeg\H{o} kernel admits 
an asymptotic expansion with respect to high powers of $L$. 
As an application, we give a Szeg\H{o} kernel proof of the Kodaira type embedding 
theorem on Levi-flat CR manifolds due to Ohsawa and Sibony.
\end{abstract}

\maketitle \tableofcontents

\section{Introduction and statement of the main results} \label{s-intro}

The problem of global embedding CR manifolds is prominent in areas such as complex analysis, 
partial differential equations and differential geometry. 
A general result is the CR embedding of strictly pseudoconvex compact CR manifolds of dimension
greater than five,
due to Boutet de Monvel \cite{BdM1:74b}.

For CR manifolds which are not strictly pseudoconvex, the idea of embedding CR manifolds by means of CR sections of 
tensor powers $L^k$ of a positive CR line bundle $L\To X$ was considered in 
\cite{Hs15,Hs14,HsM:12,Mar96,OS00}. This was of course inspired by Kodaira's embedding theorem.

One way to attack this problem is to produce CR sections by projecting appropriate
smooth sections to the space of CR sections. So it is crucial to understand the large $k$ behaviour 
of the Szeg\H{o} projection $\Pi_k$, i.\,e.\ the orthogonal projection on space $H^0_b(X,L^k)$ of CR sections,
and of its distributional kernel, the Szeg\H{o} kernel. 
To study the Szeg\H{o} projection it is convenient to link it to a parametrix of the $\ol\partial_b$-Laplacian
on $(0,1)$-forms (called Kohn Laplacian). This is also the method used in \cite{BdM1:74b}, where the
parametrix turns out to be a pseudodifferential operator of order $1/2$. 

In~\cite{HsM:12}, we established analogues of the holomorphic Morse inequalities of Demailly \cite{De:85,MM07}
for CR manifolds and we deduced that the space $H^0_b(X,L^k)$ is large under 
the assumption that the curvature of the line bundle is adapted to the Levi form. 
In~\cite{Hs15}, the first author introduced 
a microlocal cut-off function technique and could remove 
the assumptions linking the curvatures of the line bundle and the Levi form 
under rigidity conditions on $X$ and the line bundle. 
Moreover, in~\cite{Hs14}, the first author established partial Szeg\H{o} 
kernel asymptotic expansions and Kodaira embedding theorems on CR manifolds 
with transversal CR $S^1$-action, see also \cite{HLM16}. 

All these developments need the assumptions that either the curvature of the line bundle is 
adapted to the Levi form or rigidity conditions on $X$ and the line bundle. 
The difficulty of this kind of problem comes from the presence of positive 
eigenvalues of the curvature of the line bundle and negative eigenvalues 
of the Levi form of $X$. Thus, it is very interesting to consider 
Levi-flat CR manifolds. In this case, the eigenvalues of the Levi form are zero and we will show that
it is possible to remove the assumptions linking the curvatures of the line bundle and the Levi form
or the rigidity conditions on $X$ and the line bundle. 

Levi-flat CR manifolds are foliated
by complex manifolds and there is a subtle interplay between the
function theory on the leaves and the dynamics of the foliation.
Levi-flat CR submanifolds in projective manifolds play an important role
in classical complex analysis \cite{Gr:58,Gra:63,Na:63,Nm:99} linked to the Levi problem,
foliations and dynamical systems
\cite{Ad:14,Barret:90,Bru08,Cal:07,CLNS:88,CanGon:15,
DeDu:16,In92,LN99,Oh:06a,Oh:07,Oh12,OS00,Si00}.
They admit Lefschetz pencil structures of degree $k$, for any $k$ 
large enough, cf.\ \cite{MarTor:11}.
The topology and  dynamics of Levi-flat hypersurfaces in complex surfaces of general 
type was thoroughly explored in \cite{DeDu:16}, 
%{\color{red}
where it is shown that
all possible Thurston geometries except the spherical one can occur.
In this context it is important to have a general criterion for the projective
Levi-flat manifolds, analogue to the Kodaira embedding theorem for K\"ahler manifolds.
This is provided by Ohsawa-Sibony theorem \cite{OS00}, see Theorem \ref{t-embleviflat}.
A related result is the projective embedding of compact laminations 
\cite{Der08}, \cite[p.\,401--402]{Gro:99}.
%}
In the program of classifying Levi-flat CR manifolds one is sometimes
led to non-existence results.
There are no compact Levi-flat real hypersurfaces
in a Stein manifold, due to the maximum principle. 
On the other hand, the non-existence of smooth Levi-flat hypersurfaces in 
complex projective spaces $\CP^n$ attracted a lot of attention, 
cf.\ \cite{LN99,Si00}.
The non-existence has been settled for $n\geq3$ but a famous still open conjecture 
is whether this is true for $n=2$.

Viewing Levi-flat CR manifolds as families of complex manifolds,
we can expect analogy with classical results from complex geometry
such as Kodaira embedding theorem.
The natural function theoretical objects on a CR manifold
are CR functions or CR sections of a bundle. 
Actually, Ohsawa and Sibony~\cite{OS00}, cf.\ also \cite{Oh12}, constructed a CR 
projective embedding of class $\cC^\kappa$ 
for any  $\kappa\in\mathbb N$
of a Levi-flat CR manifold by using $\ol\partial$-estimates. 
%They showed that for a compact Levi-flat 
%CR manifold $X$, if $X$ admits a positive CR line bundle $L$, 
%then for every $\kappa\in\mathbb N$, $L$ is $C^\kappa$-ample, 
%that is, there exists $n_0=n_0(\kappa)>0$ such that for evey $k\geq n_0$, 
%we can find $C^\kappa$ CR sections $s_0,s_1,\ldots,s_{N_k}$ of $L^k$ 
%such that the map $x\mapsto[s_0(x),s_1(x),\ldots,s_{N_k}(x)]$ embeds $X$ into 
%$\Complex\mathbb P^{N_k}$. 
A natural question is whether we can improve the regularity to $\kappa=\infty$. 
Adachi~\cite{Ad:14} showed that the answer is no, in general. 
The analytic difficulty of this problem comes from the fact that 
the Kohn Laplacian is not hypoelliptic on Levi flat manifolds.
Hypoellipticity and subelliptic estimates are used on CR manifolds with non-degenerate 
Levi form in order to find parametrices of the Kohn Laplacian and establish 
the Hodge decomposition, e.\,g.\ \cite{BdM1:74b,CS:01,HsM:14a,Koh64}.
%Moreover, the Szeg\H{o} projection $\Pi_k$ is not a Fourier integral operator 
%(FIO) in the case of a positive line bundle over a Levi flat CR manifold 
%(see Section \ref{SS:exSz}).

In this paper, we establish a semiclassical Hodge decomposition for
the the Kohn Laplacian acting on powers $L^k$ as $k\to\infty$ and 
we show that the composition $\Pi_k\mathcal{A}_k$ of $\Pi_k$
with an appropriate pseudodifferential operator $\mathcal{A}_k$
is a semiclassical Fourier integral operator, admitting an 
asymptotic expansion in $k$ (see Theorem~\ref{t-main}).
From this result, we can understand the large 
$k$ behaviour of the Szeg\H{o} projection and produce many global CR functions. 
As an application, we give a Szeg\H{o} kernel proof of Ohsawa and Sibony's 
Kodaira type embedding theorem on Levi-flat CR manifolds.

We now formulate the main results. 
Let $(X,T^{1,0}X)$ be an orientable compact Levi-flat CR manifold of dimension 
$2n-1$, $n\geqslant2$.
We fix a Hermitian metric $\langle\,\cdot\,|\,\cdot\,\rangle$ on 
$TX\otimes_\Real\Complex=:\Complex TX$ such that 
$T^{1,0}X$ is orthogonal to $T^{0,1}X$. The Hermitian metric 
$\langle\,\cdot\,|\,\cdot\,\rangle$ on $TX\otimes_\Real\Complex$ induces a Hermitian metric 
$\langle\,\cdot\,|\,\cdot\,\rangle$ on the bundle $\Lambda^{0,q}(T^*X)$ of $(0,q)$ forms of $X$. 
We denote by $dv_X$ the volume form on $X$ induced by $\langle\,\cdot\,|\,\cdot\,\rangle$. 
Let $(L,h)$ be a CR complex line bundle over $X$, where
the Hermitian fiber metric on $L$ is denoted by $h$. We will denote by
$R^L$ the curvature of $L$ (see Definition~\ref{d-suIII-I}). 
We say that $L$ is positive if $R^L_x$ is positive definite at every $x\in X$. 
Let 
\begin{equation}\label{eigrl}
\lambda_1(x)\leq\ldots\leq\lambda_{n-1}(x),
\end{equation}
be the eigenvalues of $R^L_x$ with respect to 
$\langle\,\cdot\,|\,\cdot\,\rangle$
and set 
\begin{equation}\label{detrl}
\det R^L_x:=\lambda_1(x)\ldots\lambda_{n-1}(x).
\end{equation}
%---------------
For $k>0$, let $(L^k,h^{k})$ be the $k$-th tensor power of the line bundle $(L,h)$. In this paper, we assume that $k\gg1$. 
For $u,v\in\Lambda^{0,q}_x(T^*X)\otimes L^k_x$ we denote by $\langle\,u\,|\,v\,\rangle_{h^k}$ the induced
pointwise scalar product induced by  $\langle\,\cdot\,|\,\cdot\,\rangle$ and $h^k$.
We then get natural a global $L^2$ inner product $(\,\cdot\,|\,\cdot\,)_{k}$
on $\Omega^{0,q}(X, L^k)$, $(\,\alpha\,|\,\beta\,)_{k}:=\int_X \big\langle\,\alpha\,|\,\beta\,\big\rangle_{h^k}\, dv_X$.
Similarly, we have an $L^2$ inner product $(\,\cdot\,|\,\cdot\,)$ on $\Omega^{0,q}(X)$.
 We denote by $L^2_{(0,q)}(X,L^k)$ and $L^2_{(0,q)}(X)$ the completions of 
 $\Omega^{0,q}(X,L^k)$ and $\Omega^{0,q}(X)$ with respect to 
 $(\,\cdot\,|\,\cdot\,)_{k}$ and $(\,\cdot\,|\,\cdot\,)$, respectively. 
 For $q=0$, we write $L^2(X):=L^2_{(0,0)}(X)$, $L^2(X,L^k):=L^2_{(0,0)}(X,L^k)$. 

Let $\ddbar_{b,k}:\cC^\infty(X,L^k)\To\Omega^{0,1}(X,L^k)$ be the tangential Cauchy-Riemann 
operator cf.\ \eqref{e-suVI}. We extend 
$\ddbar_{b,k}$ to $L^2(X,L^k)$ by $\ddbar_{b,k}:{\rm Dom\,}\ddbar_{b,k}\subset L^2(X,L^k)\To L^2_{(0,1)}(X,L^k),\:\: 
u\longmapsto \ddbar_{b,k}u$, 
with ${\rm Dom\,}\ddbar_{b,k}:=\{u\in L^2(X, L^k);\, \ddbar_{b,k}u\in L^2_{(0,1)}(X, L^k)\}$, 
where $\ddbar_{b,k} u$ is defined in the sense of distributions. The \emph{Szeg\H{o} projection} 
\begin{equation}\label{sz-def}
\Pi_k:L^2(X,L^k)\To{\rm Ker\,}\ddbar_{b,k}
\end{equation}
is the orthogonal projection with respect to $(\,\cdot\,|\,\cdot\,)_{k}$\,. 
%We will first show that (see Section~\ref{s-gue140909})

The Szeg\H{o} projection $\Pi_k$ is not a smoothing operator. 
Nevertheless, our first result shows that it 
enjoys the following regularity property.
%--------------
\begin{thm}\label{t-mainb}
Let $X$ be an orientable compact Levi-flat CR manifold and let $(L,h)$ be a positive CR line bundle
on $X$. Then for every $\ell\in\mathbb N_0$ there exists $N_\ell>0$ 
such that  for every $k\geq N_\ell$\,, 
%$\Pi_k(\cC^\infty(X,L^k))\subset \cC^\ell(X,L^k)$,  
$\Pi_k(\cC^\infty(X,L^k))$ is an infinite dimensional subspace of $\cC^\ell(X,L^k)$
and the induced projection
\(\Pi_k: \cC^\infty(X,L^k)\To \cC^\ell(X,L^k)\)
is continuous.  
\end{thm} 
%--------------
The regularity statement of Theorem \ref{t-mainb}
is related to the regularity of the $\ddbar$-Neumann problem
on weakly pseudoconvex domains endowed with a positive line bundle \cite{Koh73,Ta:83}. In that case one has to take high enough powers to achieve
$\cC^\ell$-regularity, too.

Let us recall now that the Szeg\H{o} kernel $\Pi(x,y)$ of the boundary $X$ of a 
relatively compact strictly pseudoconvex domain $G$ 
is a Fourier integral operator with complex phase,
by a result of Boutet de Monvel-Sj\"ostrand~\cite{BS:76} (here we consider the
projection on the space of CR functions or CR sections of a fixed CR line bundle). 
In particular, $\Pi(x,y)$ is smooth outside the diagonal $x=y$ of $X\times X$
and there is a precise description of the singularity on the diagonal $x=y$,
where $\Pi(x,y)$ has a certain asymptotic expansion.
More precisely, let $G=\{\rho<0\}\Subset G'$ be a strictly pseudoconvex domain in
a $(n+1)$-dimensional complex manifold $G'$,
where $\rho\in\cC^\infty(G')$ is a defining function of $G$.
Then by taking
an almost-analytic extension 
$\varphi=\varphi(x,y):G'\times G'\to\mathbb{C}$ of $\rho$ with certain properties
\cite[(1.1)-(1.3)]{BS:76} we have
%----------------
\begin{equation}\label{e:szego_spsc}
\Pi(x,y)=\int^{\infty}_{0}\!\! e^{i\varphi(x, y)t}s(x, y, t)dt+R(x,y),
\end{equation}
%---------------
where $s(x, y, t)\in S^{n}(X\times X\times\mathbb{R}_+)$ and $R(x,y)$
is a smooth function.

For a Levi-flat CR manifold we do not have such a neat characterization of the singularities
of the Szeg\H{o} kernel $\Pi_k(x,y)$ for fixed $k$.
The smoothing properties of $\Pi_k$ are linked to the singularities of its kernel
$\Pi_k(x,y)$ and to its large $k$ behaviour. Although it is quite difficult to describe them
directly, we will show that $\Pi_k$ still admits an asymptotic expansion in weak sense
(that is, in Sobolev spaces, see Theorem \ref{t-gue150129} and Section 
\ref{SS:exSz} for an explicit example). 

Let $s$ be a local trivializing section of $L$ on an open set $D\subset X$.
We define the weight of the metric with respect to $s$ to be the function
$\phi\in \cC^\infty(D)$ satisfying $\abs{s}^2_{h}=e^{-2\phi}$. 
We have an isometry 
\begin{equation}\label{isom}
U_{k,s}:L^2(D)\to L^2(D,L^k),\:\: u\longmapsto ue^{k\phi}s^k,
\end{equation}
 with
inverse $U_{k,s}^{-1}:L^2(D,L^k)\to L^2(D)$, $\alpha\mapsto e^{-k\phi}s^{-k}\alpha$. 
The localization of $\Pi_k$ with respect to the trivializing section $s$ is given by 
\begin{equation}\label{e-gue141001}
\Pi_{k,s}:L^2_{\mathrm{comp}}(D)\To L^2(D),\:\:
\Pi_{k,s}= U_{k,s}^{-1}\Pi_kU_{k,s},
%u\mapsto e^{-k\phi}s^{-k}\Pi_k(U_{k,s}u),
\end{equation}
where $L^2_{\mathrm{comp}}(D)$ is the subspace of elements of $L^2(D)$ 
with compact support in $D$. 
The second main result of this work shows that for $k\to\infty$, 
$\Pi_k$ is rapidly decreasing outside the diagonal,
and describes the singularities of  $\Pi_k$ semi-clasically
in terms of an oscillatory integral.

\begin{thm}\label{t-gue150129}
Let $X$ be an orientable compact Levi-flat CR manifold of dimension $2n-1$, $n\geq2$. 
Assume that there is a positive CR line bundle $L$ over $X$. 
Then for every $\ell\in\mathbb N_0$, 
there is $N_\ell>0$ such that for every $k\geq N_\ell$ we have:
\\[2pt]
(i)\:\:$\Td\chi\,\Pi_{k}\chi=O(k^{-\infty}):\cC^\infty(X,L^k)\To\cC^\ell(X,L^k)$, 
for all $\chi, \Td\chi\in\cC^\infty(X)$ with 
$\supp\chi\cap\supp\Td\chi=\emptyset$;
\\[2pt]
(ii)\:\:$\Pi_{k,s}-\mathcal{S}_k=O(k^{-\infty}):\cC^\infty_0(D)\To\cC^\ell(D)$, where  
$\mathcal{S}_k:\cC^\infty_0(D)\To \cC^\infty(D)$ is a continuous operator whose kernel satisfies
\begin{equation} \label{e-gue150129}
%S_k(x,y)\equiv\int_{\mathbb R} e^{ik\psi(x,y,u)}s(x,y,u,k)du\mod O(k^{-\infty}),
\mathcal{S}_k(x,y)-\int_{\mathbb{R}} e^{ik\psi(x,y,u)}s(x,y,u,k)du=
O(k^{-\infty}):H^s_{{\rm comp\,}}(D)\To H^s_{{\rm loc\,}}(D),\ \ \forall s\in\mathbb Z,
\end{equation}
where 
\begin{equation}  \label{e-gue150129I}
\begin{split}
&s(x,y,u,k)\sim\sum^\infty_{j=0}s_j(x,y,u)k^{n-j}\text{ in }S^{n}_{{\rm loc\,}}
(1;D\times D\times\Real), \\
%&s(x,y,u,k), s_j(x,y,u)\in \cC^\infty(D\times D\times\Real),\ \ j=0,1,2,\ldots,\\
&s_0(x,x,u)=\frac{1}{2}\pi^{-n}\abs{\det R^L_x},\ \ \ \forall x\in D,\ \ \forall u\in\Real,
\end{split}
\end{equation}
and the phase function 
$\psi\in \cC^\infty(D\times D\times\Real)$ satisfies ${\rm Im\,}\psi(x,y,u)\geq0$ and 
\begin{equation}\label{e-gue140915}
\begin{split}
&d_x\psi|_{(x,x,u)}=-2{\rm Im\,}\ddbar_b\phi(x)+u\omega_0(x),\ \ 
x\in D,\:u\in\Real,\\
&d_y\psi|_{(x,x,u)}=2{\rm Im\,}\ddbar_b\phi(x)-u\omega_0(x),\ \ x\in D,\:u\in\Real,\\
&\frac{\pr\psi}{\pr u}(x,x,u)=0\:\text{and $\psi(x,x,u)=0$},\\
&\mbox{if $x\neq y$ then $\frac{\pr\psi}{\pr u}(x,y,u)\neq0$ or $\psi(x,y,u)\neq0$},
\end{split}
\end{equation} 
and there exists $c>0$ such that
\begin{equation}\label{e-gue150129II}
\abs{d_y\psi(x,y,u)}\geq c\abs{u}, \ \ \forall u\in\Real,\ \ \forall (x,y)\in D\times D.
\end{equation}
Here 
$\omega_0\in \cC^\infty(X,T^*X)$ is 
the positive $1$-form of unit length orthogonal to 
$\Lambda^{1,0}(T^*X)$ and $\Lambda^{0,1}(T^*X)$, see Definition \ref{puf}.
\end{thm}
%-------------
Theorem \ref{t-gue150129} shows that the (localized) Szeg\H{o} projector
is close in the semiclassical limit to an approximate Szeg\H{o} projector 
$\mathcal{S}_k$, which has an asymptotic expansion in Sobolev spaces,
given by the operator $S_k:\cC^\infty_0(D)\To\cC^\infty(D)$
with kernel 
\begin{equation}\label{e-gue160814}
S_k(x,y)=\int_{\mathbb{R}} e^{ik\psi(x,y,u)}s(x,y,u,k)du.
\end{equation}
Note that integrating by parts with respect to $y$ several times in \eqref{e-gue160814} and using 
\eqref{e-gue150129II}, we conclude  that $S_k$ is well-defined as a continuous operator 
$S_k:\cC^\infty_0(D)\To\cC^\infty(D)$. 

For fixed $u\in\Real$, the integrand in the formula \eqref{e-gue160814} of $S_k$
(hence also for $\mathcal{S}_k$ or $\Pi_k$) bears a resemblance to
the Bergman kernel $B_k$  of the $k$-th power of a positive line bundle $L$ on 
a complex manifold (cf.\ \cite{HsM:14,SZ02,Zelditch98}, see \eqref{e-gue160619}). 
Note that $B=\sum_{k\geq0}B_k$
is basically the Szeg\H{o} kernel
of the strictly pseudoconvex CR manifold given by the boundary 
of the unit disc bundle of $L^*$. The kernel of $B$ has the form \eqref{e:szego_spsc}
involving an integral $\int_0^\infty dt$ and the the $B_k$ are its Fourier
coefficients (see \cite{Zelditch98}). 
In our CR Levi-flat at setting, the
$\Pi_k$ most resemble $B_k$ in being semi-classical kernels (with a $k$ in the phase)
but also formally resemble $B$ in being integrals over an additional parameter
$u$.
But the integrals over the additional parameters in \eqref{e:szego_spsc}
and \eqref{e-gue160814} have completely
different origins.  
The integral $\int_{\Real} du$ in \eqref{e-gue160814}
arises due to the transversal direction to the leaves of the Levi foliation. 
This is a different kind of integral than that for $B$, which
arrises through summation over $k\geq0$.
%has a homogeneous phase reflecting the fact that its components $B_k$ 
%have phases $e^{k\psi}$. The formula for $B$ is essentially the sum over these
%$k$, which are positive. 
%To obtain an analogue kernel of $B$, we should have to sum $S_k$ in $k$.

For fixed $k$, $S_k$ is not a FIO since the phase function $\psi(x,y,u)$ is
not homogeneous of degree one with respect to $u$. To obtain a homogeneous FIO, 
we should have to sum $S_k$ in $k$. 
Moreover, the domain of integration in \eqref{e-gue160814}
is $\mathbb{R}$, unlike \eqref{e:szego_spsc}, where it is $\mathbb{R}_+$.
In Section \ref{SS:exSz} we show that the Szeg\H{o} projector $\Pi_k$ itself
is not a FIO, in contrast to the result of Boutet de Monvel-Sj\"ostrand \cite{BS:76}
for strictly pseudoconvex domains.
The proof of Theorem \ref{t-gue150129} is also different from 
\cite{BS:76} and is based on the heat
equation method of Menikoff-Sj\"ostrand~\cite{MS78}.
For the precise form of $\psi(x,y,u)$ see \eqref{e-gue140826Im} 
and \eqref{e-gue160629zII}. This can be compared
to the form \cite[Theorems 3.2, 3.4]{HsM:14a} of the phase function for the Szeg\H{o}
kernel on a non-degenerate CR manifold.

If $M$ is compact complex manifold of dimension $n$ endowed with a positive line bundle
$L\to M$ then the localization of the Bergman kernel $B_k$ corresponding to
$L^k$ has the form $B_{k,s}(z,w)=e^{ik\varphi(z,w)}b(z,w,k)$, where
$b(z,w,k)\sim\sum^\infty_{j=0}k^{n-1-j}b_j(z,w)$ in 
$S^{n-1}_{{\rm loc\,}}(1;D\times D)$, 
by the works of Zelditch~\cite{Zelditch98} 
and Shiffman-Zelditch~\cite{SZ02}, see also \cite{HsM:14}
(cf.\ Section \ref{SS:exSz}). We see thus that $S_k(x,y)$ is
an integrated version of the Bergman kernel on a complex manifold.  
This corresponds to the fact that the Levi-flat CR manifold is foliated by complex manifolds
and we have a transversal direction (where there are no elliptic estimates) 
in which we integrate.
Note that in the case of a strictly pseudoconvex CR manifold we always have
a `bad' direction for ellipticity. In our case of a Levi-flat manifold
endowed with a positive line bundle we have elliptic estimates in the
directions of the Levi-foliation and the `bad' direction is the transversal one. 
As a consequence, as shown by \eqref{e-gue150129}, 
$\mathcal{S}_k(x,y)$ and hence $\Pi_{k,s}(x,y)$, admits an asymptotic expansion 
$S_k(x,y)+O(k^{-\infty})$ only in Sobolev spaces 
(see also Theorem~\ref{t-gue150130} for the details). 
This is an important difference between the Levi-flat and 
the K\"ahler case.
 
The fact that we integrate over $\mathbb{R}$ in \eqref{e-gue150129}
prevents us from obtaining asymptotics in the $\cC^\ell$-topology for
the kernel of $\Pi_{k,s}$. However, by composing
%Using Theorem~\ref{t-gue150129}, we will show that by composing
with certain semiclassical pseudo-differential operators $\mathcal{A}_k$
we obtain asymptotics in the $\cC^\ell$-topology for
the kernels of $(\Pi_{k,s}-\mathcal{S}_k)\mathcal{A}_k$ and eventually
$\Pi_{k,s}\mathcal{A}_k$. The symbol of $\mathcal{A}_k$ is
supported in a large interval $(-M/2,M/2)$ in the fiber direction
and by taking $M$ large enough we recover increasingly more features of $\Pi_k$.
The freedom to choose these operators and the constant $M$ will be crucial for proving 
the embedding Theorem \ref{t-embleviflat} (e.\,g.\ in \eqref{e-gue140916}).

Let $\mathcal{A}_k$ be 
a properly supported semi-classical pseudodifferential operator on $D$ of order $0$ and 
classical symbol
(see Definition \ref{d-gue13628I})
\begin{equation}  \label{ak_s}
\begin{split}
&\alpha(x,\eta,k)\sim\sum^\infty_{j=0}k^{-j}\alpha_j(x,\eta)\:\:\text{in $S^0_{{\rm loc\,}}(1,T^*D)$},\\
\alpha(x,\eta,k)=0,\:&\alpha_j(x,\eta)=0, \:j=0,1,2,\ldots, \text{for $\abs{\eta}\geq\tfrac12 M$, 
for some $M>0$.}
\end{split}
\end{equation}
Note that $\mathcal{A}_k$ is smoothing for each $k$. 
A semi-classical pseudodifferential operator with these properties will be called \emph{good}.
%---------------------
\begin{thm}\label{t-main}
Let $X$ be an orientable compact Levi-flat CR manifold of dimension $2n-1$, $n\geq2$. 
Assume that there is a positive CR line bundle $L$ over $X$. 
Assume that $\mathcal{A}_k$ is a good semi-classical pseudodifferential operator
on $D$.
Then for every $\ell\in\mathbb N_0$, 
there is $N_\ell>0$ such that for every $k\geq N_\ell$\,, 
$(\Pi_{k,s}\mathcal{A}_k)(\cdot\,,\cdot)\in \cC^\ell(D\times D)$ and
\begin{equation} \label{e-gue13630IVabm}
(\Pi_{k,s}\mathcal{A}_k)(x,y)\equiv\int_{\mathbb{R}} e^{ik\psi(x,y,u)}a(x,y,u,k)du\mod O(k^{-\infty})\:\: 
\text{in $\cC^\ell(D\times D)$},
\end{equation}
where 
\begin{equation}  \label{e-gue13630Vabm}
\begin{split}
&a(x,y,u,k)\sim\sum^\infty_{j=0}a_j(x,y,u)k^{n-j}\text{ in }S^{n}_{{\rm loc\,}}
(1;D\times D\times(-M,M)), \\
&a(x,y,u,k), a_j(x,y,u)\in \cC^\infty_0(D\times D\times(-M,M)),\ \ j=0,1,2,\ldots,\\
&a_0(x,x,u)=\frac{1}{2}\pi^{-n}\abs{\det R^L_x}\alpha_0\big(x,u\omega_0(x)-2{\rm Im\,}\ddbar_b\phi(x)\big),
\ \ \ x\in D,\:|u|<M,
\end{split}
\end{equation}
and $\psi\in \cC^\infty(D\times D\times\Real)$ is as in Theorem~\ref{t-gue150129}.
\end{thm}
%---------------------
For more results and references about the singularities of the Szeg\H{o} kernel and embedding
of CR manifolds we refer to \cite{HsM:14a}.

As an application of Theorem~\ref{t-mainb} and Theorem~\ref{t-main}, we 
show that by projecting appropriate sections through $\Pi_k$
we obtain CR sections which separate points and tangent vectors. 
Hence we give a Szeg\H{o} kernel proof 
of the following result due to Ohsawa and Sibony \cite{Oh12,OS00}. 
%---------------------
\begin{thm}\label{t-embleviflat}
Let $X$ be an orientable compact Levi-flat CR manifold of dimension $2n-1$, $n\geq2$. 
Assume that there is a positive CR line bundle $L$ over $X$. Then, for every $\ell\in\mathbb N$
there is a $M_\ell>0$ such that for every $k\geq M_\ell$\,, we can find $N_k$ CR sections 
$s_0, s_1,\ldots,s_{N_k}\in \cC^\ell(X,L^k)$, such that the map 
$X\ni x\mapsto \left[s_0(x),s_1(x),\ldots,s_{N_k}(x)\right]\in\Complex\mathbb P^{N_k}$ is an embedding. 
\end{thm}
%---------------------
Analytic proofs of the Kodaira embedding theorem for K\"ahler and symplectic manifolds,
based on the Bergman/Szeg\H{o} asymptotics,
were given in \cite{BU,MM07,SZ02,Zelditch98} (see \cite{Hs14,HLM16} for the 
Kodaira embedding of CR manifolds). 
Let us briefly describe the idea of the proof of Theorem~\ref{t-embleviflat}.  
Using the fact that $\Pi_{k,s}\mathcal{A}_k$ is a semi-classical FIO
and the freedom to choose $\mathcal{A}_k$,
we show in Lemma~\ref{l-gue140916II} that for $k$ large enogh, 
for every $\ell\in\mathbb N$ the $\cC^\ell$ CR sections of $L^k$
give local coordinates at all points of $X$. Hence we find
a $\cC^\ell$ CR immersion $\Phi_{k}:X\To\Complex\mathbb P^N$.
In contrast to the K\"ahler or symplectic case we do not show that
$\Phi_{k}$ is injective. Rather, we use the fact that $\Phi_{k}$ separates points in
the neighborhood of the diagonal in $X\times X$ and construct
(by using Theorems~\ref{t-gue150129} and~\ref{t-main})
another $\cC^\ell$ CR map $\Psi_{m}:X\To\Complex\mathbb P^{N'}$
given by sections of a high power $L^m$,
which separates points outside a certain distance of the diagonal.
Therefore, the map $(\Phi_{k},\Psi_{m}):X\To\Complex\mathbb P^{N}\times\Complex\mathbb P^{N'}$
is injective and hence a $\cC^\ell$ embedding, which composed with the Segre embedding 
\eqref{e:segre} yields an embedding $X$ to $\Complex\mathbb P^{(N+1)(N'+1)-1}$.

The paper is organized as follows. In Section \ref{s:prelim} we collect some notations, terminology, 
definitions and statements we use throughout. In Section~\ref{s-gue140824}, 
we give an explicit formula for the semi-classical Kohn Laplacian $\Box^{(q)}_{b,k}$ 
in local coordinates and we determine the characteristic manifold for $\Box^{(q)}_{b,k}$. 
In Section~ \ref{s-gue140824I} we exhibit a semi-classical Hodge decomposition for $\Box^{(q)}_{b,k}$. 
In Section~\ref{s-gue140909}, we establish the regularity of the Szeg\H{o} projection 
and we prove Theorem~\ref{t-mainb}. 
In Section~\ref{s-gue140912}, by using the semi-classical Hodge decomposition theorem 
established in Section~\ref{s-gue140824I} and the regularity for the Szeg\H{o} projection, 
we prove Theorem~\ref{t-gue150129} and Theorem~\ref{t-main}. In Section~\ref{s-gue140915}, 
we prove Theorem \ref{t-embleviflat}.

\section{Preliminaries}\label{s:prelim}

In this section we introduce useful notions from semi-classical
analysis and CR geometry. We then
present background and examples of Levi-flat CR manifolds.
Finally, we treat an explicit example of Szeg\H{o} kernel of a positive line bundle.
\subsection{Definitions and notations from semi-classical analysis} \label{s-ssnI}
We use the following notations: $\mathbb N=\set{1,2,\ldots}$, $\mathbb N_0=\mathbb N\cup\set{0}$, $\Real$ 
is the set of real numbers, $\ol\Real_+:=\set{x\in\Real;\, x\geq0}$.
For a multiindex $\alpha=(\alpha_1,\ldots,\alpha_n)\in\mathbb N_0^n$
we set $\abs{\alpha}=\alpha_1+\ldots+\alpha_n$. For $x=(x_1,\ldots,x_n)$ we write
\[
\begin{split}
&x^\alpha=x_1^{\alpha_1}\ldots x^{\alpha_n}_n,\quad
 \pr_{x_j}=\frac{\pr}{\pr x_j}\,,\quad
\pr^\alpha_x=\pr^{\alpha_1}_{x_1}\ldots\pr^{\alpha_n}_{x_n}=
\frac{\pr^{\abs{\alpha}}}{\pr x^\alpha}\,\cdot
%\,,\ \ D_{x_j}=\frac{1}{i}\pr_{x_j}\,,\quad D^\alpha_x=
%D^{\alpha_1}_{x_1}\ldots D^{\alpha_n}_{x_n}
\end{split}
\]
Let $z=(z_1,\ldots,z_n)$, $z_j=x_{2j-1}+ix_{2j}$, $j=1,\ldots,n$, be coordinates of $\Complex^n$.
We write
\[
\begin{split}
&z^\alpha=z_1^{\alpha_1}\ldots z^{\alpha_n}_n\,,\quad\ol z^\alpha=\ol z_1^{\alpha_1}\ldots\ol z^{\alpha_n}_n\,,\quad\pr_{z_j}=\frac{\pr}{\pr z_j}=
\frac{1}{2}\Big(\frac{\pr}{\pr x_{2j-1}}-i\frac{\pr}{\pr x_{2j}}\Big),\\
&\pr_{\ol z_j}=\frac{\pr}{\pr\ol z_j}=\frac{1}{2}\Big(\frac{\pr}{\pr x_{2j-1}}+i\frac{\pr}{\pr x_{2j}}\Big),\ \ \pr^\alpha_z=\pr^{\alpha_1}_{z_1}\ldots\pr^{\alpha_n}_{z_n}=\frac{\pr^{\abs{\alpha}}}{\pr z^\alpha}\,,\\&
\pr^\alpha_{\ol z}=\pr^{\alpha_1}_{\ol z_1}\ldots\pr^{\alpha_n}_{\ol z_n}=
\frac{\pr^{\abs{\alpha}}}{\pr\ol z^\alpha}\,\cdot
\end{split}
\]

Let $M$ be a $\cC^\infty$ orientable paracompact manifold. 
We let $TM$ and $T^*M$ denote the tangent bundle of $M$ and the cotangent 
bundle of $M$ respectively.
The complexified tangent bundle of $M$ and the complexified cotangent bundle of $M$ 
will be denoted by $\Complex TM$ or $TM\otimes_\Real\Complex$ and 
$\Complex T^*M$ or $T^*M\otimes_\Real\Complex$ respectively. 
We denote by $\langle\,\cdot\,,\cdot\,\rangle$ 
the pointwise duality between $TM$ and $T^*M$.
We extend $\langle\,\cdot\,,\cdot\,\rangle$ bilinearly to 
$TM\otimes_\Real\Complex\times T^*M\otimes_\Real\Complex $.

Let $E$ be a $\cC^\infty$ vector bundle over $M$. The fiber of $E$ at $x\in M$ will be denoted by $E_x$.
Let $F$ be another vector bundle over $M$. We write 
$F\boxtimes E^*$ to denote the vector bundle over $M\times M$ 
with fiber over $(x, y)\in M\times M$ 
consisting of the linear maps from $E_x$ to $F_y$.  

Let $Y\subset M$ be an open set and take any $L^2$ inner product on 
$\cC^\infty_0(Y,E)$. By using this $L^2$ inner product, in this paper, 
we will consider  a distribution section of $E$ over $Y$ is a continuous 
linear form on $\cC^\infty_0(Y,E)$. From now on, let $\mathscr D'(Y, E)$ 
denote the space of distribution sections of $E$ over $Y$ and let 
$\mathscr E'(Y, E)$ be the subspace of $\mathscr D'(Y, E)$ whose 
elements have compact support in $Y$.
For $m\in\Real$, we let $H^m(Y, E)$ denote the Sobolev space
of order $m$ of sections of $E$ over $Y$. 
Put $H^m_{\rm loc\,}(Y, E)=\big\{u\in\mathscr D'(Y, E);\, \varphi u\in H^m(Y, E),
      \, \forall\varphi\in \cC^\infty_0(Y)\big\}$, 
      $H^m_{\rm comp\,}(Y, E)=H^m_{\rm loc}(Y, E)\cap\mathscr E'(Y, E)$.

The Schwartz kernel theorem asserts that for any continuous linear operator
\[A:\cC^\infty_0(M, E)\To \mathscr D'(M,F)\]
 there exists a unique distribution 
$A(\cdot,\cdot)\in \mathscr{D}'(M\times M,F\boxtimes E^*)$ such that
$(Au,v)=(A(\cdot,\cdot),v\otimes u)$ for any $u\in\cC^\infty_0(M,E)$, 
$v\in\cC^\infty_0(M,F^*)$ (see \cite[Theorems\,5.2.1, 5.2.6]{Hor1:83}, 
\cite[Thorem\,B.2.7]{MM07}).
The distribution $A(\cdot,\cdot)$ is
called the Schwartz distribution kernel of $A$.
We say that $A$ is properly supported if the canonical projections on the two factors restricted to
$\supp A(\cdot,\cdot)\subset M\times M$ are proper. 
If $A(\cdot,\cdot)\in\cC^\infty(M\times M, F\boxtimes E^*)$, 
we say that $A$ is a \emph{smoothing operator} and we write $A\equiv0$.
Furthermore, $A$ is smoothing if and only if for all $N\geq0$ and $s\in\Real$,
$A: H^s_{\rm comp\,}(M, E)\To H^{s+N}_{\rm loc\,}(M, F)$
is continuous.

Let $W_1$, $W_2$ be open sets in $\Real^N$ and let $E$ and $F$ be complex 
Hermitian vector bundles over $W_1$ and $W_2$. 
Let $s, s'\in\Real$ and $n_0\in\mathbb\Real$.
For a $k$-dependent continuous function $F_k:H^s_{{\rm comp\,}}(W_1,E)\To H^{s'}_{{\rm loc\,}}(W_2,F)$
we write $F_k=O(k^{n_0}):H^s_{{\rm comp\,}}(W_1,E)\To H^{s'}_{{\rm loc\,}}(W_2,F)$, 
if for any $\chi_0\in \cC^\infty(W_2), \chi_1\in \cC^\infty_0(W_1)$, there is a positive constant $c>0$ independent of $k$, such that
$\norm{(\chi_0F_k\chi_1)u}_{s'}\leq ck^{n_0}\norm{u}_{s}$, $\forall u\in H^s_{{\rm loc\,}}(W_1,E)$,
%\begin{equation} \label{e-gue13628II}
%\norm{(\chi_0F_k\chi_1)u}_{s'}\leq ck^{n_0}\norm{u}_{s},\ \ \forall u\in H^s_{{\rm loc\,}}(W_1,E),
%\end{equation}
where $\norm{\cdot}_s$ denotes the usual Sobolev norm of order $s$. We write
$F_k=O(k^{-\infty}):H^s_{{\rm comp\,}}(W_1,E)\To H^{s'}_{{\rm loc\,}}(W_2,F)$,
if $F_k=O(k^{-N}):H^s_{{\rm comp\,}}(W_1,E)\To H^{s'}_{{\rm loc\,}}(W_2,F)$, for every $N>0$. Similarly, let $\ell\in\mathbb N$, for a $k$-dependent continuous function $G_k:\cC^\infty_0(W_1,E)\To\cC^\ell(W_2,F)$
we write $G_k=O(k^{-\infty}):\cC^\infty_0(W_1,E)\To\cC^\ell(W_2,F)$,
if for any $\chi_0\in \cC^\infty(W_2), \chi_1\in \cC^\infty_0(W_1)$ and $N>0$, there are positive constants $c>0$ and $M\in\mathbb N_0$ independent of $k$, such that 
$\norm{(\chi_0G_k\chi_1)u}_{\cC^\ell(W_2,F)}\leq ck^{-N}\norm{u}_{\cC^M(W_1,E)}$, $\forall u\in\cC^\infty_0(W_1,E)$.
%\begin{equation} \label{e-gue13628IIbis}
%\norm{(\chi_0G_k\chi_1)u}_{\cC^\ell(W_2,F)}\leq ck^{-N}\norm{u}_{\cC^M(W_1,E)},\ \ \forall u\in\cC^\infty_0(W_1,E),
%\end{equation}

A $k$-dependent continuous operator
$A_k:\cC^\infty_0(W_1,E)\To\mathscr D'(W_2,F)$ is called $k$-negligible on $W_2\times W_1$
if for $k$ large enough $A_k$ is smoothing  and for any $K\Subset W_2\times W_1$, any 
multi-indices $\alpha$, $\beta$ and any $N\in\mathbb N$ there exists $C_{K,\alpha,\beta,N}>0$
such that 
\begin{equation}\label{e-gue13628III}
\abs{\pr^\alpha_x\pr^\beta_yA_k(x, y)}\leq C_{K,\alpha,\beta,N}k^{-N}\,,\:\: \text{on $K$}.
\end{equation}
%We write in this case $A_k\equiv0\mod O(k^{-\infty})$ or
%$A_k(x,y)\equiv0\mod O(k^{-\infty})$.

Let $C_k:\cC^\infty_0(W_1,E)\To\mathscr D'(W_2,F)$
be another $k$-dependent continuous operator. We write $A_k\equiv C_k\mod O(k^{-\infty})$ (on $W_2\times W_1$) or 
$A_k(x,y)\equiv C_k(x,y)\mod O(k^{-\infty})$ (on $W_2\times W_1$)
if $A_k-C_k$ is $k$-negligible on $W_2\times W_1$. 

Similarly, for $\ell\in\mathbb N_0$, $A_k:\cC^\infty_0(W_1,E)\To\mathscr D'(W_2,F)$ 
is called $k$-negligible in the $\cC^\ell$ norm on $W_2\times W_1$ if 
$A_k(x, y)\in \cC^\ell(W_2\times W_1,E_y\boxtimes F_x)$ for $k$ large and 
\eqref{e-gue13628III} holds for multi-indices $\alpha$,
$\beta$ with $\abs{\alpha}+\abs{\beta}\leq\ell$.
%$A_k(x,y)$ satisfies
%$\abs{\pr^\alpha_x\pr^\beta_yA_k(x, y)}=O(k^{-N})$ for $k$ large, locally uniformly
%on every compact set in $W_2\times W_1$, for all $N\in\mathbb N$ and all multi-indices $\alpha$,
%$\beta$ with $\abs{\alpha}+\abs{\beta}\leq\ell$. 

Let $C_k:\cC^\infty_0(W_1,E)\To\mathscr D'(W_2,F)$
be another $k$-dependent continuous operator. We write $A_k\equiv C_k\mod O(k^{-\infty})$ 
in the $\cC^\ell$ norm (on $W_2\times W_1$) or 
$A_k(x,y)\equiv C_k(x,y)\mod O(k^{-\infty})$ in $\cC^\ell$ norm (on $W_2\times W_1$)
if $A_k-C_k$ is $k$-negligible in $\cC^\ell$ norm on $W_2\times W_1$.  
%We say that $A_k:\cC^\infty_0(W_1,E)\To\mathscr D'(W_2,F)$ is $k$-negligible away the diagonal in $\cC^\ell$ norm on $W_2\times W_1$ if for every $\Td\chi\in \cC^\infty_0(W_2)$, $\widehat\chi\in \cC^\infty_0(W_1)$ with 
%$\supp\Td\chi\cap \supp\widehat\chi=\emptyset$, 
%we have $\Td\chi A_k\widehat\chi\equiv0\mod O(k^{-\infty})$ in $\cC^\ell$ norm. 

Let $B_k:L^2(X,L^k)\To L^2(X,L^k)$ be a continuous operator. Let $s$, $s_1$ be local trivializing sections of $L$ on open sets $D_0\Subset M$, $D_1\Subset M$ respectively, $\abs{s}^2_{h}=e^{-2\phi}$, $\abs{s_1}^2_{h}=e^{-2\phi_1}$. The localized operator (with respect to the trivializing sections $s$ and $s_1$) of $B_k$ is given by 
\begin{equation} \label{e-gue140824II}
\begin{split}
B_{k,s,s_1}:L^2(D_1)\cap \mathscr E'(D_1)\To L^2(D),\:\:
u\longmapsto e^{-k\phi}s^{-k}B_k(s^k_1e^{k\phi_1}u)=U^{-1}_{k,s}B_kU_{k,s_1},
\end{split}
\end{equation} 
and let $B_{k,s,s_1}(x,y)\in\mathscr D'(D\times D_1)$ be the distribution kernel of $B_{k,s,s_1}$.  We write
$B_k=O(k^{n_0}):H^s(X,L^k)\To H^{s'}(X,L^k)$, $n_0\in\mathbb\Real$, 
if for all local trivializing sections $s, s_1$ on $D$ and $D_1$ respectively, we have 
$B_{k,s,s_1}=O(k^{n_0}):H^s_{{\rm comp\,}}(D_1)\To H^{s'}_{{\rm loc\,}}(D)$. 
We write $B_k=O(k^{-\infty}):H^s(X,L^k)\To H^{s'}(X,L^k)$, $n_0\in\mathbb\Real$, 
if for all local trivializing sections $s, s_1$ on $D$ and $D_1$ respectively, we have 
$B_{k,s,s_1}=O(k^{-\infty}):H^s_{{\rm comp\,}}(D_1)\To H^{s'}_{{\rm loc\,}}(D)$. 
Fix $\ell\in\mathbb N$. We write $B_k=O(k^{-\infty}):\cC^\infty(X,L^k)\To\cC^\ell(X,L^k)$, 
if for all local trivializing sections $s, s_1$ on $D$ and $D_1$ respectively, we have 
$B_{k,s,s_1}=O(k^{-\infty}):\cC^\infty_0(D_1)\To\cC^\ell(D)$. 
%We say that $B_k$ is $k$-negligible away the diagonal in $\cC^\ell$ norm on $X$ if for all local trivializing sections $s, %s_1$ on $D$ and $D_1$ respectively, the localized operator $B_{k,s,s_1}$ is $k$-negligible away the diagonal in $\cC^\ell$ %norm on $D\times D_1$.
We recall semi-classical symbol spaces (see Dimassi-Sj\"ostrand~\cite[Chapter 8]{DiSj99}):
\begin{defn} \label{d-gue140826}
Let $W$ be an open set in $\Real^N$. Let 
%$S(1;W)=S(1)$ be the set of
%$a\in \cC^\infty(W)$ such that for every $\alpha\in\mathbb N^N_0$, there
%exists $C_\alpha>0$, such that $\abs{\pr^\alpha_xa(x)}\leq
%C_\alpha$ on $W$. 
\begin{gather*}
S(1;W):=\Big\{a\in \cC^\infty(W)\,|\, \forall\alpha\in\mathbb N^N_0: 
\sup_{x\in W}\abs{\pr^\alpha a(x)}<\infty\Big\},\\
S^0_{{\rm loc\,}}(1;W):=\Big\{(a(\cdot,k))_{k\in\N}\,|\, \forall\alpha\in\mathbb N^N_0,
\forall \chi\in\cC^\infty_0(W)\,:\:\sup_{k\in\N}\sup_{x\in W}\abs{\pr^\alpha a(x,k)}<\infty\Big\}\,.
\end{gather*}
%The space $S_{{\rm loc\,}}(1;W)$ is the set of sequences $a=(a(\cdot,k))$ in $\cC^\infty(W)$
%with the property that for any $\chi\in\cC^\infty_0(W)$ we have
%%\[
%$\sup_{k\in\N}\sup_{x\in W}\abs{\pr^\alpha a(x,k)}<\infty$\,.
%%\]
%If $a=a(x,k)$ depends on $k\in]1,\infty[$, we say that
%$a(x,k)\in S_{{\rm loc\,}}(1;W)=S_{{\rm loc\,}}(1)$ if $\chi(x)a(x,k)$ is uniformly bounded
%in $S(1)$ when $k$ varies in $]1,\infty[$, for any $\chi\in\cC^\infty_0(W)$. 
For $m\in\Real$ let
$S^m_{{\rm loc}}(1;W)=\Big\{(a(\cdot,k))_{k\in\N}\,|\,(k^{-m}a(\cdot,k))
\in S^0_{{\rm loc\,}}(1;W)\Big\}$. 
%For $m\in\Real$, we put $S^m_{{\rm loc}}(1;W)=\{(a(\cdot,k)):(k^{-m}a(\cdot,k))\in S_{{\rm loc\,}}(1;W)\}$.
%For $m\in\Real$, we put $S^m_{{\rm loc}}(1;W)=S^m_{{\rm loc}}(1)=k^mS_{{\rm loc\,}}(1)$. 
So $a(\cdot,k))\in S^m_{{\rm loc}}(1;W)$ if for every $\alpha\in\mathbb N^N_0$ 
and $\chi\in\cC^\infty_0(W)$, there
exists $C_\alpha>0$, such that $\abs{\pr^\alpha (\chi a(\cdot,k))}\leq C_\alpha k^{m}$ on $W$. 
 
Consider a sequence $a_j\in S^{m_j}_{{\rm loc\,}}(1;W)$, $j\in\N_0$, where $m_j\searrow-\infty$, 
and let $a\in S^{m_0}_{{\rm loc\,}}(1;W)$. We say that 
$a(\cdot,k)\sim
\sum\limits^\infty_{j=0}a_j(\cdot,k),\:\:\text{in $S^{m_0}_{{\rm loc\,}}(1;W)$}$, 
if for every
$\ell\in\N_0$ we have $a-\sum^{\ell}_{j=0}a_j\in S^{m_{\ell+1}}_{{\rm loc\,}}(1;W)$ .  
For a given sequence $a_j$ as above, we can always find such an asymptotic sum 
$a$, which is unique up to an element in 
$S^{-\infty}_{{\rm loc\,}}(1;W)=S^{-\infty}_{{\rm loc\,}}(1;W):=\cap _mS^m_{{\rm loc\,}}(1;W)$. 

We say that $a(\cdot,k)\in S^{m}_{{\rm loc\,}}(1;W)$ is a classical symbol on $W$ of order $m$ if 
\begin{equation} \label{e-gue13628I} 
a(\cdot,k)\sim\sum\limits^\infty_{j=0}k^{m-j}a_j\: \text{in $S^{m_0}_{{\rm loc\,}}(1;W)$},\ \ a_j(x)\in 
S_{{\rm loc\,}}(1),\ j=0,1\ldots.
\end{equation} 
The set of all classical symbols on $W$ of order $m_0$ is denoted by 
$S^{m_0}_{{\rm loc\,},{\rm cl\,}}(1;W)=S^{m_0}_{{\rm loc\,},{\rm cl\,}}(1;W)$. 
\end{defn}
%--------------
\begin{defn} \label{d-gue13628I}
Let $W$ be an open set in $\Real^N$. A semi-classical pseudodifferential 
operator on $W$ of order $m$ and classical symbol is a $k$-dependent continuous operator 
$A_k:\cC^\infty_0(W)\To \cC^\infty(W)$ such that the distribution kernel $A_k(x,y)$ 
is given by the oscillatory integral
\begin{equation}\label{psk-def}
\begin{split}
A_k(x,y)&\equiv\frac{k^N}{(2\pi)^N}\int e^{ik\langle x-y,\eta\rangle}a(x,y,\eta,k)d\eta
\mod O(k^{-\infty}),\\
&a(x,y,\eta,k)\in S^m_{{\rm loc\,},{\rm cl\,}}(1;W\times W\times\Real^N).
\end{split}
\end{equation}
We shall identify $A_k$ with $A_k(x,y)$. It is clear that $A_k$ 
has a unique continuous extension $A_k:\mathscr E'(W)\To\mathscr D'(W)$. 
Moreover, it is well-known \cite{GrSj94} that there is a symbol
\begin{equation}\label{ps-symb}
\alpha(x,\eta,k)\in S^m_{{\rm loc\,},{\rm cl\,}}(1;W\times\Real^N)=
S^m_{{\rm loc\,},{\rm cl\,}}(1;T^*W)
\end{equation}
independ on $y$ 
such that 
\begin{equation}\label{ps-def}
A_k(x,y)\equiv\frac{k^N}{(2\pi)^N}\int e^{ik\langle x-y,\eta\rangle}\alpha(x,\eta,k)d\eta
\mod O(k^{-\infty}).
\end{equation}
\end{defn} 

\subsection{CR manifolds and bundles} \label{s-su}
A Cauchy-Riemann (CR) manifold (of hypersurface type)
is a pair $(X,T^{1,0}X)$ where $X$ is a smooth manifold of dimension $2n-1$, $n\geqslant2$, 
and $T^{1,0}X$ is a sub-bundle of the complexified tangent bundle 
$\mathbb{C}TX:=\mathbb{C}\otimes TX$, of rank $(n-1)$, such that $T^{1,0}X\cap  \ol{T^{1,0}X}=\set{0}$ and the set of 
smooth sections of $T^{1,0}X$ is closed under the Lie
bracket. We call $T^{1,0}X$ the
CR structure of $X$ and we denote $T^{0,1}X:=\ol{T^{1,0}X}$.

We say that $(X,T^{1,0}X)$ is a \emph{Levi-flat CR manifold} if the set of smooth sections of 
$T^{1,0}X\oplus T^{0,1}X$ is closed under the Lie bracket.
If $X$ is Levi-flat, there exists a smooth foliation of $X$, of real codimension one and whose leaves are 
complex manifolds: it is obtained by integrating the distribution $(T^{1,0}X\oplus T^{0,1}X)\cap TX$.

In this paper, we assume throughout that $X$ is an orientable Levi-flat manifold.

Fix a smooth Hermitian metric $\langle\,\cdot\,|\,\cdot\,\rangle$ on $TX\otimes_\Real\Complex$ so that $T^{1,0}X$
is orthogonal to $T^{0,1}X$ and $\langle\,u\,|\,v\,\rangle$ is real if $u$, $v$ are real tangent vectors. 
Then locally there is a real non-vanishing vector field $T$ of length one which is pointwise orthogonal to
$T^{1, 0}X\oplus T^{0, 1}X$. $T$ is unique up to the choice of sign. 
For $u\in TX\otimes_\Real\Complex$, we write $\abs{u}^2:=\langle\,u\,|\,u\,\rangle$. 
Denote by $\Lambda^{1,0}(T^*X)$ and $\Lambda^{0,1}(T^*X)$ the dual bundles of $T^{1,0}X$ and $T^{0,1}X$, respectively.
They can be identified with subbundles of the complexified cotangent bundle $T^*X\otimes_\Real\Complex$.

Define the vector bundle of $(0, q)$-forms by $\Lambda^{0,q}(T^*X):=\Lambda^{q}(\Lambda^{0,1}(T^*X))$. 
The Hermitian metric $\langle\,\cdot\,|\,\cdot\,\rangle$ on $TX\otimes_\Real\Complex$ induces,
by duality, a Hermitian metric on $TX\otimes_\Real\Complex$ and also on the bundles of $(0,q)$ forms 
$\Lambda^{0,q}(T^*X)$, $q=0,1,\ldots,n-1$. We shall also denote all these induced metrics by 
$\langle\,\cdot\,|\,\cdot\,\rangle$. Let $\Omega^{0,q}(D)$ denote the space of smooth sections of 
$\Lambda^{0,q}(T^*X)$ over $D$ and let $\Omega^{0,q}_0(D)$ be the subspace of
$\Omega^{0,q}(D)$ whose elements have compact support in $D$. Similarly, 
if $E$ is a vector bundle over $D$, then we let $\Omega^{0,q}(D, E)$
denote the space of smooth sections of $\Lambda^{0,q}(T^*X)\otimes E$ over $D$ and let 
$\Omega^{0,q}_0(D, E)$ be the subspace of $\Omega^{0,q}(D, E)$ whose elements have compact support in $D$. 

Locally we can choose an orthonormal frame $\omega_1,\ldots,\omega_{n-1}$
of the bundle $\Lambda^{1,0}(T^*X)$. Then $\ol\omega_1,\ldots,\ol\omega_{n-1}$
is an orthonormal frame of the bundle $\Lambda^{0,1}(T^*X)$. The real $(2n-2)$-form
$\omega=i^{n-1}\omega_1\wedge\ol\omega_1\wedge\ldots\wedge\omega_{n-1}\wedge\ol\omega_{n-1}$
is independent of the choice of the orthonormal frame. Thus $\omega$ is globally
defined. Locally there is a real $1$-form $\omega_0$ of length one which is orthogonal to
$\Lambda^{1,0}(T^*X)\oplus\Lambda^{0,1}(T^*X)$. The form $\omega_0$ is unique up to the choice of sign.
Since $X$ is orientable, there is a nowhere vanishing $(2n-1)$ form $Q$ on $X$.
Thus, $\omega_0$ can be specified uniquely by requiring that $\omega\wedge\omega_0=fQ$,
where $f$ is a positive function. Therefore $\omega_0$, so chosen, is globally defined.
\begin{defn}\label{puf}
We call $\omega_0$ the positive $1$-form of unit length orthogonal to 
$\Lambda^{1,0}(T^*X)$ and $\oplus\Lambda^{0,1}(T^*X)$.
%the uniquely determined global real $1$-form. 
\end{defn}
We choose a vector field $T$ so that
\begin{equation}\label{e-suI}
\abs{T}=1\,,\quad \langle\,T\,,\,\omega_0\,\rangle=-1\,.
\end{equation}
Therefore $T$ is uniquely determined. We call $T$ the uniquely determined global real vector field. We have the
pointwise orthogonal decompositions:
\begin{equation} \label{e-suII}
T^*X\otimes_\Real\Complex=\Lambda^{1,0}(T^*X)\oplus\Lambda^{0,1}(T^*X)\oplus\Complex\omega_0,\ \ 
TX\otimes_\Real\Complex=T^{1,0}X\oplus T^{0,1}X\oplus\Complex T.
\end{equation} 
%--------------
Let $\ddbar_b:\Omega^{0,q}(X)\To\Omega^{0,q+1}(X)$
%\begin{equation} \label{e-suIV}
%\ddbar_b:\Omega^{0,q}(X)\To\Omega^{0,q+1}(X)
%\end{equation}
be the tangential Cauchy-Riemann operator. Let $U\subset X$ be an open set. We say that a function $u\in \cC^\infty(U)$ is Cauchy-Riemann (CR for short) (on $U$) if $\ddbar_b u=0$.
%--------------
\begin{defn} \label{d-suIII}
Let $L$ be a complex line bundle over a CR manifold $X$. 
We say that $L$ is a Cauchy-Riemann (CR for short) (complex) line bundle over $X$
if its transition functions are CR.
\end{defn}

\begin{defn}
The \emph{Szeg\H{o} kernel} of the pair $(X,L^k)$ is the the Schwartz distribution kernel 
$\Pi_k(\cdot,\cdot)\in\mathscr D'(X\times X,L^k\boxtimes (L^k)^*)$ of the
Szeg\H{o} projection $\Pi_k$ given by \eqref{sz-def}.
\end{defn}

%--------------
If $X$ is Levi-flat, then the restriction a CR line bundle to any leaf $Y$ of the Levi-foliation 
is a holomorphic line bundle.

From now on, we let $(L,h)$ be a CR line bundle over $X$, where
the Hermitian fiber metric on $L$ is denoted by $h$. We will denote by
$\phi$ the local weights of the Hermitian metric. More precisely, if
$s$ is a local trivializing
section of $L$ on an open subset $D\subset X$, then the local weight of $h$ with respect to $s$ is the function 
$\phi\in \cC^\infty(D,\Real)$ for which
\begin{equation} \label{e-suV}
\abs{s(x)}^2_{h}=e^{-2\phi(x)}\,,\quad x\in D.
\end{equation}
%-----------
\begin{defn} \label{d-suIII-I}
Let $s$ be a local trivializing section of $L$ on an open subset $D\subset X$ and $\phi$ 
the corresponding local weight as in \eqref{e-suV}.
For $p\in D$, we define the Hermitian quadratic form $M^\phi_p$ on $T^{1,0}_pX$ by
\begin{equation} \label{e-suXIII-II}
M^\phi_p(U, V)=\Big\langle U\wedge\ol V, d\big(\ddbar_b\phi-\pr_b\phi\big)(p)\Big\rangle,\ \ U, V\in T^{1,0}_pX,
\end{equation}
where $d$ is the usual exterior derivative and $\ol{\pr_b\phi}=\ddbar_b\ol\phi$. 
Since $X$ is Levi-flat, the definition of $M^\phi_p$ does not depend on the choice of local trivializations 
(see~\cite[Proposition 4.2]{HsM:12}).
Hence there exists a smooth section $R^L$ of the bundle of Hermitian forms on $T^{1,0}X$
such that $R^L|_D=M^\phi$. We call $R^L$ the curvature of $(L,h)$. 
We say that $(L,h)$, or $R^L$, is positive if $R^L_x$ is positive definite, for every $x\in X$.
We say that $L$ is a positive CR line bundle over $X$ 
if there is a Hermitian fiber metric $h$ on $L$ 
such that the induced curvature $R^L$ is positive.
\end{defn}

In this paper, we assume that $L$ is a positive CR line bundle over a Levi-flat CR manifold
$X$ and we fix a Hermitian fiber 
metric $h$ of $L$ such that the induced curvature $R^L$ is positive. 
Note that a positive line bundle $(L,h)$ in the sense of Definition \ref{d-suIII-I} is positive along
the leaves of the Levi-foliation: its restriction $(L,h)|_Y$ to any leaf $Y$ is positive
(that is, the curvature of the associated Chern connection is positive).

Let $L^k$, $k>0$, be the $k$-th tensor power of the line bundle $L$. 
The Hermitian fiber metric on $L$ induces a Hermitian fiber metric on $L^k$ that we shall denote 
by $h^{k}$. If $s$ is a local trivializing section
of $L$ then $s^k$ is a local trivializing section of $L^k$. 
We write $\ddbar_{b,k}$ to denote the tangential
Cauchy-Riemann operator acting on forms with values in $L^k$, defined locally by
\begin{equation} \label{e-suVI}
\ddbar_{b,k}:\Omega^{0,q}(X, L^k)\To\Omega^{0,q+1}(X, L^k)\,,\quad \ddbar_{b,k}(s^ku):=s^k\ddbar_bu,
\end{equation}
where $s$ is a local trivialization of $L$ on an open subset $D\subset X$ and $u\in\Omega^{0,q}(D)$. 
%--------------
\subsection{Background on Levi-flat CR manifolds and examples} 
Originally, Levi-flat CR manifolds first arose as Levi-flat real hypersurfaces
in the study of the Levi problem, which asks the characterization of a domain
of holomorphy by Levi pseudoconvexity of its boundary.
While the Levi problem has an affirmative answer for domains in $\mathbb{C}^n$
(by the works of Oka, Bremmerman, Norguet)
or $\mathbb{CP}^n$ (by results of Fujita and Takeuchi), 
Grauert \cite{Gr:58} pointed out that some domains with
Levi-flat boundary give counterexamples to the Levi
problem (see also \cite{Gra:63,Na:63}). These domains do not 
possess any non-constant
holomorphic functions but they are typically endowed
with a positive and ample line bundle, so the relevant function theory
here deals with sections of positive line bundles and meromorphic
functions, see e.\,g.\ \cite{Gra:63}. From an analytic point
of view this leads to the study of $\ddbar$-Neumann problem
in this situation \cite{Koh73,Ta:83}. 

On the other hand, if we look upon Levi-flat CR manifolds
intrinsically, the function theory should deal with CR functions or
sections, that is,  functions or sections 
which are holomorphic along the leaves
of the Levi foliation. By a theorem of Inaba \cite[Theorem 1]{In92}, 
every continuous CR function on a compact Levi-flat
CR manifold is constant along leaves of the Levi foliation. If the
foliation has dense leaves, it follows that continuous CR functions
are constant. Hence, as in the case of compact complex manifolds,
we are led to perform function theory with sections of positive 
line bundles.
The study of CR meromorphic functions on compact
Levi-flat CR manifolds can also be seen as an alternative generalization 
of function theory on compact compact complex manifolds 
(the leaves of the foliation).

We present here a list of interesting Levi-flat
manifolds carrying a positive line bundle.

\smallskip

\noindent
\textbf{(i) Linear hypersurfaces in tori.} Let $n\geq2$ and let $\Gamma$ be the lattice in 
$\mathbb{C}^n$ generated by $\mathbb{R}$-linearly independent vectors
$w_j=(w_{j1},\ldots,w_{jn})$, $j=1,\ldots,2n$, where $w_1 = (1,0,\ldots,0)$
and $\operatorname{Re}w_{j1}=0$ for $j=2,\ldots,2n$.
Let $T^n$ be the torus $\mathbb{C}^n/\Gamma$ and let $\pi:\mathbb{C}^n\to T^n$
be the natural map.
For $c\in\mathbb{R}$ set $X_c=\pi\big(\{z\in\mathbb{C}^n:\operatorname{Re}z_1=c\}\big)$.
Then $X_c$ is a compact Levi-flat hypersurface in $T^n$. If $T^n$ is projective, $X_c$
carries a positive CR line bundle obtained by restriction of a positive holomorphic bundle on $T^n$. 

This construction was used by Grauert in order to give an example of a pseudoconvex domain that is not
holomorphically convex, see \cite{Gr:58}, \cite[p.\,387]{Na:63}.
Namely, let $U\subset\mathbb{C}^n$ be defined by 
$0 < {\rm Re} z_1 < 1$
and let $D=\pi(U)$. Then every holomorphic function
on $D$ is constant.

\smallskip

\noindent
\textbf{(ii) Grauert tubes in topologically trivial holomorphic line bundles.} 
Let $M$ be a compact projective manifold and $\pi:F\to M$ a topologically trivial
holomorphic line bundle. There exists a finite open covering 
$(U_\alpha)$ of $M$ and holomorphic
frames $e_\alpha$ over $U_\alpha$ with $e_\beta=g_{\alpha\beta}e_\alpha$
on $U_\alpha\cap U_\beta$ for holomorphic transition functions  
$g_{\alpha\beta}:U_\alpha\cap U_\beta\to\mathbb{C}^*$ such that 
$|g_{\alpha\beta}|\equiv1$. We define a Hermitian metric $h$ on $F$
by setting $|e_\alpha|_h=1$. Then $X_c=\{v\in F:|v|_h=c\}$, $c>0$,
is a real analytic Levi-flat hypersurface in $F$, cf.\ \cite[Satz\,2]{Gra:63}.
If $L\to M$ is a positive line bundle, then $\pi^*L|_{X_c}$ is
a positive CR line bundle.
The Levi foliation of $X_c$ has dense leaves if
and only if all tensor powers $F^k$ for $k\neq0$
are holomorphically non-trivial.

Again, this construction is related to the Levi problem for 
pseudoconvex domains. Grauert \cite{Gra:63} showed that
$D_c=\{v\in F:|v|_h<c\}$, for $c>0$, are meromorphically convex but not
holomorphically convex domains. 

\smallskip

\noindent
\textbf{(iii) Circle bundles over projective manifolds.} %Boundaries of holomorphic disc bundles.}
Let $M$ be a projective compact manifold. Let $\pi:\mathcal{D}\to M$ 
be a holomorphic
fiber bundle  over $M$ with fiber the unit disc 
$\mathbb{D}\subset\mathbb{C}$.
%is called a holomorphic disc bundle over $M$. 
It can be easily seen that holomorphic trivializations form a trivializing cover, that is, 
the transition functions are locally constant.
The disc bundle is thus isomorphic to a bundle of the form
$\mathcal{D}_\rho:=M\times_{\rho}\mathbb{D}:=\widetilde{M}\times\mathbb{D}\big{/}\!\!\sim$\,,
where $\rho:\pi_1(M)\to\operatorname{Aut}(\mathbb{D})$ is a group homomorphism,
$\widetilde{M}$ is the universal cover of $M$ and the relation equivalence $\sim$
is given by $(x,\zeta)\sim(\gamma x, \rho(\gamma)\zeta)$, for $x\in \widetilde{M}$,
$\zeta\in\mathbb{D}$ and $\gamma\in\pi_1(M)$.
Since $\operatorname{Aut}(\mathbb{D})$ is a group of biholomorphisms of 
$\mathbb{D}$ consisting of M\"obius transformations
preserving $\mathbb{D}$, acting on $\CP^1$ and fixing the unit
circle $\mathbb{S}^1=\partial\mathbb{D}$, it follows that a holomorphic disc bundle is canonically
embedded in the complex manifold  $N_\rho := M \times_{\rho}\mathbb{CP}^1\to M$,
and the boundary of $\mathcal{D}_\rho$ in $N_\rho$ is a compact Levi-flat CR manifold
$X_\rho=M \times_{\rho}\partial\mathbb{D}$.
Note that $N_\rho$ is a projective manifold by \cite[Theorem\,8]{Kod:54}, so any projective
embedding of $N_\rho$ induces a positive CR line bundle on $X_\rho$.

Other positive CR line bundles over $X_\rho$ are given by the pullback 
$\pi^*L|_{X_\rho}$ of any positive line bundle $L\to M$.
It was shown in \cite[Main Theorem]{Ad:14}
that if $M$ is a compact Riemann surface,
$\pi^*L|_{X_\rho}$ is not $C^\infty$ ample if $\mathcal{D}_\rho$ has a unique 
non-holomorphic harmonic section
$h$ with $\operatorname{rank}_\mathbb{R} dh = 2$ on an open dense set.
A concrete example when the latter situation occurs is %is given by $X_\rho$ 
obtained  by taking 
$M$ to be a hyperbolic compact Riemann surface, regarding 
$\pi_1(M)\subset\operatorname{Aut}(\mathbb{D})$
as a Fuchsian representation and taking a non-trivial quasiconformal deformation
$\rho:\pi_1(M)\to\operatorname{Aut}(\mathbb{D})$ of
$\Gamma$, see \cite{Ad:14}.

The present construction was used in \cite[Section 2]{DeDu:16} in order 
to construct Levi-flat hypersurfaces with nontrivial Euler class in complex surfaces of general type. 

A generalization, particularly relevant in the context of
the Ohsawa-Sibony embedding theorem, is the following.
Let $\rho:\pi_1(M)\to\operatorname{Diff}(\mathbb{S}^1)$ 
be a group homomorphism, whose image is not necessarily contained in
the M\"obius transformation group.
Then $X_\rho=M \times_{\rho}\mathbb{S}^1$ is Levi-flat 
and if $\pi:X_\rho\to M$ is the canonical projection
and if $L\to M$ is positive, then $\pi^*L$ is a positive CR
line bundle on $X_\rho$. Theorem \ref{t-embleviflat}
gives a realization of these $X_\rho$ as
$\cC^\ell$ CR submanifolds in complex projective space 
for arbitrary large $\ell$, while it is not clear a priori whether we can construct its 
filling $\mathcal{D}_\rho$ and its ambient $N_\rho$. 
Actually, for some special choice of $M$ and $\rho$, it can be shown that $X_\rho$
cannot be realized as a $\cC^\infty$ Levi-flat real hypersurface, see 
\cite{Barret:90,In92}. For example, there does not exist a $\cC^\infty$ Levi-flat hypersurface 
$X$ in a two-dimensional complex manifold such that the Levi foliation 
of $X$ is homeomorphic to Reeb's foliation of $\mathbb{S}^3$.
An open question is whether such Levi-flat manifolds $X_\rho$ can be realized as 
$\cC^\ell$ Levi-flat
real hypersurfaces for some finite $\ell\in\N$.

\smallskip

\noindent
\textbf{(iv) Levi-flat boundaries of Stein domains.} In the examples (i) and (ii),
Grauert constructed Levi-flat hypersurfaces bounding pseudoconvex non-Stein
domains. Nemirovski \cite{Nm:99} constructed examples of compact complex surfaces which contain
a smooth Levi-flat hypersurface splitting the surface in two Stein domains.
This construction admits a generalization to  complex manifolds of arbitrary 
dimension as noted in \cite{Nm:99}, \cite[p.\,168]{Oh:07}.

%Let $E\to\Sigma$ be a holomorphic line bundle
%over a compact Riemann surface. We denote the zero section in 
%$E$ by $0\subset E$. Consider a meromorphic section $s$ of $L$ over $\Sigma$
%having at most simple zeros and poles, and
%$s^{-1}(0)\cup s^{-1}(\infty)\neq\emptyset$.
%Consider the set
%$\mathbb{R}s:= \{ \lambda s ( z )\in E :
%z \in \Sigma\setminus(s^{-1}(0)\cup s^{-1}(\infty)),\:\:\lambda\in\mathbb{R}\}$
%and let $Y$ be the closure of $\mathbb{R}s$
%in $E\setminus0$.
%The real hypersurface $Y$ contains the fibers of the bundle $E\setminus0$ 
%over $s^{-1}(0)\cup s^{-1}(\infty)$ and it meets all other fibers along real lines.
%
%Fix now a real number $r>1$ and consider on $E\setminus0$ 
%the equivalence relation $v\sim rv$. 
%The quotient space $S=\pi(E\setminus0)$ is an elliptic surface 
%(a locally trivial fibration with base $\Sigma$ and fiber a torus) and 
%$X=\pi(Y)$ is a smooth Levi-fiat hypersurface
%in $S$, dividing $S$ in two Stein domains.

Consider a holomorphic $\mathbb{C}^*$-bundle 
$B\to S$ where $S$ is a projective manifold and
the action of $\Z$ generated by 
$(w, z)\to(w, 2z)$ in terms of the local
coordinate $w$ of $S$ and the fiber coordinate $z$. 
Then, for any meromorphic section $s$ of the 
associated $\widehat{\Complex}$-bundle associated to $B$ such that its zeros
and poles are mutually disjoint and of order one, a Levi flat hypersurface $X$ 
in a torus bundle $B/\Z \to S$
is obtained as the closure of the union of $\Real^*s(x)/\Z$, 
where $x$ runs through the complement of $s^{-1}(0)\cup s^{-1}(\infty)$. 
If $S\setminus s^{-1}(0)\cup s^{-1}(\infty)$ is Stein, $X$ bounds an annulus 
bundle over a Stein manifold which is Stein (since holomorphic fiber bundles over 
Stein manifolds with one-dimensional Stein fibers
are Stein). If the torus bundle $B/\Z$ is projective, then $X$ carries a
positive line bundle. 
%But this can happen even if $B/\Z$ is not projective.
%Take $S=\CP^1$, $B=\C^2\setminus\{0\}$,   
%$\C^2\setminus\{0\}\to\CP^1$ the natural projection and 
%$s([z_0,z_1]=(1, z_1/z_0)$. Then $B/\Z$ is a Hopf manifold
%containing the Levi-flat hypersurface defined by $\mathrm{Im}z_0=0$.
%The line bundle associated to the divisor $z_1 = 0$ is positive on $X$.
%Although $X$ is given as a hypersurface in a non-K\"ahler surface, the embedding Theorem \ref{t-embleviflat} shows that $X$ can be realized as a CR submanifold in a projective space.

\smallskip

\noindent
\textbf{(v) Fibered Levi-flats in singular holomorphic fibrations.}
Such a fibration stands for a holomorphic map $f : B\to S$ where $B$ is a complex surface
and $S$ is a compact Riemann surface. The fibers are not necessarily connected.
Let $\{p_1,\ldots, p_n\}$ be the singular values of $f$. 
A fibered Levi-flat hypersurface in $B$ has
the form $f^{-1}(\gamma)$, where 
$\gamma\subset S\setminus\{p_1,\ldots, p_n\}$ is a simple closed path.
In \cite[Section 2]{DeDu:16} examples of
fibered Levi-flat hypersurfaces are given, which carry the geometry 
of $\Real^3$, $\mathbb{H}^3$, $S^2\times\Real$, $\mathbb{H}^2\times\Real$, 
$\textrm{Nil}$, or $\textrm{Sol}$. 
In particular, $\mathbb{H}^3$ and $\mathbb{H}^2\times\Real$ are carried by fibered
Levi-flat hypersurfaces in surfaces of general type.

\smallskip
\noindent
\textbf{(vi) Levi-flat hypersurfaces in two dimensional tori an Kummer surfaces.}
For these examples we refer to \cite{Oh:06a,Oh:06b}.

\smallskip

\noindent
\textbf{(vii) Taut Levi-flat foliations.} Let $X$ be a Levi-flat CR $3$-manifold.
The Levi-foliation $\mathcal{F}$ is called taut if  
if there exists a $C^1$ embedded
circle (called transversal) in $X$ which transversely intersects every leaf of $\mathcal{F}$,
cf.\ \cite[Section\,4.4]{Cal:07}.
By results of Sullivan and Rummler \cite[Theorem\,4.31]{Cal:07}, this 
is equivalent to the fact 
that $X$ admits a $C^2$ Riemannian metric for which leaves of $\mathcal{F}$ are minimal surfaces.
Using this characterization one shows \cite[Lemma 13]{MarTor:11}:
\begin{prop}
A compact Levi-flat CR $3$-manifold possesses
a smooth CR line bundle which is positive along leaves if and only if the
Levi foliation is taut.
\end{prop}
Indeed, if $X$ possesses a positive CR line bundle then
the Ohsawa-Sibony embedding theorem implies that $X$ can be CR
embedded in a complex projective space by a $C^2$ map. 
We obtain thus a $C^2$ Riemannian metric on
$X$ by pulling back the Fubini-Study metric. Then, any leaf of $\mathcal{F}$ is minimal
since any complex submanifold in a K\"ahler manifold is minimal.
Conversely, if $X$ is taut, by smoothing a closed transversal
and regarding its intersection with the leaves of $\mathcal{F}$
as a divisor, we can construct a smooth positive CR line bundle on $X$.

\smallskip

\noindent
\textbf{(viii) Positive normal bundle.} An important CR line bundle on 
a Levi-flat CR manifold is the normal line bundle $N_\mathcal{F}$ 
to the Levi foliation $\mathcal{F}$, 
cf.\ \cite[Definition 2.15]{Ad:15}, \cite[p.\,89]{OS00}.
Brunella \cite{Bru08} observed that the positivity
of $N_\mathcal{F}$ implies convexity properties of the complement 
of a Levi-flat hypersurface in a complex
manifold (see \cite{Ad:15} for the converse and the relation to the  
Diederich-Fornaess exponent).
Explicit examples of Levi-flat CR manifold
with positive normal line bundle can be found in
\cite[Example 4.5]{Ad:15}, \cite[Example 4.2]{Bru08}.
In \cite[Th\'eor\`eme 2.2.3]{CanGon:15} the following general result
is proved for
three dimensional compact Levi-flat manifolds: if
the Levi foliation $\mathcal{F}$ has no invariant transverse measure then 
$N_\mathcal{F}$ is positive.

\smallskip

Let us finally note  that if $X$ is a Levi-flat CR manifold and $M$ 
is a projective manifold,
and $L\to X$, $E\to M$ are positive line bundles, then $X\times M$
is a Levi-flat CR manifold possessing the positive line bundle
$L\boxtimes E\to X\times M$. We can also construct examples
of Levi-flat CR manifolds possessing a positive line bundle by 
taking Galois coverings or quotients by discrete groups
of a given Levi-flat manifold with positive line bundle.

\subsection{An explicit example of Szeg\H{o} kernel}\label{SS:exSz}
Let $(L,h^L)$ be a holomorphic line bundle over a compact complex 
manifold $M$ of dimension $n-1$, where $h^L$ is a Hermitian fiber 
metric of $L$. Let $R^L$ be the curvature induced by $h^L$ 
and we assume that $iR^L>0$ on $M$. Consider $X:=M\times S^1$. 
We will identify $S^1$ with $(-\pi,\pi]$. Then, $X$ is a Levi-flat CR manifold 
and the pull-back of $(L,h^L)$ is a positive CR line bundle over $X$, denoted also $(L,h^L)$.
In this simple example, 
we will give an explicit formula for the phase function $\psi(x,y,u)$ and we will 
see that $\psi(x,y,u)$ fails to be positively homogeneous in $u$ and $\Pi_k$ 
is not a Fourier integral operator.  

Fix $k>0$. Taking a Hermitian metric on $T^{1,0}M$ with volume form
$dv_M$ and the metric $d\theta$ on $S^1$, 
we endow $X$ with the product Hermitian metric whose volume
form is $dv_X=dv_M\wedge d\theta$.
We then get natural $L^2$ inner products $(\,\cdot\,|\,\cdot\,)_k$ on $L^2(M,L^k)$ 
and $L^2(X,L^k)$. Let $B_k:L^2(M,L^k)\To{\rm Ker\,}\ddbar$
be the orthogonal projection (Bergman projection). For $f\in L^2(X,L^k)$ we
have the Fourier decomposition 
$f=\sum_{m\in\mathbb Z}e^{im\theta}f_m$ with $f_m\in L^2(M,L^k)$, 
for $m\in\mathbb Z$. We can check that the Szeg\H{o} projection $\Pi_k$ is given by
\begin{equation}\label{e-gue160617}
\begin{split}
\Pi_k:L^2(X,L^k)\To{\rm Ker\,}\ddbar_b,\quad
f=\sum_{m\in\mathbb Z}e^{im\theta}f_m\longmapsto\sum_{m\in\mathbb Z}e^{im\theta}B_kf_m.
\end{split}
\end{equation}
We now study the distribution kernel of $\Pi_k$. 
Let $s$ be a local trivializing section of $L$ on an open set $D\subset M$, 
$\abs{s}^2_{h^L}=e^{-2\phi}$, and let $B_{k,s}$ be the localization of $B_k$ 
with respect to the trivializing section $s$ (see \eqref{e-gue141001}). 
We write $x=(z,x_{2n-1})$, $y=(w,y_{2n-1})$, 
to denote the coordinates of $M\times S^1$, 
where $z=(z_1,\ldots,z_{n-1})$, $w=(w_1,\ldots,w_{n-1})$, denote coordinates
on $M$ and $x_{2n-1}$, $y_{2n-1}$, 
coordinates on $S^1$. By the works of Zelditch~\cite{Zelditch98} 
and Shiffman-Zelditch~\cite{SZ02}, see also \cite{HsM:14}, we know that 
the kernel $B_{k,s}(z,w)$ of $B_{k,s}$ has the form
\begin{equation}\label{e-gue160619}
B_{k,s}(z,w)=e^{ik\varphi(z,w)}b(z,w,k)\  \text{on $D\times D$}, 
\end{equation}
where $\varphi(z,w)\in\cC^\infty(D\times D)$, 
${\rm Im\,}\varphi(z,w)\thickapprox\abs{z-w}^2$,
$b(z,w,k)\sim\sum^\infty_{j=0}k^{n-1-j}b_j(z,w)$ in 
$S^{n-1}_{{\rm loc\,}}(1;D\times D)$ (see Definition~\ref{d-gue140826}). 
From \eqref{e-gue160619} and \eqref{e-gue160617}, 
for any $f\in\cC^\infty_0(D\times(-\pi,\pi])$, we have
\begin{equation}\label{e-gue160619I}
\begin{split}
&(\Pi_{k,s}f)(x)\\&=\sum_{m\in\mathbb Z}e^{imx_{2n-1}}
\int_M\int_{-\pi}^\pi e^{ik\varphi(z,w)}b(z,w,k)
e^{-imy_{2n-1}}f(w,y_{2n-1})\,dy_{2n-1}\,dv_M(w)\\
&=\int_M e^{ik\varphi(z,w)}b(z,w,k)f(w,x_{2n-1})\,dv_M(w)\\
&=\frac{1}{2\pi}\int_M\int_{-\pi}^\pi\int_{\mathbb R} 
e^{ik\varphi(z,w)+i\langle x_{2n-1}-y_{2n-1},\eta\rangle }
b(z,w,k)f(w,y_{2n-1})\,d\eta\, dy_{2n-1}\,dv_M(w)\\
&=\frac{1}{2\pi}\int_M\int_{-\pi}^\pi\int_{\mathbb R} 
e^{ik(\varphi(z,w)+\langle x_{2n-1}-y_{2n-1},u\rangle )}
kb(z,w,k)f(w,y_{2n-1})\,du\,dy_{2n-1}\, dv_M(w)\\
&=\frac{1}{2\pi}\int_M\int_{-\pi}^\pi \Pi_{k,s}(x,y)f(y)\,dv_X(y),
\end{split}
\end{equation}
where
\begin{equation}\label{e-gue160619II}
\Pi_{k,s}(x,y)=\int_{\mathbb R}e^{ik\psi(x,y,u)}s(x,y,u,k)\,du
\end{equation}
with 
\begin{equation}\label{e-gue160620}
\psi(x,y,u)=\varphi(z,w)+\big\langle x_{2n-1}-y_{2n-1},u\big\rangle\,,\quad 
s(x,y,u,k)=\frac{1}{2\pi}kb(z,w,k) .
\end{equation}
Formulas \eqref{e-gue160619II} and \eqref{e-gue160620} 
show that $\Pi_k$ is not a Fourier integral operator
with complex phase. The phase function $\psi(x,y,u)$ in \eqref{e-gue150129}
fails to be positively homogeneous of degree $1$ with respect to $u$.
Note also that \eqref{e-gue160619II} and \eqref{e-gue160620} 
exhibit the Szeg\H{o} kernel in the form given in Theorem \ref{t-gue150129}.
%Another difference to the representation of the Szeg\H{o}
%kernel for strictly pseudoconvex domains \eqref{e:szego_spsc}
%is the fact that the domain of integration is the whole real line, not
%the positive half-line.

%--------------
\section{The semi-classical Kohn Laplacian} \label{s-gue140824}

%{\color{red}
In this section we introduce the Kohn Laplacian $\Box^{(q)}_{b,k}$ acting on
sections of $L^k$ and we determine its local form $\Box^{(q)}_{s,k}$ with respect to a frame
$s$ and its characteristic manifold. We show that the standard symplectic form
of the cotangent bundle is non-degenerate on the characteristic manifold.
This will be used in running the heat equation method in Section
\ref{s-gue140824I}, for solving the eikonal equation \eqref{e-dhmpVIII}
(see Theorems \ref{t-gue140825}, \ref{t-dhlkmimI}, \ref{t-dhlkmimII}).
%}

We start with some notations. For $v\in\Lambda^{0,q}(T^*X)$ we denote by
$v{\wedge}:\Lambda^{0,\bullet}(T^*X)\To\Lambda^{0,\bullet+q}(T^*X)$ 
the exterior multiplication by $v$ and
let $v^{\wedge,*}:\Lambda^{0,\bullet}(T^*X)\To\Lambda^{0,\bullet-q}(T^*X)$ 
be the adjoint of $v\wedge$ 
with respect to $\langle\,\cdot\,|\,\cdot\,\rangle$. Hence, 
$\langle\,v\wedge u\,|\,g\,\rangle=\langle\,u\,|\,v^{\wedge,*}g\,\rangle$, 
for all $u\in\Lambda^{0,p}(T^*X)$, $g\in\Lambda^{0,p+q}(T^*X)$.

For any $r=0,1,\ldots,n-2$, we denote by $\ol{\pr}^{*}_{b,k}:
{\rm Dom\,}\ol{\pr}^{*}_{b,k}\subset L^2_{(0,r+1)}(X, L^k)\To L^2_{(0,r)}(X, L^k)$
%\begin{equation}\label{e-gue140828III}
%\ol{\pr}^{*}_{b,k}:{\rm Dom\,}\ol{\pr}^{*}_{b,k}\subset L^2_{(0,r+1)}(X, L^k)\To L^2_{(0,r)}(X, L^k)
%\end{equation}
the Hilbert space adjoint of $\ddbar_{b,k}$ with respect to $(\,\cdot\,|\,\cdot\, )_{k}$.
Let $\Box^{(q)}_{b,k}$ denote the (Gaffney extension of the) \emph{Kohn Laplacian} given by 
\begin{equation}\label{e-suIX}
\begin{split}
{\rm Dom\,}\Box^{(q)}_{b,k}=
\{u\in{\rm Dom\,}\ddbar_{b,k}&\cap{\rm Dom\,}\ol{\pr}^{*}_{b,k}
\subset L^2_{(0,q)}(X,L^k);\\
&\, \ddbar_{b,k}u\in{\rm Dom\,}\ol{\pr}^{*}_{b,k},\ \ol{\pr}^{*}_{b,k}u\in{\rm Dom\,}\ddbar_{b,k}\}\,,
 \end{split}
\end{equation}
and $\Box^{(q)}_{b,k}u=\ddbar_{b,k}\ol{\pr}^{*}_{b,k}u+\ol{\pr}^{*}_{b,k}\ddbar_{b,k}u$ for 
$s\in {\rm Dom\,}\Box^{(q)}_{b,k}$. Note that ${\rm Ker\,}\Box^{(0)}_{b,k}={\rm Ker\,}\ddbar_{b,k}$. 
By a result of Gaffney \cite[Proposition\,3.1.2]{MM07}, $\Box^{(q)}_{b,k}$ is a positive self-adjoint operator. 

Let $s$ be a local trivializing of $L$ on an open subset $D\subset X$. 
By using the map \eqref{isom} we have define
localizations  $\ddbar_{s,k}$ of $\ddbar_{b,k}$, $\ol{\pr}^*_{s,k}$ of
$\ol{\pr}^*_{b,k}$ and $\Box^{(q)}_{s,k}$ of $\Box^{(q)}_{b,k}$
with respect to $s$ through unitary identifications:
\begin{equation} \label{e-gue140824III} 
\left\{\begin{aligned}  
\cC^\infty_0(D,\Lambda^{0,q}(T^*X))&\longleftrightarrow \cC^\infty_0(D,L^k\otimes \Lambda^{0,q}(T^*X)) \\
u&\longleftrightarrow \Td u=U_{k,s}u, \ \ u=U_{k,s}^{-1}\Td u,\\
\ddbar_{s,k}&\longleftrightarrow\ddbar_{b,k},\ \ \ddbar_{s,k}u=U_{k,s}^{-1}\ddbar_{b,k}U_{k,s},\\
\ol{\pr}^*_{s,k}&\longleftrightarrow\ol{\pr}^*_{b,k},\ \ 
\ol{\pr}^*_{s,k}u=U_{k,s}^{-1}\ol{\pr}^*_{b,k}U_{k,s},\\
\Box^{(q)}_{s,k}&\longleftrightarrow\Box^{(q)}_{b,k},\ \ 
\Box^{(q)}_{s,k}u=U_{k,s}^{-1}\Box^{(q)}_{b,k}U_{k,s}.
\end{aligned}
\right.
\end{equation} 
It is easy to see that 
\begin{equation} \label{e-gue140824IV}
\ddbar_{s,k}=\ddbar_b+k(\ddbar_b\phi){\wedge}\,, \ \ 
\ol{\pr}^*_{s,k}=\ol{\pr}^*_b+k(\ddbar_b\phi)^{\wedge,*}
\end{equation} 
%\begin{equation} \label{e-gue140824V} 
%\ol{\pr}^*_{s,k}=\ol{\pr}^*_b+k(\ddbar_b\phi)^{\wedge,*}
%\end{equation} 
where $\ol{\pr}^*_b:\Omega^{0,q+1}(X)\To\Omega^{0,q}(X)$ is the formal 
adjoint of $\ddbar_b$ with respect to $(\,\cdot\,|\,\cdot\,)$, and 
\begin{equation} \label{e-gue140824VI}
\Box^{(q)}_{s,k}=\ddbar_{s,k}\ol{\pr}^*_{s,k}+\ol{\pr}^*_{s,k}\ddbar_{s,k}.
\end{equation}
The operator $\Box^{(q)}_{s,k}$ will be called the \emph{localized Kohn Laplacian}.

Let us choose a smooth orthonormal frame
$\{e_j\}_{j=1}^{n-1}$
for $\Lambda^{0,1}(T^*X)$ on $D$.
Let $\{Z_j\}_{j=1}^{n-1}$ denote the dual frame of $T^{0,1}X$. 
Let $Z^*_j$ be the formal adjoint of $Z_j$ with respect to $(\,\cdot\,|\,\cdot\,)$, $j=1,\ldots,n-1$, that is,
$(Z_jf\ |\ h)=(f\ |\ Z_j^*h), f, h\in \cC^\infty_0(D)$. 
%------------
\begin{prop} [{\cite[Proposition 3.1]{Hs14}}]\label{p-gue140824} 
With the notations used before, using the identification \eqref{e-gue140824III}, 
we can identify the Kohn Laplacian $\Box^{(q)}_{b,k}$ with
\begin{equation} \label{e-gue140824VII} 
\begin{split}
\Box^{(q)}_{s,k}&=\ddbar_{s,k}\ol{\pr}^*_{s,k}+\ol{\pr}^*_{s,k}\ddbar_{s,k} \\
       &=\sum^{n-1}_{j=1}(Z^*_j+k\ol Z_j(\phi))(Z_j+kZ_j(\phi))\\
       &+\sum^{n-1}_{j,t=1}e_j\wedge e^{\wedge,*}_t\circ[Z_j+kZ_j(\phi), Z^*_t+k\ol Z_t(\phi)]  \\
&+\varepsilon(Z+kZ(\phi))+\varepsilon(Z^*+k\ol Z(\phi))+f,
\end{split}
\end{equation}
where 
$\varepsilon(Z+kZ(\phi))$ denotes remainder terms of the form $\sum a_j(Z_j+k Z_j(\phi))$ 
with $a_j$ smooth, matrix-valued and independent of $k$, for all $j$, 
and similarly for $\varepsilon(Z^*+k\ol Z(\phi))$ and $f$ is a smooth function independent of $k$. 
\end{prop}
Note that  the bracket in \eqref{e-gue140824VII} is the commutator of 
$Z_j + kZ_j(\phi)$ and $Z^*_t+k\ol Z_t(\phi)$, 
$Z_j + kZ_j(\phi)$($Z^*_t+k\ol Z_t(\phi)$) is a vector field plus a function.

Until further notice, we work with some real local coordinates $x=(x_1,\ldots,x_{2n-1})$ 
defined on $D$. Let $\xi=(\xi_1,\ldots,\xi_{2n-1})$ denote the dual variables of $x$. 
Then $(x, \xi)$ are local coordinates of the cotangent bundle $T^*D$.
Let $q_j(x, \xi)$ be the semi-classical principal symbol
of $Z_j+kZ_j(\phi)$, $j=1,\ldots,n-1$. If $r_j(x,\xi)$ denotes the principal symbol of $Z_j$, 
then $q_j(x,\xi)=r_j(x,\xi)+Z_j(\phi)$. The semi-classical principal symbol of $\Box^{(q)}_{s,k}$ is given by 
\begin{equation} \label{e-gue140824VIII}
p_0=\sum^{n-1}_{j=1}\ol q_jq_j.
\end{equation}
The characteristic manifold $\Sigma$ of $\Box^{(q)}_{s,k}$ is 
\begin{equation} \label{e-gue140824a}
\begin{split}
\Sigma&=\set{(x, \xi)\in T^*D;\,  p_0(x, \xi)=0}\\
&=\set{(x, \xi)\in T^*D;\, q_1(x, \xi)=\ldots=q_{n-1}(x,\xi)=\ol q_1(x, \xi)=\ldots=\ol q_{n-1}(x, \xi)=0}.
\end{split}
\end{equation}
From \eqref{e-gue140824a}, we see that $p_0$ vanishes to second order at $\Sigma$. 

\begin{prop} \label{p-gue140824I} 
We have 
\begin{equation} \label{e-gue140824aI} 
\Sigma=\set{(x, \xi)\in T^*D;\, \xi=\lambda\omega_0(x)-2{\rm Im\,}\ddbar_b\phi(x),\lambda\in\Real}.
\end{equation}
\end{prop} 

We refer the reader to~\cite[Proposition 3.2]{Hs14} for the proof of Proposition~\ref{p-gue140824I}.

Let $\sigma=d\xi\wedge dx$ denote the canonical two form on $T^*D$. 
We are interested in whether $\sigma$ is non-degenerate at $\rho\in\Sigma$. 
We recall that $\sigma$ is non-degenerate at $\rho\in\Sigma$ if $\sigma(u,v)=0$ 
for all $v\in T_\rho\Sigma\otimes_\Real\Complex$, where $u\in T_\rho\Sigma\otimes_\Real\Complex$, then $u=0$. From now on, for any $f\in\cC^\infty(T^*D,\Complex)$, we write $H_f$ to denote the Hamilton field of $f$. That is, in local symplectic coordinates $(x, \xi)$, 
\[H_f=\sum^{2n-1}_{j=1}\Bigr(\frac{\pr f}{\pr\xi_j}\frac{\pr}{\pr x_j}-\frac{\pr f}{\pr x_j}\frac{\pr}{\pr\xi_j}\Bigr).\] 
For $f, g\in\cC^\infty(T^*D,\Complex)$, 
$\set{f, g}$ denotes the Poisson bracket of $f$ and $g$. We recall that
\[\set{f, g}=\sum^{2n-1}_{s=1}(\frac{\pr f}{\pr \xi_s}\frac{\pr g}{\pr x_s}-
\frac{\pr f}{\pr x_s}\frac{\pr g}{\pr\xi_s}). 
\]
First, we need the following. 
\begin{lem} \label{l-patI}
For $\rho=(p, \lambda_0\omega_0(p)-2{\rm Im\,}\ddbar_b\phi(p))\in\Sigma$, we have 
\begin{equation} \label{e-patIV} 
\sigma(H_{q_j}, H_{q_t})|_\rho=0,\ \ j, t=1,\ldots,n-1,
\end{equation} 
\begin{equation} \label{e-patV} 
\sigma(H_{\ol q_j}, H_{\ol q_t})|_\rho=0,\ \ j, t=1,\ldots,n-1,
\end{equation}
and 
\begin{equation} \label{e-patVI} 
\begin{split}
\sigma(H_{\ol q_j}, H_{q_t})|_\rho&=i\langle[\ol Z_j, Z_t](p), \ddbar_b\phi(p)-\pr_b\phi(p)\rangle \\
&-i(\ol Z_j Z_t+Z_t\ol Z_j)\phi(p),\ \ j, t=1,\ldots,n-1,
\end{split}
\end{equation} 
where $Z_j$ are as in \eqref{e-gue140824VII} and $q_j$ is the 
semi-classical principal symbol of $Z_j+kZ_j(\phi)$, $j=1,\ldots,n-1$.
\end{lem} 

\begin{proof} 
We write $\rho=(p, \xi_0)$. It is straightforward to see that 
\begin{equation} \label{s1-e33pat} 
\sigma(H_{q_j}, H_{q_t})|_\rho=\set{q_j, q_t}(\rho)=-\langle[Z_j, Z_t](p), \xi_0\rangle 
+i[Z_j,Z_t]\phi(p).
\end{equation} 
We have 
\begin{equation} \label{s1-e34pat}
\begin{split}
&\langle[Z_j, Z_t](p), \xi_0\rangle =\langle[Z_j, Z_t](p), \lambda_0\omega_0(p)-2{\rm Im\,}\ddbar_b\phi(p)\rangle \\
&=\lambda_0\langle[Z_j, Z_t](p), \omega_0(p)\rangle +i\langle[Z_j, Z_t](p), \ddbar_b\phi(p)-\pr_b\phi(p)\rangle .
\end{split}
\end{equation} 
Since $[Z_j, Z_t](p)\in T^{0,1}_pX$, we have 
\begin{equation} \label{s1-e35pat}
\langle[Z_j, Z_t](p), \omega_0(p)\rangle =0
\end{equation} 
and 
\begin{equation} \label{s1-e36pat} 
\langle[Z_j, Z_t](p), \pr_b\phi(p)\rangle =0.
\end{equation}
Thus, 
\begin{equation} \label{s1-e37pat}
\langle[Z_j, Z_t](p), \ddbar_b\phi(p)-\pr_b\phi(p)\rangle =\langle[Z_j, Z_t](p), \ddbar_b\phi(p)\rangle =[Z_j, Z_t]\phi(p). 
\end{equation} 
From \eqref{s1-e34pat}, \eqref{s1-e35pat} and \eqref{s1-e37pat}, we get 
\[\langle[Z_j, Z_t](p), \xi_0\rangle =i[Z_j, Z_t]\phi(p).\]
Combining this with \eqref{s1-e33pat}, we get \eqref{e-patIV}. The proof of \eqref{e-patV} is the same.

As in \eqref{s1-e33pat}, it is straightforward to see that
\begin{equation} \label{s1-e38pat} 
\sigma(H_{\ol q_j}, H_{q_t})|_\rho=\set{\ol q_j, q_t}(\rho)=\langle[\ol Z_j, Z_t](p), \xi_0\rangle 
-i(\ol Z_jZ_t+Z_t\ol Z_j)\phi(p),
\end{equation} 
where $j, t=1,\ldots,n-1$. We have 
\begin{equation} \label{s1-e39pat} 
\begin{split}
&\langle[\ol Z_j, Z_t](p), \xi_0\rangle =\langle[\ol Z_j, Z_t](p), \lambda_0\omega_0(p)-2{\rm Im\,}\ddbar_b\phi(p)\rangle \\
&=\lambda_0\langle[\ol Z_j, Z_t](p), \omega_0(p)\rangle +i\langle[\ol Z_j, Z_t](p), \ddbar_b\phi(p)-\pr_b\phi(p)\rangle .
%&=2i\lambda_0\mathcal{L}_p(\ol Z_j, Z_t)+
%i\langle[\ol Z_j, Z_t](p), \ddbar_b\phi(p)-\pr_b\phi(p)\rangle .
\end{split}
\end{equation} 
Since $X$ is Levi-flat, $\lambda_0\langle[\ol Z_j, Z_t](p), \omega_0(p)\rangle =0$ and hence 
\begin{equation} \label{s1-e39patt} 
\langle[\ol Z_j, Z_t](p), \xi_0\rangle =i\langle[\ol Z_j, Z_t](p), \ddbar_b\phi(p)-\pr_b\phi(p)\rangle .
\end{equation} 
Combining \eqref{s1-e39patt} with \eqref{s1-e38pat}, \eqref{e-patVI} follows.
\end{proof}

We need the following. 

\begin{lem}[{\cite[Lemma 4.1]{HsM:12}}]\label{l-dhI}
For any $U, V\in T^{1, 0}_pX$, pick $\mU, \mV\in\cC^\infty(D,T^{1, 0}X)$ such that $\mU(p)=U$,
$\mV(p)=V$. Then,
\begin{equation} \label{e-dhI}
R^L_p(U, \ol V)=-\big\langle\big[\,\mU, \ol{\mV}\,\big](p), \ddbar_b\phi(p)-\pr_b\phi(p)\big\rangle
+\big(\mU\ol{\mV}+\ol{\mV}\mU\big)\phi(p).
\end{equation}
\end{lem}

Now, we can prove:

%We recall that we work with the assumption that $X$ is Levi-flat. From this observation and 
%in~\cite[Theorem 3.5]{Hs14}, we conclude that:

\begin{thm}\label{t-gue140824}
$\sigma$ is non-degenerate at every point of $\Sigma$\,. 
\end{thm} 

\begin{proof}
Note that
\[\Sigma=\set{(x, \xi)\in T^*D;\, q_j(x, \xi)=\ol q_j(x, \xi)=0,\ \ j=1,\ldots,n-1}.\]
Let $\Complex T_\rho\Sigma$ and $\Complex T_\rho(T^*D)$ be the
complexifications of $T_\rho\Sigma$ and $T_\rho(T^*D)$ respectively. 
Let $T_\rho\Sigma^\bot$ be the orthogonal to
$\Complex T_\rho\Sigma$ in $\Complex T_\rho(T^*D)$ with
respect to the canonical two form $\sigma$. 
We notice that ${\rm dim}_\Complex T_\rho\Sigma^\bot=2n-2$. It is easy to check that 
\[\sigma(v,H_{q_j})|_\rho=\langle dq_j(\rho), v\rangle ,\ \ \sigma(v,H_{\ol q_j})|_\rho=\langle d\ol q_j(\rho), v\rangle ,\] 
$j=1,\ldots,n-1$, $v\in\Complex T_\rho(T^*D)$. Thus, if $v\in \Complex T_\rho\Sigma$, we get 
$\sigma(H_{q_j}, v)|_\rho=0$, $\sigma(H_{\ol q_j}, v)|_\rho=0$, $j=1,\ldots,n-1$. We conclude that
$H_{q_1},\ldots,H_{q_{n-1}},H_{\ol q_1},\ldots,H_{\ol q_{n-1}}$
is a basis for $T_\rho\Sigma^\bot$. 

Let $\nu\in \Complex T_\rho\Sigma\bigcap T_\rho\Sigma^\bot$. We write $\nu=\sum^{n-1}_{j=1}(\alpha_jH_{q_j}(\rho)+\beta_jH_{\ol q_j}(\rho))$. Since $\nu\in\Complex T_\rho\Sigma$, we have 
\[\sigma(\nu, H_{q_t})|_\rho=\sigma(\nu, H_{\ol q_t})|_\rho=0,\] 
$t=1,\ldots,n-1$. In view of \eqref{e-patIV}, \eqref{e-patV}, \eqref{e-patVI} and \eqref{e-dhI}, we see that 
\begin{equation} \label{s1-e45dh} 
\begin{split}
\sigma(\nu, H_{q_t})|_\rho&=\sum^{n-1}_{j=1}\beta_j\Bigr(-iR^L_p(\ol Z_j, Z_t)\Bigr)\\
&=-iR^L_p(Y,Z_t)=0, 
\end{split}
\end{equation} 
for all $t=1,\ldots,n-1$, where $Y=\sum^{n-1}_{j=1}\beta_j\ol Z_j(p)\in T^{1,0}_pX$. 
Since $R^L_p$ is non-degenerate, we get 
$Y=0$. Thus, $\beta_j=0$, $j=1,\ldots,n-1$. Similarly, we can repeat 
the process above to show that $\alpha_j=0$, $j=1,\ldots,n-1$. We conclude that 
$\Complex T_\rho\Sigma\bigcap T_\rho\Sigma^\bot=0$. Hence $\sigma$ is non-degenerate at 
$\rho$. The theorem follows.
\end{proof}

\section{Semi-classical Hodge decomposition for the localized Kohn Laplacian}
\label{s-gue140824I}

%{\color{red}
In this section, we will apply the method introduced in~\cite{Hs14} to establish 
semi-classical Hodge decomposition theorems for $\Box^{(0)}_{s,k}$,
based on the heat equation method of Menikoff-Sj\"ostrand \cite{MS78}.
We first add one extra variable to the local $(2n-1)$
coordinates on $X$ and introduce the operator $\Box^{(q)}_s$ acting 
in $2n$ variables and linked to the localized Kohn Laplacian $\Box^{(q)}_{s,k}$
by \eqref{e-dhmpIII}. We use the heat equation method \cite{MS78}, 
\cite[Proposition 6.5]{Hs08}, to construct a parametrix for $\Box^{(0)}_s$
in Theorem \ref{t-dcmimpII}. The corresponding Szeg\H{o} operator $S$ in
that Theorem (cf.\ \eqref{e:pm}) turns 
out to be a complex Fourier integral operator cf.\ Theorem \ref{t-dcgewI}
with phase function $\Phi$.
Returning to $\Box^{(q)}_{s,k}$ this yields the semiclassical Hodge decomposition
by Theorem \ref{t-gue140826a}, with Szeg\H{o} operators $\mathcal{S}_k$
having an expansion in Sobolev spaces cf.\ Theorem \ref{t-gue150130} given by a 
kernel with phase function $\psi$ given by the restriction of $\Phi$.
We then refine the result to show that composing with certain pseudodifferential
operators $\mathcal{A}_k$ we obtain an expansion of $\mathcal{S}_k\mathcal{A}_k$
in the $\cC^{\infty}$ topology and calculate its leading term (Theorems \ref{t-gue140826aI}
and \ref{t-gue140909}).
%}
\subsection{The heat equation for the local operator $\Box^{(0)}_s$}
\label{s-gue140824II}

Let $\Omega$ be an open set in $\Real^N$ and 
let $f$, $g$ be positive continuous functions on $\Omega$. We write $f\asymp g$ 
if for every compact set $K\subset\Omega$ there is a 
constant $c_K>0$ such that $f\leq c_Kg$ and $g\leq c_Kf$ on $K$. 

Let $s$ be a local trivializing section of $L$ on an open subset $D\Subset X$ and 
$\abs{s}^2_{h}=e^{-2\phi}$. In this section, we work with some real local coordinates 
$x=(x_1,\ldots,x_{2n-1})$ defined on $D$. We write $\xi=(\xi_1,\ldots,\xi_{2n-1})$ or 
$\eta=(\eta_1,\ldots,\eta_{2n-1})$ to denote the dual coordinates of $x$. 
We consider the domain $\widehat D:=D\times\Real$. We write 
$\widehat x:=(x,x_{2n})=(x_1,x_2,\ldots,x_{2n-1},x_{2n})$ to denote the coordinates 
of $D\times\Real$, where $x_{2n}$ is the coordinate of $\Real$. 
We write $\widehat\xi:=(\xi,\xi_{2n})$ or $\widehat\eta:=(\eta,\eta_{2n})$ to denote 
the dual coordinates of $\widehat x$, where $\xi_{2n}$ and $\eta_{2n}$ 
denote the dual coordinate of $x_{2n}$.
 We shall use the following notations: 
$\langle x,\eta\rangle:=\sum\limits^{2n-1}_{j=1}x_j\eta_j$, 
$\langle x,\xi\rangle:=\sum\limits^{2n-1}_{j=1}x_j\xi_j$, 
$\langle\widehat x,\widehat\eta\rangle:=\sum\limits^{2n}_{j=1}x_j\eta_j$, 
$\langle\widehat x,\widehat\xi\rangle:=\sum\limits^{2n}_{j=1}x_j\xi_j$. 

Let $\Lambda^{0,q}(T^*\widehat D)$ be the bundle with fiber 
\[\Lambda^{0,q}_{\widehat x}(T^*\widehat D):=\set{u\in\Lambda^{0,q}(T^*X);\, 
\widehat x=(x,x_{2n})}\]
at $\widehat x\in\widehat D$. From now on, for every point $\widehat x=(x,x_{2n})\in\widehat D$, 
we identify $\Lambda^{0,q}_{\widehat x}(T^*\widehat D)$ with $\Lambda^{0,q}_x(T^*X)$. 
Let $\langle\,\cdot\,|\,\cdot\,\rangle$ be the Hermitian metric on 
$T^*\widehat D\otimes_\Real\Complex$ given by 
$\langle\,\widehat\xi\,|\,\widehat\eta\,\rangle=\langle\,\xi\,|\,\eta\,\rangle+\xi_{2n}\ol{\eta_{2n}}$, 
$(\widehat x,\widehat\xi), (\widehat x,\widehat\eta)\in 
T^*\widehat D\otimes_\Real\Complex$. Let $\Omega^{0,q}(\widehat D)$ 
denote the space of smooth sections of $\Lambda^{0,q}(T^*\widehat D)$ over $\widehat D$ and put 
$\Omega^{0,q}_0(\widehat D):=\Omega^{0,q}(\widehat D)\cap 
\mathscr E'(\widehat D,\Lambda^{0,q}(T^*\widehat D))$.
Using 
$ku(x)=e^{-ikx_{2n}}\left(-i\frac{\pr}{\pr x_{2n}}
\left(e^{ikx_{2n}}u\right)(x)\right)$, $u\in\Omega^{0,q}(D)$, 
we consider the following operators
\begin{equation}\label{s3-lkmidhI}
\begin{split}
&\ddbar_s:\Omega^{0,r}(\widehat D)\To\Omega^{0,r+1}(\widehat D),\ \ \ddbar_{s,k}u=
e^{-ikx_{2n}}\ddbar_s(ue^{ikx_{2n}}),\ \ u\in\Omega^{0,r}(D),\\
&\ol{\pr}^*_s:\Omega^{0,r+1}(\widehat D)\To\Omega^{0,r}(\widehat D),\ \ \ol{\pr}^*_{s,k}u=
e^{-ikx_{2n}}\ol{\pr}^*_s(ue^{ikx_{2n}}),\ \ u\in\Omega^{0,r+1}(D),
\end{split}
\end{equation}
where $r=0,1,\ldots,n-1$ and $\ddbar_{s,k}$, $\ol{\pr}^*_{s,k}$ are given by \eqref{e-gue140824III}.
 From \eqref{e-gue140824IV} %and \eqref{e-gue140824V} 
 it is easy to see that 
\begin{equation}\label{e-dhmpI}
\begin{split}
&\ddbar_s=\sum^{n-1}_{j=1}\left(e_j\wedge\left(Z_j-iZ_j(\phi)\frac{\pr}{\pr x_{2n}}\right)+
(\ddbar_b e_j){\wedge} e^{\wedge,*}_j\right),\\
&\ol{\pr}^*_{s}=
\sum^{n-1}_{j=1}\left(e^{\wedge,*}_j
\left(Z^*_j-i\ol Z_j(\phi)\frac{\pr}{\pr x_{2n}}\right)+
e_j\wedge(\ddbar_b e_j)^{\wedge,*}\right),
\end{split}
\end{equation}
where $Z_1,\ldots,Z_{n-1}$, $Z^*_1,\ldots,Z^*_{n-1}$ and 
$e_1,\ldots,e_{n-1}$ are as in Proposition~\ref{p-gue140824}. Put 
\begin{equation}\label{e-dhmpII}
\Box^{(q)}_s:=\ddbar_s\ol{\pr}^*_s+
\ol{\pr}^*_s\ddbar_s:\Omega^{0,q}(\widehat D)\To\Omega^{0,q}(\widehat D).
\end{equation}
From \eqref{s3-lkmidhI}, we have 
\begin{equation}\label{e-dhmpIII}
\Box^{(q)}_{s,k}u=e^{-ikx_{2n}}\Box^{(q)}_s(ue^{ikx_{2n}}),\ \ 
\forall u\in\Omega^{0,q}(D),
\end{equation}
where $\Box^{(q)}_{s,k}$ is given by \eqref{e-gue140824III}. 
Let $u\in\Omega^{0,q}_0(\widehat D)$.
Note that
\[k\int e^{-ikx_{2n}}u(x)dx_{2n}=\int i\frac{\pr}{\pr x_{2n}}(e^{-ikx_{2n}})u(x)dx_{2n}=
\int e^{-ikx_{2n}}\Bigl(-i\frac{\pr u}{\pr x_{2n}}(x)\Bigr)dx_{2n}.\]
From this observation and the explicit formulas for $\ddbar_{s,k}$,
$\ol\pr^*_{s,k}$, $\ddbar_s$ and $\ol\pr^*_s$ (see \eqref{e-gue140824IV} and \eqref{e-dhmpI}), we conclude that
\begin{equation} \label{s3-e9bis}
\Box^{(q)}_{s,k}\int e^{-ikx_{2n}}u(x)dx_{2n}=\int e^{-ikx_{2n}}(\Box^{(q)}_su)(x)dx_{2n},
\ \ u\in\Omega^{0,q}_0(\widehat D).
\end{equation}
As in Proposition 4.1 in~\cite{Hs14}, we have: 
%----------
\begin{prop} \label{p-dhmpI} 
With the notations used before, we have
\begin{equation} \label{e-dhmpIV} 
\begin{split}
\Box^{(q)}_{s}&=\ddbar_{s}\ol{\pr}^*_{s}+\ol{\pr}^*_{s}\ddbar_{s} \\
       &=\sum^{n-1}_{j=1}\left(Z^*_j-i\ol Z_j(\phi)\frac{\pr}{\pr x_{2n}}\right)
       \left(Z_j-iZ_j(\phi)\frac{\pr}{\pr x_{2n}}\right)\\
       &+\sum^{n-1}_{j,t=1}e_j\wedge e^{\wedge,*}_t\left[Z_j-iZ_j(\phi)\frac{\pr}{\pr x_{2n}}, 
       Z^*_t-i\ol Z_t(\phi)\frac{\pr}{\pr x_{2n}}\right]\\
&+\varepsilon\left(Z-iZ(\phi)\frac{\pr}{\pr x_{2n}}\right)+
\varepsilon\left(Z^*-i\ol Z(\phi)\frac{\pr}{\pr x_{2n}}\right)+\mbox{zero order terms},
\end{split}
\end{equation}
where $\varepsilon(Z-iZ(\phi)\frac{\pr}{\pr x_{2n}})$ denotes remainder terms of the form $\sum a_j(Z_j-iZ_j(\phi)\frac{\pr}{\pr x_{2n}})$ with $a_j$ smooth, matrix-valued, for all $j$, and similarly for $\varepsilon(Z^*-i\ol Z(\phi)\frac{\pr}{\pr x_{2n}})$.
\end{prop}  
%-------------
In this paper, we will only consider $q=0$. Consider the following problem for the
heat equation
\begin{equation} \label{e-dhmpV}
\left\{ \begin{array}{ll}
(\pr_t+\Box^{(0)}_s)u(t,\widehat x)=0  & \text{ in }\Real_+\times\widehat D,  \\
u(0,\widehat x)=v(\widehat x). \end{array}\right.
\end{equation}
%%--------------
\begin{defn} \label{d-dhikbmiI}
We say that $a(t,\widehat x,\widehat \eta)\in \cC^\infty(\ol\Real_+\times T^*\widehat D)$
is quasi-homogeneous of
degree $j$ if $a(t,\widehat x,\lambda\widehat\eta)=\lambda^ja(\lambda t,\widehat x,\widehat\eta)$ for all $\lambda>0$, $\abs{\widehat\eta}\geq1$. 
%\end{defn} 
%
%\begin{defn} \label{d-dhikbmiII}
We say that $b(\widehat x,\widehat \eta)\in \cC^\infty(T^*\widehat D)$
is positively homogeneous of degree $j$ if $b(\widehat x,\lambda\widehat\eta)=\lambda^jb(\widehat x,\widehat\eta)$ for all $\lambda>0$, $\abs{\widehat\eta}\geq1$.
\end{defn} 

We look for an approximate solution of \eqref{e-dhmpV} of the form $u(t,\widehat x)=A(t)v(\widehat x)$,
\begin{equation} \label{e-dhmpVI}
A(t)v(\widehat x)=\frac{1}{(2\pi)^{2n}}\iint\!\! e^{i(\Psi(t,\widehat x,\widehat\eta)-\langle\widehat y,\widehat \eta\rangle)}a(t,\widehat x,\widehat\eta)v(\widehat y)d\widehat yd\widehat\eta
\end{equation}
where formally $a(t,\widehat x,\widehat \eta)\sim\sum\limits^\infty_{j=0}a_j(t,\widehat x,\widehat \eta)$, 
$a_j(t, \widehat x, \widehat\eta)\in \cC^\infty(\ol\Real_+\times T^*\widehat D)$, $a_j(t,\widehat x,\widehat \eta)$ is a quasi-homogeneous function of degree $-j$.
The phase $\Psi(t,\widehat x,\widehat\eta)$ should solve the eikonal equation 
\begin{equation} \label{e-dhmpVIII}
\begin{split}
 \frac{\displaystyle\pr\Psi}{\displaystyle\pr t}-i\widehat p_0(\widehat x,\Psi'_{\widehat x})&=
   O(\abs{{\rm Im\,}\Psi}^N),\forall N\geq0,   \\
 \Psi|_{t=0}&=\langle\widehat x,\widehat\eta\rangle 
 \end{split}
\end{equation}
with ${\rm Im\,}\Psi\geq0$, where $\widehat p_0$ denotes the principal symbol 
of $\Box^{(0)}_s$. From \eqref{e-dhmpIV}, we have 
\begin{equation}\label{e-dhmpIX}
\widehat p_0=\sum^{n-1}_{j=1}\ol{\widehat q}_j\widehat q_j,
\end{equation}
where $\widehat q_j$ is the principal symbol of $Z_j-iZ_j(\phi)\frac{\pr}{\pr x_{2n}}$, $j=1,\ldots,n-1$. The characteristic 
manifold $\widehat\Sigma$ of $\Box^{(0)}_s$ is given by 
\begin{equation}\label{e-dhmpX}
\widehat\Sigma=
\set{(\widehat x,\widehat\xi)\in T^*\widehat D;\, \widehat q_1(\widehat x,\widehat\xi)
=\ldots=\widehat q_{n-1}(\widehat x,\widehat\xi)
=\ol{\widehat q}_1(\widehat x,\widehat\xi)=\ldots
=\ol{\widehat q}_{n-1}(\widehat x,\widehat\xi)=0}.
\end{equation}
From \eqref{e-dhmpX}, we see that $\widehat p_0$ vanishes to second order at 
$\widehat\Sigma$. Let $\widehat\sigma$ denote the canonical two form on $T^*\widehat D$. 
As in Proposition~\ref{p-gue140824I} and Theorem~\ref{t-gue140824}, we have 

\begin{thm} \label{t-gue140825} 
With the notations used above, we have
\begin{equation} \label{e-dhmpXI} 
\widehat\Sigma=\set{(\widehat x, \widehat\xi)\in T^*\widehat D;\, \widehat\xi=
(\lambda\omega_0(x)-2{\rm Im\,}\ddbar_b\phi(x)\xi_{2n},\xi_{2n}),\lambda\in\Real}.
\end{equation}
%--------------
Put 
\begin{equation}\label{e-gue140825}
\begin{split}
&\widehat\Sigma_+=\set{(\widehat x, \widehat\xi)\in T^*\widehat D;\, \widehat\xi=
(\lambda\omega_0(x)-2{\rm Im\,}\ddbar_b\phi(x)\xi_{2n},\xi_{2n}),\lambda\in\Real, \xi_{2n}>0},\\
&\widehat\Sigma_-=\set{(\widehat x, \widehat\xi)\in T^*\widehat D;\, \widehat\xi=
(\lambda\omega_0(x)-2{\rm Im\,}\ddbar_b\phi(x)\xi_{2n},\xi_{2n}),\lambda\in\Real, \xi_{2n}<0}.
\end{split}
\end{equation}
Then, $\widehat\sigma$ is non-degenerate at every point of $\widehat\Sigma_+\cup\widehat\Sigma_-$.
\end{thm}
%--------------
Consider the conic open set of $T^*\widehat D$ defined by
\begin{equation}\label{e-dhmpXIII}
U=\set{(\widehat x,\widehat\xi)\in T^*\widehat D;\, \widehat\xi=(\xi,\xi_{2n}), \xi_{2n}>0}.
\end{equation}
Until further notice, we work in $U$. 
Since $\widehat\sigma$ is non-degenerate at each point of $U\cap \widehat\Sigma=\widehat\Sigma_+$, 
\eqref{e-dhmpVIII} can be solved with ${\rm Im\,}\Psi\geq0$ on $U$. More precisely, we have the following. 
%--------------
\begin{thm} \label{t-dhlkmimI}
There exists $\Psi(t,\widehat x,\widehat\eta)\in \cC^\infty(\ol\Real_+\times U)$ such
that $\Psi(t,\widehat x,\widehat\eta)$ is quasi-homogeneous of degree $1$ and ${\rm Im\,}\Psi\geq 0$ 
and such that \eqref{e-dhmpVIII} holds where the error term is uniform on every set of the form 
$[0,T]\times K$ with $T>0$ and $K\subset U$ compact.
Furthermore, $\Psi$ is unique up to a term which is
$O(\abs{{\rm Im\,}\Psi}^N)$ locally uniformly for every $N$ and
\begin{equation}\label{e-dhlkmimIII}
\begin{split}
\mbox{$\Psi(t,\widehat x,\widehat \eta)=\langle\widehat x,\widehat \eta\rangle$ on $\widehat\Sigma_+$},\\
\mbox{$d_{\widehat x,\widehat\eta}(\Psi-\langle\widehat x,\widehat\eta\rangle)=0$ on $\widehat\Sigma_+$}.
\end{split}
\end{equation}
Moreover, we have
\begin{equation} \label{e-dhlkmimIIIa}
{\rm Im\,}\Psi(t, \widehat x,\widehat\eta)\asymp\Big(\abs{\widehat\eta}
\frac{t\abs{\widehat\eta}}{1+t\abs{\eta}}\Big)\Big(\dist 
\big((\widehat x, \frac{\widehat\eta}{\abs{\widehat\eta}}),\widehat\Sigma_+\big)\Big)^2,\ \ t\geq0,\ \ (\widehat x,\widehat\eta)\in U.
\end{equation}
Furthermore, we can take $\Psi(t,\widehat x,\widehat\eta)$ so that 
\begin{equation}\label{e-dhlkmimVI}
\Psi(t,\widehat x,\widehat\eta)=\Psi(t,(x,0),\widehat\eta)+x_{2n}\eta_{2n}.
\end{equation}
\end{thm} 

\begin{thm} \label{t-dhlkmimII}
There exists a function $\Psi(\infty,\widehat x,\widehat\eta)\in \cC^\infty(U)$ with a
uniquely determined Taylor expansion at each point of\, $\widehat\Sigma_+$ such that 
$\Psi(\infty,\widehat x,\widehat\eta)$ is positively homogeneous of degree $1$ and
for every compact set $K\subset U$ there is a $c_K>0$ such that   
${\rm Im\,}\Psi(\infty,\widehat x,\widehat \eta)\geq 
c_K\abs{\widehat\eta}\Big(\dist \big((\widehat x,\frac{\widehat\eta}{\abs{\widehat\eta}}),\widehat\Sigma_+\big)\Big)^2$,  
$d_{\widehat x,\widehat\eta}(\Psi(\infty, \widehat x, \widehat\eta)-\langle\widehat x,\widehat\eta\rangle)=
0\text{ on } \widehat\Sigma_+$.
If $\lambda\in C(U)$, $\lambda>0$ and 
$\lambda(\widehat x,\widehat\xi)<\min\lambda_j(\widehat x,\widehat\xi)$, 
for all $(\widehat x,\widehat\xi)=
(\widehat x,(\lambda\omega_0(x)-2{\rm Im\,}\ddbar_b\phi(x)\xi_{2n},\xi_{2n}))\in\widehat\Sigma_+$, 
where $\lambda_j(\widehat x,\widehat\xi)$ are the eigenvalues of the Hermitian quadratic form $\xi_{2n}R^L_x$, 
then the solution $\Psi(t,\widehat x,\widehat\eta)$ of \eqref{e-dhmpVIII} can be chosen so that for 
every compact set $K\subset U$ and all indices $\alpha$, $\beta$, $\gamma$,
there is a constant $c_{\alpha,\beta,\gamma,K}>0$ such that
\begin{equation} \label{e-dhlkmimIV}
\abs{\pr^\alpha_{\widehat x}\pr^\beta_{\widehat\eta}\pr^\gamma_t(\Psi(t,\widehat x,\widehat\eta)-
\Psi(\infty,\widehat x,\widehat\eta))}
\leq c_{\alpha,\beta,\gamma,K}e^{-\lambda(\widehat x,\widehat \eta)t} \text{ on }\ol\Real_+\times K.
\end{equation}
\end{thm} 
%--------------
For the proofs of Theorem~\ref{t-dhlkmimI} and Theorem~\ref{t-dhlkmimII}, 
we refer to Menikoff-Sj\"{o}strand~\cite{MS78}, \cite{Hs08} and~\cite[Section 4.1]{Hs14}. 

From now on, we assume that $\Psi(t,\widehat x,\widehat\eta)$ has the form \eqref{e-dhlkmimVI} and hence 
\begin{equation}\label{e-dhlkmimVII}
\Psi(\infty,\widehat x,\widehat\eta)=\Psi(\infty,(x,0),\widehat\eta)+x_{2n}\eta_{2n}.
\end{equation}
%-------------
We let the full symbol of $\Box^{(0)}_s$ be $\sum^2_{j=0}\widehat p_j(\widehat x,\widehat\xi)$, 
where $\widehat p_j(\widehat x,\widehat\xi)$ is positively homogeneous of order $2-j$. We apply $\pr_t+\Box^{(0)}_s$ formally
under the integral in \eqref{e-dhmpVI} and then introduce the asymptotic expansion of
$\Box^{(0)}_s(ae^{i\Psi})$. Setting $(\pr_t+\Box^{(0)}_s)(ae^{i\Psi})\sim 0$ and regrouping
the terms according to the degree of quasi-homogeneity, we obtain for each $N$ the transport equations
\begin{equation} \label{e-dhlkmimVIII}
\left\{ \begin{array}{ll}
 T(t,\widehat x,\widehat\eta,\pr_t,\pr_{\widehat x})a_0=O(\abs{{\rm Im\,}\Psi}^N),   \\
 T(t,\widehat x,\widehat\eta,\pr_t,\pr_{\widehat x})a_j+R_j(t,\widehat x,\widehat \eta,a_0,\ldots,a_{j-1})= 
 O(\abs{{\rm Im\,}\Psi}^N)\,.
 \end{array}\right.
\end{equation}
Here $$T(t,\widehat x,\widehat\eta,\pr_t,\pr_{\widehat x})=
\pr_t-i\sum^{2n}_{j=1}\frac{\pr\widehat p_0}{\pr\xi_j}(\widehat x,\Psi'_{\widehat x})\frac{\pr}{\pr x_j}
+q(t,\widehat x,\widehat\eta),$$ 
where $$q(t,\widehat x,\widehat\eta)=\widehat p_1(\widehat x,\Psi'_{\widehat x})+
\frac{1}{2i}\sum^{2n}_{j,t=1}\frac{\pr^2\widehat p_0(\widehat x,\Psi'_{\widehat x})}
    {\pr\xi_j\pr\xi_t}\frac{\pr^2\Psi(t,\widehat x,\widehat\eta)}{\pr x_j\pr x_t}$$
and $R_j$ is a linear differential operator acting on $a_0,a_1,\ldots,a_{j-1}$. 
We note that $q(t,\widehat x,\widehat \eta)\To q(\infty,\widehat x,\widehat\eta)$ as $t\to\infty$, 
exponentially fast in the sense of \eqref{e-dhlkmimIV} 
and the same is true for the coefficients of $R_j$, for all $j$.

Following~\cite{Hs14}, we can solve the transport equations \eqref{e-dhlkmimVIII}. 
To state the results precisely, we pause and introduce some symbol spaces. 

\begin{defn} \label{d-d-msmapamo}
Let $\mu\geq0$ be a non-negative constant. We say that $a\in\widehat S^m_\mu(\ol\Real_+\times U)$
if $a\in \cC^\infty(\ol\Real_+\times U)$
and for all indices $\alpha, \beta\in\mathbb N^{2n}_0$, $\gamma\in\mathbb N_0$, 
every compact set $K\Subset\widehat D$, there exists a constant $c>0$ such that
$\abs{\pr^\gamma_t\pr^\alpha_{\widehat x}\pr^\beta_{\widehat\eta}a(t,\widehat x,\widehat\eta)}\leq
 ce^{-t\mu\abs{\eta_{2n}}}(1+\abs{\eta})^{m+\gamma-\abs{\beta}}$, $\widehat x\in K$, $(\widehat x,\widehat\eta)\in U$.
 \end{defn}
 
Put $\widehat S^{-\infty}_\mu(\ol\Real_+\times U):=\bigcap _{m\in\Real}\widehat S^m_\mu(\ol\Real_+\times U)$.
Let $a_j\in\widehat S^{m_j}_\mu(\ol\Real_+\times U)$, $j\in\N_0$, with 
$m_j\To-\infty$, $j\To\infty$. Then there exists $a\in\widehat S^{m_0}_\mu(\ol\Real_+\times U)$, 
unique modulo $\widehat S^{-\infty}_\mu(\ol\Real_+\times U)$, such that 
$a-\sum\limits^{k-1}_{j=0}a_j\in\widehat S^{m_k}_\mu(\ol\Real_+\times U)$ for $k\in\N_0$. 
If $a$ and $a_j$ have the properties above, we write $a\sim\sum\limits^{\infty}_{j=0}a_j$ 
in $\widehat S^{m_0}_\mu(\ol\Real_+\times U)$. Following the proof of  \cite[Theorem 4.15]{Hs14} we get:
%Let $\widehat S^m_{{\rm qcl\,},\mu}(\ol\Real_+\times U)$ be the space of all symbols 
%$a(t,\widehat x,\widehat\eta)\in\widehat S^m_\mu(\ol\Real_+\times U)$ with 
%\[\mbox{$a(t,\widehat x,\widehat\eta)\sim\sum\limits^\infty_{j=0}a_{m-j}(t,\widehat x,\widehat\eta)$ 
%in $\widehat S^m_\mu(\ol\Real_+\times U)$},\]
%with $a_k(t,\widehat x,\widehat\eta)\in\widehat S^k_\mu(\ol\Real_+\times U)$ quasi-homogeneous of degree $k$, $\forall k$.

%-------------
\begin{thm} \label{t-aldhmpII}
We can find solutions
$a_j(t,\widehat x,\widehat\eta)\in\widehat S^{-j}_0(\ol\Real_+\times U)$, $j=0, 1,\ldots$ of the system \eqref{e-dhlkmimVIII}, where $a_j(t,\widehat x,\widehat\eta)$ is a quasi-homogeneous function of degree $-j$, for each $j$, with 
\begin{equation}\label{e-gue140825a}
\begin{split}
\mbox{$a_0(0, \widehat x, \widehat\eta)=1$ on $U$},\ \ \mbox{$a_j(t,\widehat x,\widehat\eta)=0$ on $U$, $j= 1,2,\ldots$},
\end{split}
\end{equation}
\begin{equation}\label{e-gue140825aI}
\begin{split}
&a_j(t,\widehat x,\widehat\eta)-a_j(\infty,\widehat x,\widehat\eta)\in\widehat S^{-j}_\mu(\ol\Real_+\times U),\ \ j=0,1,2,\ldots,\\
&a_0(\infty,\widehat x,\widehat\eta)\neq0,\ \ \forall (\widehat x,\widehat\eta)\in\widehat\Sigma_+,
\end{split}
\end{equation}
where $\mu>0$ is a constant and $a_j(\infty,\widehat x,\widehat\eta)\in \cC^\infty(U)$, $j=0,1,\ldots$, $a_j(\infty,\widehat x,\widehat\eta)$ is a positively homogeneous function of degree $-j$, for each $j$. 
\end{thm}
%------------- 
Let $m\in\Real$, $0\leq\rho,\delta\leq1$. For a conic open subset $\Gamma$ of $T^*\widehat D$, let $S^m_{\rho,\delta}(\Gamma)$ denote the H\"{o}rmander symbol space on $\Gamma$ of order $m$ type $(\rho,\delta)$ (see~\cite[Definition 1.1]{GrSj94})
and let $S^m_{{\rm cl\,}}(\Gamma)$
denote the space of classical symbols on $\Gamma$ of order $m$ (see~\cite[p.\,35]{GrSj94}). 
Let $B\subset D$ be an open set. Let 
$L^m_{\frac{1}{2},\frac{1}{2}}(B)$ and $L^m_{{\rm cl\,}}(B)$ denote the space of
pseudodifferential operators on $B$ of order $m$ type $(\frac{1}{2},\frac{1}{2})$ and the space of classical 
pseudodifferential operators on $B$ of order $m$. The classical result of Calderon and Vaillancourt 
\cite[Theorem\,18.6.6]{Hor3:83}
tells us that for any $A\in L^m_{\frac{1}{2},\frac{1}{2}}(B)$, 
\begin{equation}\label{e-mslknaXI}
\mbox{$A:H^s_{\rm comp}(B)\To H^{s-m}_{\rm loc}(B)$ is continuous, for every $s\in\Real$}. 
\end{equation}

We return to our situation. For $j\in\N_0$, let $a_j(t,\widehat x,\widehat\eta)\in\widehat S^{-j}_0(\ol\Real_+\times U)$
and $a_j(\infty,\widehat x,\widehat\eta)\in \cC^\infty(U)$ be as in Theorem~\ref{t-aldhmpII}. Let
\begin{equation} \label{e-dcmV}
\begin{split}
&\mbox{$a(\infty,\widehat x,\widehat\eta)\sim\sum\limits^\infty_{j=0}a_j(\infty,\widehat x, \widehat\eta)$ in $S^{0}_{1,0}(U)$},\\
&\mbox{$a(t,\widehat x,\widehat \eta)\sim\sum\limits^\infty_{j=0}a_j(t,\widehat x,\widehat\eta)$ in $\widehat S^{0}_0(\ol\Real_+\times U)$},\\
&a(t,\widehat x,\widehat\eta)-a(\infty,\widehat x,\widehat\eta)\in\widehat S^0_{\mu}(\ol\Real_+\times U),\ \ \mu>0.
\end{split}
\end{equation}
Take $\alpha(\eta_{2n})\in \cC^\infty(\Real)$ with $\alpha(\eta_{2n})=1$ if $\eta_{2n}\leq\frac{1}{2}$, $\alpha(\eta_{2n})=0$ if $\eta_{2n}\geq 1$. Choose $\chi\in \cC^\infty_0
(\Real^{2n})$ so that $\chi(\widehat\eta)=1$ when $\abs{\widehat\eta}<1$ and 
$\chi(\widehat\eta)=0$ when $\abs{\widehat\eta}>2$. For $\varepsilon>0$, put 
\[\begin{split}
G_\varepsilon(\widehat x,\widehat y) &= \frac{1}{(2\pi)^{2n}}
\int\Bigr(\int^{\infty}_0\!\!\bigr(e^{i(\Psi(t,\widehat x,\widehat\eta)-\langle\widehat y,\widehat\eta\rangle)}a(t,\widehat x,\widehat\eta)  \\
  &\quad -e^{i(\Psi(\infty,\widehat x,\widehat\eta)-\langle\widehat y,\widehat\eta\rangle)}
  a(\infty,\widehat x,\widehat\eta)\bigr)(1-\chi(\widehat\eta))\chi(\varepsilon\widehat\eta)(1-\alpha(\eta_{2n}))dt\Bigr)d\widehat\eta.
\end{split}\]
By Chapter 5 in part {\rm I\,} of~\cite{Hs08}, we have
for any $u\in \cC^\infty_0(\widehat D)$, $\lim_{\varepsilon\To0}\int G_{\varepsilon}(\widehat x, \widehat y)u(\widehat y)d\widehat y\in \cC^\infty(\widehat D)$, 
and the operator $G: \cC^\infty_0(\widehat D)\To \cC^\infty(\widehat D)$, $u\mapsto\lim_{\varepsilon\To0}\int G_{\varepsilon}(\widehat x, \widehat y)u(\widehat y)dy$,
is continuous, has a unique continuous extension: $G: \mathcal E'(\widehat D)\To\mathcal D'(\widehat D)$
and $G\in L^{-1}_{\frac{1}{2},\frac{1}{2}}(\widehat D)$
with symbol
\[q(\widehat x,\widehat\eta)=\int^{\infty}_0\Bigr(e^{i(\Psi(t,\widehat x,\widehat\eta)-\langle\widehat x,\widehat\eta\rangle)}a(t,\widehat x,\widehat\eta)-
e^{i(\Psi(\infty,\widehat x,\widehat\eta)-\langle\widehat x,\widehat\eta\rangle)}a(\infty,\widehat x,\widehat\eta)\Bigr)dt(1-\alpha(\eta_{2n}))\]
in $S^{-1}_{\frac{1}{2},\frac{1}{2}}(T^*\widehat D)$.
We denote
\begin{equation} \label{e-dcmVIII}\begin{split}
G(\widehat x,\widehat y) &= \frac{1}{(2\pi)^{2n}}\int\Bigr(\int^{\infty}_0\!\!
\bigr(e^{i(\Psi(t,\widehat x,\widehat\eta)-\langle\widehat y,\widehat\eta\rangle)}a(t,\widehat x,\widehat\eta)  \\
  &\quad -e^{i(\Psi(\infty,\widehat x,\widehat\eta)-\langle\widehat y,\widehat\eta\rangle)}a(\infty,\widehat x,\widehat\eta)\bigr)
  (1-\chi(\widehat\eta))(1-\alpha(\eta_{2n}))dt\Bigr)d\widehat\eta.
\end{split}\end{equation}
Similarly, for $\varepsilon>0$, put 
\[
S_\varepsilon(\widehat x,\widehat y)=\frac{1}{(2\pi)^{2n}}
\int e^{i(\Psi(\infty,\widehat x,\widehat\eta)-\langle\widehat y,\widehat\eta\rangle)}
a(\infty,\widehat x,\widehat\eta)
(1-\chi(\widehat\eta))\chi(\varepsilon\widehat\eta)(1-\alpha(\eta_{2n}))d\widehat\eta.\]
By \cite[Chapter 5, part I]{Hs08})
we have for $u\in \cC^\infty_0(\widehat D)$, 
$\lim_{\varepsilon\To0}\int S_{\varepsilon}
(\widehat x, \widehat y)u(\widehat y)d\widehat y\in \cC^\infty(\widehat D)$, the operator 
\begin{equation}
S: \cC^\infty_0(\widehat D)\To \cC^\infty(\widehat D), \quad
u\mapsto\lim_{\varepsilon\To0}\int S_{\varepsilon}(\widehat x, \widehat y)u(\widehat y)dy,
\end{equation}
is continuous, has a unique continuous extension: $S: \mathcal E'(\widehat D)\To\mathcal D'(\widehat D)$
and $S\in L^{0}_{\frac{1}{2},\frac{1}{2}}(\widehat D)$
with symbol $s(\widehat x,\widehat\eta)=e^{i(\Psi(\infty,\widehat x,\widehat\eta)-\langle\widehat x,\widehat\eta\rangle)}
a(\infty,\widehat x,\widehat\eta)(1-\alpha(\eta_{2n}))\in S^{0}_{\frac{1}{2},\frac{1}{2}}(T^*\widehat D)$. We denote
\begin{equation} \label{e-dcmapaI}
S(\widehat x,\widehat y)=\frac{1}{(2\pi)^{2n}}\int e^{i(\Psi(\infty,\widehat x,\widehat\eta)-\langle\widehat y,\widehat\eta\rangle)}
a(\infty,\widehat x,\widehat\eta)\big(1-\chi(\widehat\eta)\big)\big(1-\alpha(\eta_{2n})\big)d\widehat\eta.
\end{equation} 
%----------------
Put 
\begin{equation}\label{e-gue140825b}
\Td I=(2\pi)^{-2n}\int e^{i\langle\widehat x-
\widehat y,\widehat\eta\rangle}\big(1-\alpha(\eta_{2n})\big)d\widehat\eta.
\end{equation}
%----------------
We can repeat the proof of  \cite[Proposition 6.5]{Hs08} with minor changes and obtain: 
%----------------
\begin{thm} \label{t-dcmimpII}
With the notations used above, we have
\begin{equation}\label{e:pm}
S+\Box^{(0)}_s\circ G\equiv\Td I\ \ \mbox{on $\widehat D$},\ \ 
\ddbar_s\circ S\equiv0\ \ \mbox{on $\widehat D$},\ \ 
\Box^{(0)}_s\circ S\equiv 0\ \ \mbox{on $\widehat D$}.
\end{equation} 
\end{thm} 
%----------------
The next result follows from the complex stationary phase formula~\cite{MS74} 
with essentially the same proof as of \cite[Theorem 4.29]{Hs14}.
%----------------
\begin{thm} \label{t-dcgewI}
With the notations and assumptions above, let 
$S=S(\widehat x,\widehat y)\in L^0_{\frac{1}{2},\frac{1}{2}}(\widehat D)$ 
be as in Theorem~\ref{t-dcmimpII}. Then, on $\widehat D$, we have 
\begin{equation}\label{e-dcgewI}
S(\widehat x, \widehat y)\equiv\int_{u\in\Real, t\in\Real_+}
 e^{i\Phi(\widehat x, \widehat y, u,t)}b(\widehat x, \widehat y, u,t)(1-\alpha(t))dudt
\end{equation}
with symbol
\begin{equation} \label{e-dcgewII}
\begin{split}
&\mbox{$b(\widehat x, \widehat y, u, t)\sim\sum\limits^\infty_{j=0}
b_j(\widehat x,\widehat y, u,t)$ in 
$S^{n-1}_{1,0}(\widehat D\times\widehat D\times\Real\times\Real_+)$},\\
%&b_j(\widehat x,\widehat y, u,t)\in \cC^\infty(\widehat D\times\widehat D\times\Real\times\Real_+),\ \ j=0,1,2,\ldots,\\
&b_j(\widehat x,\widehat y,\lambda u,\lambda t)=
\lambda^{n-1-j}b_j(\widehat x,\widehat y,u,t),\ \ 
\forall(\widehat x,\widehat y,u,t)\in\widehat D\times\widehat D\times\Real\times\Real_+,\ \ 
\lambda\geq1,\ \ \forall j,\\
&b_0(\widehat x,\widehat x,u,t)\neq0,\ \ 
\forall(\widehat x,\widehat y,u,t)\in\widehat D\times\widehat D\times\Real\times\Real_+,\ \ 
\lambda\geq1,
\end{split}
\end{equation}
and phase function
\begin{equation}\label{e-gue140825f}
\begin{split}
&\Phi(\widehat x,\widehat y,u,t)=(x_{2n}-y_{2n})t+\varphi(x,y,u,t),\ \ \varphi(x,y,u,t)\in \cC^\infty(D\times D\times\Real\times\Real_+),\\
&\varphi(x,y,\lambda u,\lambda t)=\lambda\varphi(x,y,u,t),\ \ \forall(x,y,u,t)\in D\times D\times\Real\times\Real_+,\ \ \lambda\geq1,\\
&{\rm Im\,}\varphi(x,y,u,t)\geq0,\ \ \varphi(x,x,u,t)=0,\ \ \forall x\in D,\ \ u\in\Real,\ \ t\in\Real_+,\\
&d_x\varphi|_{(x,x,u,t)}=-2t{\rm Im\,}\ddbar_b\phi(x)+u\omega_0(x),\ \ \forall x\in D,\ \ u\in\Real,\ \ t\in\Real_+,\\
&d_y\varphi|_{(x,x,u,t)}=2t{\rm Im\,}\ddbar_b\phi(x)-u\omega_0(x),\ \ \forall x\in D,\ \ u\in\Real,\ \ t\in\Real_+,\\
&\mbox{$\frac{\pr\varphi}{\pr u}(x,y,u,t)=0$ and $\frac{\pr\varphi}{\pr t}(x,y,u,t)=0$ if and only if $x=y$}.
\end{split}
\end{equation}
\end{thm}

We can repeat the method in~\cite[Section 4.4]{Hs14} with minor changes to 
compute the tangential Hessian of the phase function $\varphi(x,y,u,t)$.  
This will yield theTaylor expansion of the phase function $\psi$ from 
Theorems \ref{t-gue150129} and \ref{t-main},
see Theorem \ref{t-phase}. 
Since the computation is simpler we therefore omit the details. 
We only state the result. Fix $p\in D$ and let $\ol Z_{1},\ldots,\ol Z_{n-1}$ be an orthonormal frame 
of $T^{1,0}_xX$ varying smoothly with $x$ in a neighbourhood of $p$, 
for which the Hermitian quadratic form $R^L_x$ is diagonalized at $x=p$. 
Let $s$ be a local trivializing section of $L$ and let $x=(x_1,\ldots,x_{2n-1})$, $z_j=x_{2j-1}+ix_{2j}$, $j=1,\ldots,n-1$, 
be local coordinates of $X$ defined in some small neighbourhood of $p$ such that 
\begin{equation}\label{e-gue140826m}
\begin{split}
&x(p)=0,\ \omega_0(p)=dx_{2n-1},\ T(p)=-\frac{\pr}{\pr x_{2n-1}},\\ 
& \Big\langle\,\frac{\pr}{\pr x_j}(p)\,|\,\frac{\pr}{\pr x_t}(p)\,\Big\rangle
=2\delta_{j,t},\ \ j,t=1,\ldots,2n-2,\\
&\ol Z_{j}(x)=\frac{\pr}{\pr z_j}+
i\sum^{n-1}_{t=1}\tau_{j,t}\ol z_t\frac{\pr}{\pr x_{2n-1}}
+c_jx_{2n-1}\frac{\pr}{\pr x_{2n-1}}+O(\abs{x}^2),\ \ j=1,\ldots,n-1,\\
%&\ol Z_{j}=\frac{\pr}{\pr z_j},\ \ j=1,\ldots,n-1,\\
%&\phi(x)=\frac{1}{2}\sum^{n-1}_{l,t=1}\mu_{t,l}z_t\ol z_l+
%\sum^{n-1}_{l,t=1}\bigr(a_{l,t}z_lz_t+\ol a_{l,t}\ol z_l\ol z_t\bigr)
%+\sum^{n-1}_{j=1}\bigr(d_jz_jx_{2n-1}+\ol d_j\ol z_jx_{2n-1}\bigr)\\
%&\quad+O(\abs{x_{2n-1}}^2)+O(\abs{x}^3),\quad\mu_{j,t}, a_{j,t}, d_j\in\Complex, 
%\,\mu_{j,t}=\ol\mu_{t,j},\, j, t=1,\ldots,n-1.\\
&\phi(x)=\beta x_{2n-1}+\sum^{n-1}_{j=1}\bigr(\alpha_jz_j+\ol\alpha_j\ol z_j\bigr)+\frac{1}{2}\sum^{n-1}_{l,t=1}\mu_{t,l}z_t\ol z_l+\sum^{n-1}_{l,t=1}\bigr(a_{l,t}z_lz_t+\ol a_{l,t}\ol z_l\ol z_t\bigr)\\
&\quad+\sum^{n-1}_{j=1}\bigr(d_jz_jx_{2n-1}+\ol d_j\ol z_jx_{2n-1}\bigr)+O(\abs{x_{2n-1}}^2)+O(\abs{x}^3),
\end{split}
\end{equation}
where $\beta\in\Real$, $\tau_{j,t}$, $c_j$, $\alpha_j$, $\mu_{j,t}$, $a_{j,t}$, $d_j$ are complex numbers, 
$\mu_{j,t}=\ol\mu_{t,j}$, $\tau_{j,t}+\ol\tau_{t,j}=0$, $j, t=1,\ldots,n-1$.
%---------
We define now the phase function
\begin{equation}\label{e-gue160630a}
\psi(x,y,u):=\varphi(x,y,u,1).
\end{equation}
Note that $\varphi(x,y,u,t)$ is positively homogeneous of degree $1$ with respect to 
$(u,t)$ but $\psi(x,y,u)$ fails to be positively homogeneous of degree $1$ with respect to $u$. 
%Using Theorem \ref{t-dcgewI} we give next the Taylor expansion of $\psi$.
We work in local coordinates as in \eqref{e-gue140826m} and for $x=(x_1,\ldots,x_{2n-1})$
we denote $x'=(x_1,\ldots,x_{2n-2})$, $\abs{x'}^2=\sum^{2n-2}_{j=1}\abs{x_j}^2$.
%-----------------
\begin{thm}\label{t-phase}
There exists a neighborhood $D_0$ of $p$ and $c>0$ 
such that for all  $(x,y,u)\in D_0\times D_0\times\Real$ we have
\begin{equation}\label{e-dgugeXI}
\begin{split}
&{\rm Im\,}\psi(x,y,u)\geq c\abs{x'-y'}^2, \\
&{\rm Im\,}\psi(x,y,u)+\abs{\frac{\pr\psi}{\pr u}(x,y,u)}\geq 
c(\abs{x_{2n-1}-y_{2n-1}}+\abs{x'-y'}^2)
\end{split}
\end{equation}
Moreover, there exists a function $f\in\cC^\infty(D_0)$, 
$f(0,0,u)=0$ for $u\in\Real$, such that 
%\begin{equation}\label{e-gue140826Im}
%\begin{split}
%\psi(x,y,u)&= -i\sum^{n-1}_{j=1}\alpha_j(z_j-w_j)+
%i\sum^{n-1}_{j=1}\ol\alpha_j(\ol z_j-\ol w_j)+u(x_{2n-1}-y_{2n-1})\\
%&\quad-\frac{i}{2}\sum^{n-1}_{j,l=1}(a_{l,j}+a_{j,l})(z_jz_l-w_jw_l)+
%\frac{i}{2}\sum^{n-1}_{j,l=1}(\ol a_{l,j}+\ol a_{j,l})(\ol z_j\ol z_l-\ol w_j\ol w_l)\\
%&\quad+\sum^{n-1}_{j=1}(-id_j)(z_jx_{2n-1}-w_jy_{2n-1})+
%\sum^{n-1}_{j=1}(i\ol d_j)(\ol z_jx_{2n-1}-\ol w_jy_{2n-1})\\
%&\quad-\frac{i}{2}\sum^{n-1}_{j=1}\lambda_j(z_j\ol w_j-\ol z_jw_j)+
%\frac{i}{2}\sum^{n-1}_{j=1}\lambda_j\abs{z_j-w_j}^2\\
%&\quad+(x_{2n-1}-y_{2n-1})f(x,y,u)+O(\abs{(x, y)}^3),\\ 
%\end{split}
%\end{equation}
\begin{equation}\label{e-gue140826Im}
\begin{split}
\psi&(x,y,u)
=-i\sum^{n-1}_{j=1}\alpha_j(z_j-w_j)+
i\sum^{n-1}_{j=1}\ol\alpha_j(\ol z_j-\ol w_j)+u(x_{2n-1}-y_{2n-1})\\
&-\frac{i}{2}\sum^{n-1}_{j,l=1}(a_{l,j}+a_{j,l})(z_jz_l-w_jw_l)
+\frac{i}{2}\sum^{n-1}_{j,l=1}(\ol a_{l,j}+\ol a_{j,l})(\ol z_j\ol z_l-\ol w_j\ol w_l)\\
&+\frac{1}{2}\sum^{n-1}_{j,l=1}iu(\ol\tau_{l,j}-\tau_{j,l})(z_j\ol z_l-w_j\ol w_l)
+\sum^{n-1}_{j=1}(-ic_j\beta-uc_j-id_j)(z_jx_{2n-1}-w_jy_{2n-1})\\
&+\sum^{n-1}_{j=1}(i\ol c_j\beta-u\ol c_j+i\ol d_j)(\ol z_jx_{2n-1}-\ol w_jy_{2n-1})
-\frac{i}{2}\sum^{n-1}_{j=1}\lambda_j(z_j\ol w_j-\ol z_jw_j)\\
&+\frac{i}{2}\sum^{n-1}_{j=1}\lambda_j\abs{z_j-w_j}^2+(x_{2n-1}-y_{2n-1})f(x,y,u)+O(\abs{(x, y)}^3),
\end{split}
\end{equation}
where $\lambda_j=\lambda_j(p)>0$, $j=1,\ldots,n-1$, are the eigenvalues of 
$R^L_p$ with respect to $\langle\,\cdot\,|\,\cdot\,\rangle$.
\end{thm}
%-----------------
The form of $\psi$ should be compared
to the form \cite[Theorems 3.2, 3.4]{HsM:14a} of the phase function for the Szeg\H{o}
kernel on a non-degenerate CR manifold.

\begin{rem}\label{r-gue160817}
The phase function $\Phi(\widehat x,\widehat y,u,t)$ has the following properties: there is a 
\[h(\widehat x,\widehat y,u,t)\in\cC^\infty(\hat D\times\hat D\times\Real\times\Real_+,\Lambda^{0,1}(T^*\hat D))\] such that
\begin{equation}\label{e-gue160629z}
\begin{split}
&\mbox{$\ddbar_s\Phi(\widehat x,\widehat y,u,t)-h(\widehat x,\widehat y,u,t)\Phi(\widehat x,\widehat y,u,t)$ vanishes to infinite order 
at $\widehat x=\widehat y$},\\ 
&{\rm Im\,}\Phi(\widehat x,\widehat y,u,t)\approx t\abs{z-w}^2.
\end{split}
\end{equation}
The phase function $\Phi$ is not unique. Any complex phase function
$\Phi_1(\widehat x,\widehat y,u,t)$ 
satisfying \eqref{e-gue160629z} 
\eqref{e-gue140825f} and \eqref{e-gue140826Im},
is equivalent to $\Phi$ in the sense of Melin-Sj\"ostrand~\cite{MS74}. 
From this observation, given $p\in D$, if we take
local coordinates $x$ and local holomorphic trivializing section $s$,  
$\abs{s}^2_{h^L}=e^{-2\phi}$ such that \eqref{e-gue140826m} holds, 
then near $p$, we can take $\Phi(\widehat x,\widehat y,u,t)$ so that for every $N\in\mathbb N$, 
\begin{equation}\label{e-gue160629zI}
\begin{split}
&\Phi(\widehat x,\widehat y,u,t)=t(x_{2n}-y_{2n})+u(x_{2n-1}
-y_{2n-1})+it(\phi(x)+\phi(y))\\
&\quad-it\Bigr(\sum_{\abs{\alpha}+\abs{\beta}\leq N}\frac{\pr^{\abs{\alpha}+
\abs{\beta}}\phi}{\pr z^\alpha\pr\ol z^\beta}(0,x_{2n-1})
\frac{z^\alpha}{\alpha!}\frac{\ol w^\beta}{\beta!}
+\sum_{\abs{\alpha}+\abs{\beta}\leq N}\frac{\pr^{\abs{\alpha}
+\abs{\beta}}\phi}{\pr z^\alpha\pr\ol z^\beta}(0,y_{2n-1})
\frac{z^\alpha}{\alpha!}\frac{\ol w^\beta}{\beta!}\Bigr)\\
&\quad+O(\abs{z-w}^{N+1}).
\end{split}
\end{equation}
From \eqref{e-gue160629zI}, we have for every $N\in\mathbb N$,  
\begin{equation}\label{e-gue160629zII}
\begin{split}
&\psi(x,y,u)=u(x_{2n-1}-y_{2n-1})+i(\phi(x)+\phi(y))\\
&\quad-i\Bigr(\sum_{\abs{\alpha}+\abs{\beta}\leq N}\frac{\pr^{\abs{\alpha}+
\abs{\beta}}\phi}{\pr z^\alpha\pr\ol z^\beta}(0,x_{2n-1})
\frac{z^\alpha}{\alpha!}\frac{\ol w^\beta}{\beta!}+
\sum_{\abs{\alpha}+\abs{\beta}\leq N}\frac{\pr^{\abs{\alpha}+
\abs{\beta}}\phi}{\pr z^\alpha\pr\ol z^\beta}(0,y_{2n-1})
\frac{z^\alpha}{\alpha!}\frac{\ol w^\beta}{\beta!}\Bigr)\\
&\quad+
O(\abs{z-w}^{N+1}).
\end{split}
\end{equation}
\end{rem}

\subsection{Semi-classical Hodge decomposition for $\Box^{(0)}_{s,k}$} \label{s-gue140826} 

In this section we apply Theorem~\ref{t-dcmimpII} and Theorem~\ref{t-dcgewI} 
to describe the semi-classical Hodge theory for 
$\Box^{(0)}_{s,k}$. In particular we define the approximate
Szeg\H{o} projector $\mathcal{S}_k$ which appears in Theorem \ref{t-gue150129}
and study its kernel.

Let $s$ be a local trivializing section of $L$ on an open subset $D\subset X$ and $\abs{s}^2_{h}=e^{-2\phi}$. 
Let $\chi(x_{2n}), \chi_1(x_{2n})\in \cC^\infty_0(\Real)$, $\chi, \chi_1\geq 0$. 
We assume that $\chi_1=1$ on $\supp\chi$. We take $\chi$ so that  $\int\chi(x_{2n})dx_{2n}=1$. Put
\begin{equation} \label{e-gue13628IIa}
\chi_k(x_{2n})=e^{ikx_{2n}}\chi(x_{2n}).
\end{equation} 
%-----------
We say that a sequence $(g_k)$ in $\mathbb{C}$ is rapidly decreasing and write $g_k=O(k^{-\infty})$
if for every $N>0$, there exists $C_N>0$ independent of $k$ such that for all $k$
we have $\abs{g_k}\leq C_Nk^{-N}$.
%-----------
\begin{prop}\label{p-gue140826a}
Let 
$\Td I=(2\pi)^{-2n}\int e^{i\langle\widehat x-\widehat y,
\widehat\eta\rangle}(1-\alpha(\eta_{2n}))d\widehat\eta$ 
be as in \eqref{e-gue140825b}. Let $\Td I_k$ be the continuous operator 
$\cC^\infty_0(D)\To \cC^\infty(D)$ given by 
\[
\Td I_{k}:\cC^\infty_0(D)\To \cC^\infty(D),\:\:
f\longmapsto \int e^{-ikx_{2n}}\chi_1(x_{2n})\Td I(\chi_kf)(\widehat x)dx_{2n}.
\]
Then, $\Td I_k=(1+g_k)I$ on $\cC^\infty_0(D)$,
where $I$ is the identity map on $\cC^\infty_0(D)$ and $(g_k)$ is a rapidly decreasing sequence.
\end{prop}
%---------------
\begin{proof}
It is easy to see that 
\[
I=(2\pi)^{-2n}\int e^{i\langle\widehat x-\widehat y,\widehat\eta\rangle-
ik(x_{2n}-y_{2n})}\chi_1(x_{2n})\chi(y_{2n})\,d\widehat\eta\, dy_{2n}\,dx_{2n}\:\: \text{on $\cC^\infty_0(D)$}.\]
From this observation, we can check that $\Td I_k=(1+g_k)I$ where
\begin{equation}\label{e-gue140826b}
g_k=-(2\pi)^{-2n}\int e^{i\langle x_{2n}-y_{2n},\eta_{2n}-k\rangle}\alpha(\eta_{2n})
\chi_1(x_{2n})\chi(y_{2n})d\eta_{2n}\,dy_{2n}\,dx_{2n}.
\end{equation}
Since $\alpha(\eta_{2n})=0$ if $\eta\geq1$, we can integrate by parts in \eqref{e-gue140826b} with respect to
$y_{2n}$ several times and conclude that $g_k=O(k^{-\infty})$.  
\end{proof}
%---------------
Let $S\in L^{0}_{\frac{1}{2},\frac{1}{2}}(\widehat D)$ and 
$G\in L^{-1}_{\frac{1}{2},\frac{1}{2}}(\widehat D)$ be as in Theorem~\ref{t-dcmimpII}. 
For $s\in\mathbb N_0$ define
\begin{equation} \label{e-geu13629I}
\mathcal{S}_{k}: H^s_{{\rm comp\,}}(D)\To H^s_{{\rm loc\,}}(D) ,\:\:
f\longmapsto \frac{1}{1+g_k}\int e^{-ikx_{2n}}\chi_1(x_{2n})S(\chi_kf)(\widehat x)dx_{2n}\,,
\end{equation}
\begin{equation} \label{e-gue13630VI}
\mathcal{G}_{k}: H^s_{{\rm loc\,}}(D)\To H^{s+1}_{{\rm loc\,}}(D),\:\: 
f\longmapsto\frac{1}{1+g_k}\int e^{-ikx_{2n}}\chi_1(x_{2n})G(\chi_kf)(\widehat x)dx_{2n}\,.
\end{equation}
The operator $\mathcal{S}_{k}$ is the approximate Szeg\H{o} projector
and $\mathcal{G}_{k}$
is the corresponding Green operator. 
From \eqref{e-geu13629I}, \eqref{e-gue13630VI} and the fact that 
$S:H^s_{{\rm comp\,}}(\widehat D)\To H^{s}_{{\rm loc\,}}(\widehat D)$ is continuous
for every $s\in\Real$, 
$G:H^s_{{\rm comp\,}}(\widehat D)\To H^{s+1}_{{\rm loc\,}}(\widehat D)$ 
is continuous for every $s\in\Real$, it is straightforward to check that
\begin{equation} \label{e-geu13629II}
\begin{split}
&\mathcal{S}_{k}=O(k^s): H^s_{{\rm comp\,}}(D)\To H^{s}_{{\rm loc\,}}(D),\ \ \forall s\in\mathbb N_0,\\
&\mathcal{G}_{k}=O(k^s): H^s_{{\rm comp\,}}(D)\To H^{s+1}_{{\rm loc\,}}(D),\ \ \forall s\in\mathbb N_0.
\end{split}
\end{equation}
%---------------
Repeating the proof of \cite[Theorem 5.4]{Hs14} by 
making use of Proposition~\ref{p-gue140826a} we get the
semiclassical Hodge theory for the localized Kohn laplacian
$\Box^{(0)}_{s,k}$: 
%---------------
\begin{thm}\label{t-gue140826a}
Let $s$ be a local trivializing section of $L$ on an open subset $D\subset X$ 
and $\abs{s}^2_{h}=e^{-2\phi}$. Let $\mathcal{S}_k$ and $\mathcal{G}_k$ be as 
in \eqref{e-geu13629I}, \eqref{e-gue13630VI} respectively. Then, 
\begin{equation} \label{e-gue13630Ia}
\begin{split}
&\mathcal{S}^*_k, \mathcal{S}_{k}=O(k^s): H^s_{{\rm comp\,}}(D)\To H^{s}_{{\rm loc\,}}(D),\ \ \forall s\in\mathbb Z,\\
&\mathcal{G}^*_k, \mathcal{G}_{k}=O(k^s): H^s_{{\rm comp\,}}(D)\To H^{s+1}_{{\rm loc\,}}(D),\ \ \forall s\in\mathbb Z,
\end{split}
\end{equation}
and we have on $D$,
\begin{gather}
\ddbar_{s,k}\mathcal{S}_k\equiv0\mod O(k^{-\infty})\,,\label{e-gue141031b}\\
\Box^{(0)}_{s,k}\mathcal{S}_{k}\equiv 0,\ \ \mathcal{S}_{k}^*\Box^{(0)}_{s,k}\equiv0\mod O(k^{-\infty})\,, 
\label{e-gue13630IIa} \\
\mathcal{S}_{k}+\Box^{(0)}_{s,k}\mathcal{G}_{k}=I\,, \label{e-gue13630IIIa}\\
\mathcal{G}^*_{k}\Box^{(0)}_{s,k}+\mathcal{S}^*_{k}=I\,,\label{e-gue1374IV}
\end{gather}
where $\mathcal{S}_{k}^*$, $\mathcal{G}_{k}^*$ are the formal adjoints of
$\mathcal{S}_{k}$, $\mathcal{G}_{k}$ with respect to $(\,\cdot\,|\,\cdot\,)$
respectively and $\Box^{(0)}_{s,k}$ is given by \eqref{e-gue140824III}.
\end{thm} 
%---------------
We study further the kernel of the approximate Szeg\H{o} projector.
\begin{thm}\label{t-gue150130}
Let $\psi$ be the phase function \eqref{e-gue160630a}. There exists a symbol
\begin{equation}  \label{e-gue13630Vaf}
\begin{split}
&s(x,y,u,k)\in S^{n}_{{\rm loc\,},{\rm cl\,}}(1;D\times D\times\Real), \\
&s(x,y,u,k)\sim\sum^\infty_{j=0}s_j(x,y,u)k^{n-j}\text{ in }S^{n}_{{\rm loc\,}}
(1;D\times D\times\Real), \\
%&s_j(x,y,u)\in \cC^\infty(D\times D\times\Real),\ \ j=0,1,2,\ldots,
\end{split}
\end{equation}
such that the operator $S_k$ with kernel
\[
S_k(x,y)=\int_{\mathbb{R}} e^{ik\psi(x,y,u)}s(x,y,u,k)du,
\]
satisfies
\begin{equation} \label{e-gue13630IVaf}
\mathcal{S}_{k}(x,y)-S_k(x,y)=
O(k^{-\infty}):H^s_{{\rm comp\,}}(D)\To H^s_{{\rm loc\,}}(D),\  \forall s\in\mathbb Z.
\end{equation}
\end{thm}

\begin{proof}
From the definition \eqref{e-geu13629I} of
$\mathcal{S}_{k}$ and Theorem~\ref{t-dcgewI}, we see that the distribution kernel of $\mathcal{S}_{k}$ is given by
\begin{equation} \label{e-gue140905}
\begin{split}
&\mathcal{S}_{k}(x,y)\\
&\equiv\int\limits_{t\in\Real_+}e^{i\Phi(\widehat x,\widehat y,u,t)-ikx_{2n}+iky_{2n}}
b(\widehat x,\widehat y,u,t)\chi_1(x_{2n})\chi(y_{2n})(1-\alpha(t))dx_{2n}dtdy_{2n}du  \\
&\equiv\int\limits_{\substack{u\in\Real\\\sigma\in\Real_+}}\!\!e^{ik\varrho(x,y,u,\sigma)}
k^2\sigma b(\widehat x,\widehat y,k\sigma u,k\sigma)
\chi_1(x_{2n})\chi(y_{2n})(1-\alpha(k\sigma))dx_{2n}d\sigma dy_{2n}du,
\end{split}
\end{equation}
$\operatorname{mod}O(k^{-\infty})$, where
\[\varrho(x,y,u,\sigma)=\sigma\psi(x,y,u)+(x_{2n}-y_{2n})(\sigma-1),\]
and the integrals above are defined as oscillatory integrals. Let $\gamma(\sigma)\in \cC^\infty_0(\Real_+)$ with
$\gamma(\sigma)=1$ in some small neighbourhood of $1$. Denote by $I_0(x,y)$
the integral
\begin{equation*}\label{e-gue13630VIIaf}
\begin{split}
\int\limits_{\sigma\geq0}&e^{ik\varrho(x,y,u,\sigma)}\gamma(\sigma)k^2
\sigma b(\widehat x,\widehat y,k\sigma u,k\sigma)(1-\alpha(k\sigma))
\chi_1(x_{2n})\chi(y_{2n})dx_{2n}d\sigma dy_{2n}du,
\end{split}
\end{equation*}
and by $I_1(x,y)$ the integral
\begin{equation*}\label{e-gue13630VIIIaf}
\begin{split}
\int\limits_{\sigma\geq0}&e^{ik\varrho(x,y,u,\sigma)}(1-\gamma(\sigma))
k^2\sigma b(\widehat x,\widehat y,k\sigma u,k\sigma)(1-\alpha(k\sigma))\chi_1(x_{2n})\chi(y_{2n})dx_{2n}d\sigma dy_{2n}du.
\end{split}
\end{equation*}
Then, 
\begin{equation}\label{e-gue140909f}
\mathcal{S}_{k}(x,y)\equiv I_0(x,y)+I_1(x,y)\mod O(k^{-\infty}).
\end{equation}
First, we study $I_1(x,y)$. Note that when $\sigma\neq1$,
$d_{y_{2n}}\bigr(\sigma\psi(x,y,u)+(x_{2n}-y_{2n})(\sigma-1)\bigr)=1-\sigma\neq0$. Thus, we can integrate by
parts in $y_{2n}$ several times and get that 
\begin{equation}\label{e-gue1373ff}
I_1=O(k^{-\infty}):H^s_{{\rm comp\,}}(D)\To H^s_{{\rm loc\,}}(D),\ \ \forall s\in\mathbb Z.
\end{equation}
Next, we study the kernel $I_0(x,y)$. We may assume that $b(\widehat x,\widehat y,k\sigma u,k\sigma)$ 
is supported in some small neighbourhood of $\widehat x=\widehat y$. We want to apply the stationary 
phase method of Melin and Sj\"{o}strand~\cite[p.\,148]{MS74} to carry out the $dx_{2n}d\sigma$ 
integration in \eqref{e-gue13630VIIaf}. Put 
\[\Theta(\widehat x,\widehat y,\sigma):=\sigma\psi(x,y,u)+(x_{2n}-y_{2n})(\sigma-1).\]
We first notice that $d_\sigma\Theta(\widehat x,\widehat y,\sigma)|_{\widehat x=\widehat y}=0$ and
$d_{x_{2n}}\Theta(\widehat x,\widehat y,\sigma)|_{\sigma=1}=0$.
Thus, $x=y$ and $\sigma=1$ are real critical points. Furthermore, we have 
\[\begin{split}\Theta''_{\sigma\sigma}(\widehat x,\widehat x,1)=0,\ \  
\Theta''_{x_{2n}\sigma}(\widehat x,\widehat x,1)=1,\ \ 
\Theta''_{\sigma x_{2n}}(\widehat x,\widehat x,1)=1,\ \ 
\Theta''_{x_{2n}x_{2n}}(\widehat x,\widehat x,1)=0.\end{split}\]
Thus, the Hessian of $\Theta(\widehat x,\widehat y,\sigma)$ 
with respect to $(\sigma, x_{2n})$ at $\widehat x=\widehat y$, $\sigma=1$,
is given by
\[\left(
\begin{array}[c]{cc}
  \Theta''_{\sigma\sigma}(\widehat x,\widehat x,1)& \Theta''_{x_{2n}\sigma}(\widehat x,\widehat x,1) \\
  \Theta''_{\sigma x_{2n}}(\widehat x,\widehat x,1) & \Theta''_{x_{2n}x_{2n}}(\widehat x,\widehat x,1)
\end{array}\right)=\left(
\begin{array}[c]{cc}
 0 & 1 \\
 1 &0
\end{array}\right).\]
Thus, $\Theta(\widehat x,\widehat y,\sigma)$ is a non-degenerate complex valued phase function
in the sense of Melin-Sj\"{o}strand~\cite{MS74}. Let
$\Td\Theta(\Td{\widehat x},\Td{\widehat y},\Td\sigma):=
\Td\psi(\Td x,\Td y,u)\Td\sigma+(\Td x_{2n}-\Td y_{2n})(\Td\sigma-1)$
be an almost analytic extension of $\Theta(\widehat x,\widehat y,\widehat\sigma)$, where
$\Td\psi(\Td x,\Td y,u)$ is an almost analytic extension of $\psi(x,y,u)$. Here we fix $u$. We can
check that given $y_{2n} $ and $(x,y)$, $\Td x_{2n}=y_{2n}-\psi(x,y,u)$, $\Td\sigma=1$ are the solutions of 
$\frac{\pr\Td\Theta}{\pr\Td\sigma}=0$, $\frac{\pr\Td\Theta}{\pr\Td x_{2n} }=0$.
From this and by the stationary phase formula of Melin-Sj\"{o}strand~\cite{MS74}, we
get 
\begin{equation} \label{e-gue1373VIIIf}
I_0(x,y)-\int e^{ik\psi(x,y,u)}s(x,y,u,k)du
=O(k^{-\infty}):H^s_{{\rm comp\,}}(D)\To H^s_{{\rm loc\,}}(D),\ \ \forall s\in\mathbb Z,
\end{equation}
where $s(x,y,u,k)\in S^{n}_{{\rm loc\,},{\rm cl\,}}(1,D\times D\times\Real)$, 
\[s(x,y,u,k)\sim\sum\limits^\infty_{j=0}s_j(x,y,u)k^{n-j}\:\:
\text{in $S^{n}_{{\rm loc\,}}(1,D\times D\times\Real$)},
\] 
with $s_j(x,y,u)\in \cC^\infty(D\times D\times\Real)$, $j=0,1,2,\ldots$\,.
From \eqref{e-gue1373ff}, \eqref{e-gue1373VIIIf} and \eqref{e-gue140909f}, the theorem follows. 
\end{proof}
We show now that the composition of $\mathcal{S}_k$ with a
classical semi-classical pseudodifferential has an asymptotic
expansion and calculate the leading term.
From Theorem~\ref{t-gue150130} and the stationary phase method of 
Melin and Sj\"{o}strand, we deduce:

\begin{thm}\label{t-gue140826aI}
Let $\mathcal{A}_k$ be a properly supported classical semi-classical pseudodifferential 
operator on $D$ of order $0$ as in \eqref{ps-def} and \eqref{ps-symb} with symbol 
$\beta\in S^{0}_{{\rm loc\,},{\rm cl\,}}(1;T^*D)$ such that
$\beta(x,\eta,k)=0$ if $\abs{\eta}\geq\frac12 M$, 
for some large $M>0$. We have 
\begin{equation} \label{e-gue13630IVa}
(\mathcal{S}_{k}\circ\mathcal{A}_k)(x,y)\equiv\int e^{ik\psi(x,y,u)}a(x,y,u,k)du\mod O(k^{-\infty}),
\end{equation}
where
\begin{equation}  \label{e-gue13630Va}
\begin{split}
&a(x,y,u,k)\in \cC^\infty_0(D\times D\times(-M,M))\cap  
S^{n}_{{\rm loc\,},{\rm cl\,}}(1;D\times D\times(-M,M)),\\
&a(x,y,u,k)\sim\sum^\infty_{j=0}a_j(x,y,u)k^{n-j}\text{ in }S^{n}_{{\rm loc\,}}
(1;D\times D\times-(M,M)), \\
&a_j(x,y,u)\in \cC^\infty_0(D\times D\times(-M,M)),\ \ j=0,1,2,\ldots,
\end{split}
\end{equation}
and $\psi(x,y,u)=\varphi(x,y,u,1)$, $\varphi(x,y,u,t)$ is as in Theorem~\ref{t-dcgewI}.
\end{thm}

Recall that $\mathcal{A}_k$ is called properly supported if
the restrictions of the projections $(x,y)\mapsto x$ and $(x,y)\mapsto y$
to $\supp\mathcal{A}_k(\cdot,\cdot)\subset X\times X$ are proper.
Let \[\mathcal{A}_k\equiv\frac{k^{2n-1}}{(2\pi)^{2n-1}}
\int e^{ik\langle x-y,\eta\rangle}\beta(x,\eta,k)d\eta\mod O(k^{-\infty})
\] 
be as in Theorem~\ref{t-gue140826aI}. Put 
\begin{equation}\label{e-gue140909I}
\begin{split}
\beta(x,\eta,k)\sim\sum^{\infty}_{j=0}\beta_j(x,\eta)k^{-j},\:\:
\beta_j(x,\eta)\in \cC^\infty(T^*D),\ \ j=0,1,2,\ldots.
\end{split}
\end{equation}
From the last formula of \eqref{e-dcgewII}, it is straightforward to see that 
\begin{equation}\label{e-gue140909II}
a_0(x,x,u)\neq0\ \ \mbox{if $\beta_0(x,u\omega_0(x)-2{\rm Im\,}\ddbar_b\phi(x))\neq0$},
\end{equation}
where $a_0(x,y,u)$ is as in \eqref{e-gue13630Va}. In the rest of this section, we will calculate $a_0(x,x,u)$. 

Fix $D_0\Subset D$ and let $\chi, \widehat\chi\in \cC^\infty_0(D,[0,1])$, $\chi=\widehat\chi=1$ on $D_0$ and $\chi=1$ on some neighbourhood of $\supp\widehat\chi$.

\begin{lem}\label{l-gue141029a}
With the notations above, we have 
\begin{equation}\label{e-gue141029b}
(\widehat\chi\mathcal{A}^*_k\mathcal{S}^*_k\chi)(\chi\mathcal{S}_k\mathcal{A}_k\widehat\chi)\equiv\widehat\chi\mathcal{A}^*_k\mathcal{S}_k\mathcal{A}_k\widehat\chi\mod O(k^{-\infty}),
\end{equation}
where $\mathcal{A}^*_k$ is the formal adjoint of $\mathcal{A}_k$.
\end{lem} 

\begin{proof}
From \eqref{e-gue1374IV}, we have 
\begin{equation}\label{e-gue141029bI}
\widehat\chi\mathcal{A}^*_k\mathcal{G}^*_k\Box^{(0)}_{s,k}\chi+\widehat\chi\mathcal{A}^*_k\mathcal{S}^*_k\chi
=\widehat\chi\mathcal{A}^*_k\chi.
\end{equation}
From \eqref{e-gue141029bI}, we have 
\begin{equation}\label{e-gue141029bII}
\widehat\chi\mathcal{A}^*_k\mathcal{G}^*_k\Box^{(0)}_{s,k}\chi^2\mathcal{S}_k\mathcal{A}_k\widehat\chi+
\widehat\chi\mathcal{A}^*_k\mathcal{S}^*_k\chi^2\mathcal{S}_k\mathcal{A}_k\widehat\chi
=\widehat\chi\mathcal{A}^*_k\chi^2\mathcal{S}_k\mathcal{A}_k\widehat\chi.
\end{equation}
From \eqref{e-gue13630IVa}, it is not difficult to check that $\mathcal{S}_k\mathcal{A}_k$ is $k$-negligible away the diagonal. From this observation, \eqref{e-gue13630Ia} and \eqref{e-gue13630IIa}, we conclude that 
\begin{equation}\label{e-gue141029bIII}
\widehat\chi\mathcal{A}^*_k\mathcal{G}^*_k\Box^{(0)}_{s,k}\chi^2\mathcal{S}_k\mathcal{A}_k\widehat\chi\equiv0\mod O(k^{-\infty}).
\end{equation}
From \eqref{e-gue141029bIII} and \eqref{e-gue141029bII}, we get 
\begin{equation}\label{e-gue141029bIV}
\widehat\chi\mathcal{A}^*_k\mathcal{S}^*_k\chi^2\mathcal{S}_k\mathcal{A}_k\widehat\chi
\equiv\widehat\chi\mathcal{A}^*_k\chi^2\mathcal{S}_k\mathcal{A}_k\widehat\chi\mod O(k^{-\infty}).
\end{equation}
Again, since $\mathcal{S}_k\mathcal{A}_k$ is $k$-negligible away the diagonal, we deduce that 
\begin{equation}\label{e-gue141029bV}
\widehat\chi\mathcal{A}^*_k\chi^2\mathcal{S}_k\mathcal{A}_k\widehat\chi
\equiv\widehat\chi\mathcal{A}^*_k\mathcal{S}_k\mathcal{A}_k\widehat\chi\mod O(k^{-\infty}).
\end{equation}
From \eqref{e-gue141029bIV} and \eqref{e-gue141029bV}, we get \eqref{e-gue141029b}.
\end{proof}

\noindent
From \eqref{e-gue141029b}, \eqref{e-gue13630IVa} and the complex stationary phase formula 
of Melin-Sj\"ostrand~\cite{MS74}, we have $\operatorname{mod}O(k^{-\infty})$,
\begin{equation}\label{e-gue140909IV}
\bigr((\widehat\chi\mathcal{A}^*_k\mathcal{S}^*_k\chi)(\chi\mathcal{S}_k\mathcal{A}_k\widehat\chi)\bigr)(x,y)\equiv(\widehat\chi\mathcal{A}^*_k\mathcal{S}_k\mathcal{A}_k\widehat\chi)(x,y)\equiv\int e^{ik\psi(x,y,u)}g(x,y,u,k)du,
\end{equation}
where
\begin{equation}  \label{e-gue140909V}
\begin{split}
&g(x,y,u,k)\in \cC^\infty_0(D\times D\times(-M,M))\cap  S^{n}_{{\rm loc\,},{\rm cl\,}}(1;D\times D\times(-M,M)),\\
&g(x,y,u,k)\sim\sum^\infty_{j=0}g_j(x,y,u)k^{n-j}\text{ in }S^{n}_{{\rm loc\,}}
(1;D\times D\times\Real), \\
&g_j(x,y,u)\in \cC^\infty_0(D\times D\times(-M,M)),\ \ j=0,1,2,\ldots,
\end{split}
\end{equation}
and 
\begin{equation}\label{e-gue140909VI}
g_0(x,x,u)=a_0(x,x,u)\ol{\beta_0}(x,u\omega_0(x)-2{\rm Im\,}\ddbar_b\phi(x)),\ \ \forall (x,x,u)\in D_0\times D_0\times(-M,M).
\end{equation}
On the other hand, we can repeat the procedure of Section 5 in~\cite{Hs14} (see the discussion after Theorem 5.6 in~\cite{Hs14}) and deduce that
\begin{equation}\label{e-gue13712I}
\bigr((\widehat\chi\mathcal{A}^*_k\mathcal{S}^*_{k}\chi)(\chi\mathcal{S}_k\mathcal{A}_k\widehat\chi)\bigr)(x,y)\equiv\int e^{ik\psi_1(x,y,u)}h(x,y,u,k)du\mod
O(k^{-\infty})
\end{equation}
with
\begin{equation}  \label{e-gue13712II}
\begin{split}
&h(x,y,u,k)\in S^{n}_{{\rm loc\,},{\rm cl\,}}(1,D\times D\times(-M,M))\cap  \cC^\infty_0(D\times D\times(-M,M)),\\
&h(x,y,u,k)\sim\sum^\infty_{j=0}h_j(x,y,u)k^{n-j}\text{ in }S^{n}_{{\rm loc\,}}
(1,D\times D\times(-M,M)), \\
&h_j(x,y,u)\in \cC^\infty_0(D\times D\times(-M,M)),\ \ j=0,1,2,\ldots,
\end{split}
\end{equation}
\begin{equation}  \label{e-gue140909VII}
\begin{split}
&h_0(x,x,u)=2\pi^{n}\abs{\det R^L_x}^{-1}\abs{a_0(x,x,u)}^2,\ \ \forall (x,x,u)\in D_0\times D_0\times(-M,M),\\
&g_0(x,x,u)=h_0(x,x,u),\ \ \forall (x,x,u)\in D\times D\times(-M,M),
\end{split}
\end{equation}
and for all $(x,x,u)\in D\times D\times(-M,M)$, we have 
\begin{equation} \label{e-gue13712III}
\begin{split}
&\psi_1(x, x,u)=0,\ \ d_x\psi_1(x, x,u)=d_x\psi(x,x,u),\ \ d_y\psi_1(x, x,u)=d_y\psi(x, x,u),\\
&{\rm Im\,}\psi_1(x,y,u)\geq0,\ \ \forall (x,y,u)\in D\times D\times(-M,M).
\end{split}
\end{equation}
%------------
From \eqref{e-gue140909VII} and \eqref{e-gue140909VI}, we get for all $(x,x,u)\in D_0\times D_0\times(-M,M)$, 
\begin{equation}  \label{e-gue140909VIII}
a_0(x,x,u)\ol{\beta_0}(x,u\omega_0(x)-2{\rm Im\,}\ddbar_b\phi(x))=2\pi^{n}\abs{\det R^L_x}^{-1}\abs{a_0(x,x,u)}^2.
\end{equation}
If the quantity $\ol{\beta_0}(x,u\omega_0(x)-2{\rm Im\,}\ddbar_b\phi(x))=0$, 
we get $a_0(x,x,u)=0$. 
If this quantity doesn't vanish,  
in view of \eqref{e-gue140909II}, we know that $a_0(x,x,u)\neq0$. 
From this observation and \eqref{e-gue140909VIII}, we obtain:
%-------------------
\begin{thm}\label{t-gue140909}
For $a_0(x,y,u)$ in \eqref{e-gue13630Va}, 
\[a_0(x,x,u)=\frac{1}{2}\pi^{-n}\abs{\det R^L_x}\beta_0\big(x,u\omega_0(x)-
2{\rm Im\,}\ddbar_b\phi(x)\big),
\ \  (x,x,u)\in D\times D\times(-M,M),\]
where $\beta_0(x,\eta)\in \cC^\infty(T^*D)$ is as in \eqref{e-gue140909I} and 
$\det R^L_x$ as in \eqref{detrl}.
\end{thm}

\begin{rem}\label{r-gue160705}
It should be noticed that by using the complex stationary phase 
formula of Melin-Sj\"{o}strand and the method in \cite{Hs12}, 
we can write down a general recurrsion relation for the symbols 
$a_j(x,y,u)$ and $\beta_j(x,\eta)$. 
We only calculate the leading term $a_0(x,x,u)$ in this paper.
\end{rem}

\section{Regularity of the Szeg\H{o} projection $\Pi_k$}\label{s-gue140909}

In this section, we will prove Theorem~\ref{t-mainb}.
For this purpose we first establish the spectral gap for the Kohn Laplacian
$\Box^{(1)}_{b,k}$ and 
then Sobolev estimates for the associated Green operator and finally for
$\Pi_k$. 

We start with a local form of the spectral gap estimate for $(0,1)$-forms.
%------------
\begin{lem}\label{l-gue140909a}
Let $s$ be a local trivializing section of $L$ on an open set $D\subset X$. 
Then, there is a constant $C>0$ independent of $k$ such that 
\[
\norm{\ddbar_{b,k}u}^2_{k}+\|\ol{\pr}^*_{b,k}u\|^2_{k}
\geq\Big(Ck-\frac{1}{C}\Big)\norm{u}^2_{k},\:\:
\text{for all $u\in\Omega^{0,1}_0(D,L^k)$}.
\]
\end{lem}

\begin{proof}
Let $u\in\Omega^{0,1}_0(D,L^k)$. Put $u=s^k\widehat u$, 
$\widehat u\in\Omega^{0,1}_0(D)$. In view of \eqref{e-gue140824III}, we have 
\begin{equation}\label{e-gue140909a}
\Box^{(1)}_{b,k}u=e^{k\phi}s^k\Box^{(1)}_{s,k}(e^{-k\phi}\widehat u).
\end{equation}
Put $\widehat u=\sum\limits^{n-1}_{j=1}\widehat u_je_j$, where $e_1,\ldots,e_{n-1}\in \Lambda^{0,1}(T^*X)$ is as in Proposition~\ref{p-gue140824}. From \eqref{e-gue140824VII}, we have 
\begin{equation}\label{e-gue140909aI}
\begin{split}
&(\Box^{(1)}_{s,k}(e^{-k\phi}\widehat u)\,|\,e^{-k\phi}\widehat u\,)
=\sum^{n-1}_{j=1}\norm{(Z_j+kZ_j(\phi))(e^{-k\phi}\widehat u)}^2\\
&\quad+
\sum^{n-1}_{j,t=1}(\,[Z_j+kZ_j(\phi),-\ol Z_t+
k\ol Z_t(\phi)](e^{-k\phi}\widehat u_t)\,|\,e^{-k\phi}\widehat u_j\,)\\
&\quad+(\,(\varepsilon(Z+kZ(\phi))+
\varepsilon(Z^*+k\ol Z(\phi)))(e^{-k\phi}\widehat u)\,|\,e^{-k\phi}\widehat u\,)+
(\,fe^{-k\phi}\widehat u\,|\,e^{-k\phi}\widehat u\,).
\end{split}
\end{equation}
Here we use the same notations as in Proposition~\ref{p-gue140824}. Fix $j, t=1,2,\ldots,n-1$. Put 
\[
[Z_j-\ol Z_t]=\sum^{n-1}_{s=1}(a^{j,t}_sZ_s-b^{j,t}_s\ol Z_s)\,,\quad 
a^{j,t}_s, b^{j,t}_s\in \cC^\infty(D).
\]
%, $a^{j,t}_s, b^{j,t}_s\in \cC^\infty(D), \forall s$.
Recall than by \cite[Lemma 4.1]{HsM:12}, for any $U, V\in T^{1,0}_pX$ and any 
$\mU, \mV\in C^\infty(D,T^{1, 0}X)$ that satisfy $\mU(p)=U$,
$\mV(p)=V$, we have
\begin{equation} \label{s5-e2}
R^L_p(U,V)=M^\phi_p(U,V)=-\big\langle\big[\,\mU, \ol{\mV}\,\big](p), \ddbar_b\phi(p)-\pr_b\phi(p)\big\rangle
+\big(\mU\ol{\mV}+\ol{\mV}\mU\big)\phi(p).
\end{equation}
By using \eqref{s5-e2} we obtain
\begin{equation}\label{e-gue140909aII}
\begin{split}
&[Z_j+kZ_j(\phi),-\ol Z_t+k\ol Z_t(\phi)]=
\sum^{n-1}_{s=1}(a^{j,t}_sZ_s-b^{j,t}_s\ol Z_s)+k(Z_j\ol Z_t+\ol Z_tZ_j)(\phi)\\
&=\sum^{n-1}_{s=1}(a^{j,t}_s(Z_s+kZ_s(\phi))+
b^{j,t}_s(-\ol Z_s+k\ol Z_s(\phi)))-k\langle\,[Z_j-\ol Z_t]\,,\,\ddbar_b\phi-
\pr_b\phi\,\rangle\\&\qquad\qquad+k(Z_j\ol Z_t+\ol Z_tZ_j)(\phi)\\
&=\varepsilon(Z+kZ(\phi))+\varepsilon(-\ol Z+k\ol Z(\phi))+kR^L_x(\ol Z_t,Z_j).
\end{split}
\end{equation}
From \eqref{e-gue140909aII} and \eqref{e-gue140909aI}, we get 
\begin{equation}\label{e-gue140909aIII}
\begin{split}
&\big(\Box^{(1)}_{s,k}(e^{-k\phi}\widehat u)\,|\,e^{-k\phi}\widehat u\,\big)
=\sum^{n-1}_{j=1}\norm{(Z_j+kZ_j(\phi))(e^{-k\phi}\widehat u)}^2\\
&\quad+k\sum^{n-1}_{j,t=1}\big(\,R^L_x(\ol Z_t,Z_j)(e^{-k\phi}\widehat u_t)\,|\,e^{-k\phi}
\widehat u_j\,\big)\\
&\quad+(\,(\varepsilon(Z+kZ(\phi))+\varepsilon(Z^*+k\ol Z(\phi)))(e^{-k\phi}\widehat u)\,|\,e^{-k\phi}\widehat u\,)+(\,\Td fe^{-k\phi}\widehat u\,|\,e^{-k\phi}\widehat u\,),
\end{split}
\end{equation}
where $\Td f$ is a smooth function independent of $k$. Since $R^L>0$, from \eqref{e-gue140909aIII}, 
it is not difficult to see that 
\begin{equation}\label{e-gue140909b}
\Bigr(\Box^{(1)}_{s,k}(e^{-k\phi}\widehat u)\,|\,e^{-k\phi}\widehat u\Bigl)\geq
\Bigl(\Td Ck-\frac{1}{\Td C}\Bigr)\norm{e^{-k\phi}\widehat u}^2,
\end{equation} 
where $\Td C>0$ is a constant independent of $k$ and $u$. From \eqref{e-gue140909a}, we can check that 
\[\big(\,\Box^{(1)}_{s,k}(e^{-k\phi}\widehat u)\,|\,e^{-k\phi}\widehat u\,\big)=
\big(\,\Box^{(1)}_{b,k}u\,|\,u\,\big)_{k}=\|\ddbar_{b,k}u\|^2_{k}+
\|\ol{\pr}^*_{b,k}\widehat u\|^2_{k}.\]
Moreover, it is clearly that $\norm{u}_{k}=\|e^{-k\phi}\widehat u\|$. 
From this observation and \eqref{e-gue140909b}, the lemma follows.
\end{proof}
%---------------
Ohsawa and Sibony \cite{OS00} established analogues of the Nakano and 
Akizuki vanishing theorems for Levi flat CR manifolds.
The following result can be seen as an analogue of the spectral gap 
and Kodaira-Serre vanishing theorem \cite[Theorems\,1.5.5-6]{MM07}.
%---------------
\begin{thm}\label{t-gue140910}
There is a constant $C_0>0$ independent of $k$ such that 
\[\norm{\ddbar_{b,k}u}^2_{k}+\|\ol{\pr}^*_{b,k}u\|^2_{k}\geq
\Big(C_0k-\frac{1}{C_0}\Big)\norm{u}^2_{k},\ \ 
\forall u\in{\rm Dom\,}\ddbar_{b,k}\cap {\rm Dom\,}\ol{\pr}^*_{b,k}\subset L^2_{(0,1)}(X,L^k).\]
Hence, for $k$ large, ${\rm Ker\,}\Box^{(1)}_{b,k}=\set{0}$ and $\Box^{(1)}_{b,k}$ has $L^2$ closed range.
\end{thm}
From Theorem~\ref{t-gue140910}, we deduce that $\Box^{(1)}_{b,k}$ is injective for large $k$
so we can consider the Green operator $N^{(1)}_k: L^2_{(0,1)}(X,L^k)\To{\rm Dom\,}\Box^{(1)}_{b,k}$, 
which is the inverse of $\Box^{(1)}_{b,k}$. We have 
\begin{equation}\label{e-gue141029}
\begin{split}
\mbox{$\Box^{(1)}_{b,k}N^{(1)}_k=I$ on $L^2_{(0,1)}(X)$},\ \ \mbox{$N^{(1)}_k\Box^{(1)}_{b,k}=I$ on ${\rm Dom\,}\Box^{(1)}_{b,k}$}.
\end{split}
\end{equation}
\begin{proof}
We first claim that there is a constant $C_0>0$ independent of $k$ such that 
\begin{equation}\label{e-gue140910}
\norm{\ddbar_{b,k}u}^2_{k}+\|\ol{\pr}^*_{b,k}u\|^2_{k}\geq 
\Big(C_0k-\frac{1}{C_0}\Big)\norm{u}^2_{k},\ \ \forall u\in\Omega^{0,1}(X,L^k).
\end{equation}
Let $X=\bigcup^N_{j=1}D_j$, where $D_j\subset X$ is an open set with $L|_{D_j}$ is trivial. 
Take $\chi_j\in \cC^\infty_0(D_j,[0,1])$, $j=1,\ldots,N$, with $\sum\limits^N_{j=1}\chi_j=1$ on $X$. 
Let $u\in\Omega^{0,1}(D,L^k)$. From Lemma~\ref{l-gue140909a}, we see that for every $j=1,2,\ldots,N$, 
we can find a constant $C_j>0$ independent of $k$ and $u$ such that
\begin{equation}\label{e-gue140910I}
\norm{\ddbar_{b,k}(\chi_ju)}^2_{k}+\|\ol{\pr}^*_{b,k}(\chi_ju)\|^2_{k}\geq
\Big(C_jk-\frac{1}{C_j}\Big)\norm{\chi_ju}^2_{k}.
\end{equation}
It is easy to see that 
\begin{equation}\label{e-gue140910II}
\begin{split}
\norm{\ddbar_{b,k}(\chi_ju)}^2_{k}+\|\ol{\pr}^*_{b,k}(\chi_ju)\|^2_{k}&\leq 
\norm{\chi_j\ddbar_{b,k}u}^2_{k}+\|\chi_j\ol{\pr}^*_{b,k}u\|^2_{k}+M_j\norm{u}^2_{k}\\
&\leq\norm{\ddbar_{b,k}u}^2_{k}+\|\ol{\pr}^*_{b,k}u\|^2_{k}+M_j\norm{u}^2_{k},
\end{split}
\end{equation}
where $M_j>0$ is a constant independent of $k$ and $u$. From \eqref{e-gue140910II} and \eqref{e-gue140910I}, we get 
\begin{equation}\label{e-gue140911}
\begin{split}
N\Bigl(\norm{\ddbar_{b,k}u}^2_{k}+\|\ol{\pr}^*_{b,k}u\|^2_{k}\Bigr)\geq\sum^N_{j=1}\Bigr(\Big(C_jk-\frac{1}{C_j}\Big)\norm{\chi_ju}^2_{k}-M_j\norm{u}^2_{k}\Bigr)
\geq \Bigl(ck-\frac{1}{c}\Bigl)\norm{u}^2_{k},
\end{split}
\end{equation}
where $c>0$ is a constant independent of $k$. From \eqref{e-gue140911}, the claim \eqref{e-gue140910} follows.

Now, let $u\in{\rm Dom\,}\ddbar_{b,k}\cap {\rm Dom\,}\ol{\pr}^*_{b,k}$. 
From Friedrichs' Lemma (see Appendix D in~\cite{CS:01}), 
we can find $u_j\in\Omega^{0,1}(X,L^k)$, $j=1,2,\ldots$, with $u_j\To u$ in $L^2_{(0,1)}(X,L^k)$, 
$\ddbar_{b,k}u_j\To\ddbar_{b,k}u$ in $L^2_{(0,2)}(X,L^k)$ and $\ol{\pr}^*_{b,k}u_j\To\ol{\pr}^*_{b,k}u$ in $L^2(X,L^k)$.
From \eqref{e-gue140910}, we have
\[\begin{split}
\norm{\ddbar_{b,k}u}^2_{k}+\|\ol{\pr}^*_{b,k}u\|^2_{k}=
\lim_{j\To\infty}\Bigl(\norm{\ddbar_{b,k}u_j}^2_{k}+\|\ol{\pr}^*_{b,k}u_j\|^2_{k}\Bigr)&\geq \Bigl(C_0k-\frac{1}{C_0}\Bigr)\lim_{j\To\infty}\norm{u_j}^2_{k}\\
&=\Bigl(C_0k-\frac{1}{C_0}\Bigr)\norm{u}^2_{k}.
\end{split}\]
The theorem follows. 
\end{proof}

We pause and introduce some notations. Let $s$ be a local trivializing section of $L$ 
on an open set $D\subset X$, $\abs{s}^2_{h}=e^{-2\phi}$. 
Let $u\in\Omega^{0,q}_0(D,L^k)$. On $D$, we write $u=s^k\Td u$, $\Td u\in\Omega^{0,q}_0(D)$. 
For every $m\in\mathbb N_0$, define 
\[
\norm{u}^2_{m,k}:=\sum\limits_{\abs{\alpha}\leq 
m,\alpha\in\mathbb N^{2n-1}_0}\int\abs{\pr^\alpha_x(\Td ue^{-k\phi})}^2dv_X.
\] 
By using a partition of unity, we can define $\norm{u}^2_{m,k}$ for all $u\in\Omega^{0,q}(X,L^k)$ 
in the standard way. We call $\norm{\cdot}_{m,k}$ the Sobolev norm of order $m$ with respect to $h^{k}$. 
We will need the following.

\begin{prop}[{\cite[Proposition\,1]{OS00}}]\label{p-gue141029}
For every $m\in\mathbb N_0$ there is $N_m>0$ such that for every $k\geq N_m$, 
\begin{equation}\label{e-gue141029I}
\|\ol{\pr}^*_{b,k}u\|_{m,k}\leq k^{M(m)}\|\Box^{(1)}_{b,k}u\|_{m,k},\ \ 
u\in\Omega^{0,1}(X,L^k),
\end{equation}
where $M(m)>0$ is a constant independent of $k$ and $u$. 
\end{prop}

\begin{thm}\label{t-gue141029}
For every $m\in\mathbb N$, there exist $N_m>0$ and 
$M(m)>0$ 
such that for every $k\geq N_m$\,, 
\begin{equation}
\begin{split}
&\ol{\pr}^*_{b,k}N^{(1)}_k:\Omega^{0,1}(X,L^k)\To H^m(X,L^k),\\
&\|\ol{\pr}^*_{b,k}N^{(1)}_ku\|_{m,k}\leq k^{M(m)}\norm{u}_{m,k},\:\: 
u\in\Omega^{0,1}(X,L^k).
\end{split}
\end{equation}
\end{thm}

\begin{proof}
The theorem essentially follows from Proposition~\ref{p-gue141029} and the elliptic regularization method 
introduced by Kohn-Nirenberg~\cite[p.\,102]{CS:01}, \cite[p.\,449]{KN65}. 
Namely, for every $\varepsilon>0$, consider the operator 
$\Box^{(1)}_{\varepsilon,k}:=\Box^{(1)}_{b,k}+\varepsilon T^*T$, where $T$ is defined in \eqref{e-suI}
and $T^*$ is its formal adjoint with respect to $(\,\cdot\,|\,\cdot\,)_{k}$.
%We only give the outline of the proof. 
Fix $m\in\mathbb N$. From Theorem~\ref{t-gue140910} and Proposition~\ref{p-gue141029}, 
there is a $N_m>0$ such that for every $k\geq N_m$, 
\begin{equation}\label{e-gue141029II}
\begin{split}
&\norm{u}^2_{k}\leq(\,\Box^{(1)}_{b,k}u\,|\,u\,)_{k},\ \ \forall u\in\Omega^{0,1}(X,L^k),\\
&\norm{u}_{\ell,k}\leq k^{M(m)}\|\Box^{(1)}_{b,k}u\|_{\ell,k},\ \ 
\forall u\in\Omega^{0,1}(X,L^k),\ \ \forall\ell\in\mathbb N_0,\ \ \ell\leq m,
\end{split}
\end{equation}
where $M(m)>0$ is a constant independent of $k$ and $u$. 

Take $g\in\Omega^{0,1}(X,L^k)$ and put $N^{(1)}_kg=v$. We have $\Box^{(1)}_{b,k}v=g$. 
 From \eqref{e-gue141029II}, 
it is easy to see that for every $k\geq N_m$ and every $\varepsilon>0$, 
$\Box^{(1)}_{\varepsilon,k}$ is injective and has range $L^2_{(0,1)}(X,L^k)$.
%$\Box^{(1)}_{\varepsilon,k}$ has $L^2$ closed range, $\Box^{(1)}_{\varepsilon,k}$ is injective and 
%${\rm Range\,}(\Box^{(1)}_{\varepsilon,k})=L^2_{(0,1)}(X,L^k)$. 
Now, we assume that $k\geq N_m$. 
For every $\varepsilon>0$, we can find $v_\varepsilon\in\Omega^{0,1}(X,L^k)$ 
such that $\Box^{(1)}_{\varepsilon,k}v_\varepsilon=g$. Moreover, from \eqref{e-gue141029II} 
and the proof of Proposition~\ref{p-gue141029} (see also \cite[Proposition\,1]{OS00}), 
it is straightforward to see that for every $\varepsilon>0$,           
\begin{equation}\label{e-gue141029III}
\begin{split}
&\norm{v_\varepsilon}_{k}\leq\norm{g}_{k},\ \ \|\ddbar_{b,k}v_\varepsilon\|_{k}\leq\norm{g}_{k},\\
&\|\ol{\pr}^*_{b,k}v_\varepsilon\|_{\ell,k}\leq k^{M(m)}\norm{g}_{\ell,k},\ \ \forall\ell\in\mathbb N_0,\ \ \ell\leq m.
\end{split}
\end{equation}

%We also write $\ol{\pr}^*_{b,k}$ to denote the formal adjoint of $\ddbar_b$. 
From \eqref{e-gue141029III}, we can find $\varepsilon_j\searrow0$ such that $v_{\varepsilon_j}\To\Td v$ 
in $L^2_{(0,1)}(X,L^k)$ as $j\To\infty$, $\ddbar_{b,k}v_{\varepsilon_j}\To\ddbar_{b,k}\Td v$ in 
$L^2_{(0,2)}(X,L^k)$, $\ol{\pr}^*_{b,k}v_{\varepsilon_j}\To\ol{\pr}^*_{b,k}\Td v$ in $H^\ell(X,L^k)$, 
$\forall\ell\in\mathbb N_0$, $\ell\leq m$, and $\Box^{(1)}_{b,k}\Td v=g$ in the sense of distributions. 
Since $\ddbar_{b,k}\Td v\in L^2_{(0,2)}(X,L^k)$, $\ol{\pr}^*_{b,k}\Td v\in H^1(X,L^k)$, we have  
$\Td v\in{\rm Dom\,}\ddbar_{b,k}\cap {\rm Dom\,}\ol{\pr}^*_{b,k}$, 
$\ol{\pr}^*_{b,k}\Td v\in{\rm Dom\,}\ddbar_{b,k}$. 
Note that $\ol{\pr}^*_{b,k}\ddbar_{b,k}\Td v=
g-\ddbar_{b,k}\ol{\pr}^*_{b,k}\Td v\in L^2_{(0,1)}(X,L^k)$. 
From this observation, we can check that $\ddbar_{b,k}\Td v\in{\rm Dom\,}\ol{\pr}^*_{b,k}$. 
Thus, $\Td v\in{\rm Dom\,}\Box^{(1)}_{b,k}$. Since $\Box^{(1)}_{b,k}\Td v=g=\Box^{(1)}_{b,k}v$ 
and $\Box^{(1)}_{b,k}$ is injective, we conclude that $v=\Td v$. 
Thus, $\ol{\pr}^*_{b,k}N^{(1)}_kg=\ol{\pr}^*_{b,k}v\in H^m(X,L^k)$ and 
$\|\ol{\pr}^*_{b,k}N^{(1)}_kg\|_{m,k}\leq k^{M(m)}\norm{g}_{m,k}$. The theorem follows. 
\end{proof} 
%----------------
\begin{thm}\label{t-gue141029I}
With the notations above, for every $m\in\mathbb N$, $m\geq2$, 
there is a $N_m>0$ such that for every $k\geq N_m$, 
\begin{equation}\label{e-gue141029IV}
\mbox{$\Pi_k=I-\ol{\pr}^*_{b,k}N^{(1)}_k\ddbar_{b,k}$ on $\cC^\infty(X,L^k)$},
\end{equation}
\begin{equation}\label{e-gue141029V}
\Pi_k: \cC^\infty(X,L^k)\To H^m(X,L^k)
\end{equation}
and 
\begin{equation}\label{e-gue141029VI}
\norm{(I-\Pi_k)u}_{m,k}\leq k^{M(m)}\norm{\ddbar_{b,k}u}_{m,k},\ \ \forall u\in \cC^\infty(X,L^k),
\end{equation}
where $M(m)>0$ is a constant independent of $k$ and $u$.
\end{thm} 

\begin{proof}
Fix $m\in\mathbb N$, $m\geq2$ and let $N_m>0$ be as in Theorem~\ref{t-gue141029I}. 
We assume that $k\geq N_m$. Let $g\in \cC^\infty(X,L^k)$. 
From Theorem~\ref{t-gue141029}, we know that 
$\ol{\pr}^*_{b,k}N^{(1)}_k\ddbar_{b,k}g\in H^m(X,L^k)$. 
Since $m\geq2$, it is clearly that 
$\ol{\pr}^*_{b,k}N^{(1)}_k\ddbar_{b,k}g\in{\rm Dom\,}\Box^{(0)}_{b,k}$. 
Moreover, it is easy to check that 
\begin{equation}\label{e-gue141029VII}
\ol{\pr}^*_{b,k}N^{(1)}_k\ddbar_{b,k}g\perp{\rm Ker\,}\ddbar_{b,k}=
{\rm Ker\,}\Box^{(0)}_{b,k}.
\end{equation}
We claim that 
\begin{equation}\label{e-gue141029VIII}
g-\ol{\pr}^*_{b,k}N^{(1)}_k\ddbar_{b,k}g\in{\rm Ker\,}\Box^{(0)}_{b,k}.
\end{equation} 
Let $f\in \cC^\infty(X,L^k)$. We have 
\[\begin{split}
&(\,g-\ol{\pr}^*_{b,k}N^{(1)}_k\ddbar_{b,k}g\,|\,\Box^{(0)}_{b,k}f\,)_{k}=
(\,\Box^{(0)}_{b,k}g\,|\,f\,)_{k}-
(\,\ol{\pr}^*_{b,k}N^{(1)}_k\ddbar_{b,k}g\,|\,\Box^{(0)}_{b,k}f\,)_{k}\\
&=(\,\Box^{(0)}_{b,k}g\,|\,f\,)_{k}-
(\,\ddbar_{b,k}g\,|\,N^{(1)}_k\Box^{(1)}_{b,k}\ddbar_{b,k}f\,)_{k}=
(\,\Box^{(0)}_{b,k}g\,|\,f\,)_{k}-(\,\ddbar_{b,k}g\,|\,\ddbar_{b,k}f\,)_{k}=0.
\end{split}\]
The claim \eqref{e-gue141029VIII} follows. From \eqref{e-gue141029VII} and \eqref{e-gue141029VIII}, 
we get \eqref{e-gue141029IV}.
 Theorem~\ref{t-gue141029} and \eqref{e-gue141029IV} yield 
 \eqref{e-gue141029V} and \eqref{e-gue141029VI}.
\end{proof}

From Theorem~\ref{t-gue141029I} and the Sobolev embedding theorem, we get Theorem~\ref{t-mainb}.

\section{Asymptotic expansion of the Szeg\H{o} kernel}\label{s-gue140912} 

In this section, we will prove Theorem~\ref{t-gue150129} and Theorem~\ref{t-main}. 
Let $s$ be a local trivializing section of $L$ 
on an open set $D\subset X$ and let $\Pi_{k,s}$ be the localized operator of $\Pi_k$ 
(see \eqref{e-gue141001}). Let $\mathcal{S}_k$ and $\mathcal{G}_k$ be as in 
Theorem~\ref{t-gue140826a}. 
From the constructions of $\mathcal{G}_k$ and $\mathcal{S}_k$, it is straightforward to see that 
we can find $\Td{\mathcal{G}_k}:H^s_{{\rm comp\,}}(D)\To H^{s+1}_{{\rm loc\,}}(D)$, 
$\Td{\mathcal{S}}_k:H^s_{{\rm comp\,}}(D)\To H^s_{{\rm loc\,}}(D)$, for every $s\in\mathbb Z$, 
such that $\Td{\mathcal{G}_k}$ and $\Td{\mathcal{S}_k}$ 
are properly supported on $D$, 
\begin{equation}\label{e-gue150131II}\begin{split}
&\Td{\mathcal{S}}_k-\mathcal{S}_k=O(k^{-\infty}):
H^s_{{\rm comp\,}}(D)\To H^s_{{\rm loc\,}}(D),\ \ \forall s\in\mathbb Z,\\
&\Td{\mathcal{G}_k}-\mathcal{G}_k=O(k^{-\infty}):
H^s_{{\rm comp\,}}(D)\To H^{s+1}_{{\rm loc\,}}(D),\ \ \forall s\in\mathbb Z,
\end{split}\end{equation}
and 
\begin{equation}\label{e-gue150130abf}
\Td\chi\,\Td{\mathcal{S}}_k\,\chi=O(k^{-\infty}):
H^s_{{\rm comp\,}}(D)\To H^s_{{\rm loc\,}}(D),\ \ \forall s\in\mathbb Z,
\end{equation}
for every $\Td\chi, \chi\in C^\infty_0(D)$ with $\supp\Td\chi\cap\supp\chi=\emptyset$, and
\begin{equation}\label{e-gue141031}
\mbox{$\Box^{(0)}_{s,k}\,\Td{\mathcal{G}_k}+\Td{\mathcal{S}_k}=I+R_k$ on $D$},
\end{equation}
where $R_k$ is properly supported on $D$ and 
\begin{equation}\label{e-gue150130ab}
R_k=O(k^{-\infty}):H^s_{{\rm loc\,}}(D)\To H^{s-1}_{{\rm loc\,}}(D),\ \ \forall s\in\mathbb Z. 
\end{equation}
From \eqref{e-gue141031}, it is easy to see that
\begin{equation}\label{e-gue140912}
\Pi_{k,s}+\Pi_{k,s}R_k=\Pi_{k,s}\Td{\mathcal{S}_k}\ \ \mbox{on $D$}.
\end{equation}

%---------------
\begin{thm}\label{t-gue140912}
With the notations above, for every $\ell\in\mathbb N_0$, there is a $N_\ell>0$ such that for 
every $k\geq N_\ell$\,, $\Td\chi\,\Pi_{k}\chi=O(k^{-\infty}):\cC^\infty(X,L^k)\To\cC^\ell(X,L^k)$, 
for every $\chi\in \cC^\infty_0(D)$, $\Td\chi\in\cC^\infty(X)$ with 
$\supp\Td\chi\cap\supp\chi=\emptyset$, and
\begin{equation} \label{e-gue13630IVabm2}
\Pi_{k,s}-\mathcal{S}_k=O(k^{-\infty}):\cC^\infty_0(D)\To\cC^\ell(D).
\end{equation}
\end{thm} 

\begin{proof}
Fix $\ell\in\mathbb N_0$. From Theorem~\ref{t-gue141029I}, there exists 
$N_\ell>0$ such that for every $k\geq N_\ell$, 
\begin{equation}\label{e-gue141031I}
\begin{split}
&\mbox{$\Pi_k=I-\ol{\pr}^*_{b,k}N^{(1)}_k\ddbar_{b,k}$ on $\cC^\infty(X,L^k)$},\ \ 
\Pi_k: \cC^\infty(X,L^k)\To H^{\ell+n}(X,L^k),\\
&\norm{(I-\Pi_k)u}_{n+\ell,k}\leq k^{M(\ell)}\norm{\ddbar_{b,k}u}_{n+\ell,k},\ \ 
\forall u\in \cC^\infty(X,L^k),
\end{split}
\end{equation}
where $M(\ell)>0$ is a constant independent of $k$ and $u$. 
Now, we assume that $k\geq N_\ell$. 
By the Sobolev embedding theorem we have $H^{\ell+n}(X,L^k)\subset \cC^\ell(X,L^k)$.  

Fix $N_1>0$ and let $u\in\cC^\infty_0(D)$. Consider 
\begin{equation}\label{e-gue141031II}
v=U_{k,s}\Td{\mathcal{S}_k}u-\Pi_k(U_{k,s}\Td{\mathcal{S}_k}u)=(I-\Pi_k)(U_{k,s}\Td{\mathcal{S}_k}u).
\end{equation}
From \eqref{e-gue140912}, we have 
\begin{equation}\label{e-gue141031III}
\begin{split}
\mbox{$v=U_{k,s}(\Td{\mathcal{S}_k}-\Pi_{k,s}\Td{\mathcal{S}_k})u$ on $D$},\ \ 
\mbox{$v=U_{k,s}(\Td{\mathcal{S}_k}u)-\Pi_k(U_{k,s}(I+R_k)u)$ on $X$}.
\end{split}
\end{equation}
From \eqref{e-gue141031I} and \eqref{e-gue141031II}, we obtain 
\begin{equation}\label{e-gue141031IV}
\norm{(I-\Pi_k)(U_{k,s}\Td{\mathcal{S}_k}u)}_{n+\ell,k}\leq 
k^{M(\ell)}\norm{\ddbar_{b,k}(U_{k,s}\Td{\mathcal{S}_k}u)}_{n+\ell,k}.
\end{equation}
Note that 
$\ddbar_{s,k}\Td{\mathcal{S}_k}=O(k^{-\infty}):
H^s_{{\rm comp\,}}(D)\To H^{s-1}_{{\rm loc\,}}(D)$
for all $s\in\mathbb Z$. From this observation, \eqref{e-gue141031IV} and the second formula 
of \eqref{e-gue141031III} we conclude that 
\begin{equation}\label{e-gue141031V}
U_{k,s}\Td{\mathcal{S}_k}-\Pi_kU_{k,s}-\Pi_kU_{k,s}R_k=O(k^{-\infty}):\cC^\infty_0(D)\To\cC^\ell(X,L^k).
\end{equation}
From \eqref{e-gue150130ab} and \eqref{e-gue141031I}, it is easy to see that 
\begin{equation}\label{e-gue150131}
\Pi_kU_{k,s}R_k=O(k^{-\infty}):\cC^\infty_0(D)\To\cC^\ell(X,L^k).
\end{equation}
From \eqref{e-gue141031V} and \eqref{e-gue150131}, we conclude that 
\begin{equation}\label{e-gue150131I}
U_{k,s}\Td{\mathcal{S}_k}-\Pi_kU_{k,s}=O(k^{-\infty}):\cC^\infty_0(D)\To\cC^\ell(X,L^k).
\end{equation}
From \eqref{e-gue150131I} and \eqref{e-gue150131II}, \eqref{e-gue13630IVabm2} follows. 

Finally, from \eqref{e-gue150131I}, \eqref{e-gue150130abf} and noting that $\Td{\mathcal{S}_k}$ is 
properly supported on $D$, we deduce that 
$\Td\chi\,\Pi_{k}\chi=O(k^{-\infty}):\cC^\infty(X,L^k)\To\cC^\ell(X,L^k)$, 
for every $\chi\in \cC^\infty_0(D)$, $\Td\chi\in\cC^\infty(X)$ with 
$\supp\Td\chi\cap\supp\chi=\emptyset$. 
\end{proof}
\begin{proof}[Proof of Theorem~\ref{t-gue150129}]
This follows immediately from Theorems \ref{t-gue150130} and \ref{t-gue140912}.
\end{proof}

\begin{proof}[Proof of Theorem~\ref{t-main}]
Let $\mathcal{A}_k$ be as in Theorem~\ref{t-main}. It is not difficult to see that for every 
$s\in\mathbb Z$ and $N\in\mathbb N$, there exists $n(N,s)>0$ independent of $k$, such that 
\begin{equation}\label{e-gue150131b}
\mathcal{A}_k=O(k^{n(N,s)}):H^s_{{\rm comp\,}}(D)\To\cC^N_0(D).
\end{equation}
From \eqref{e-gue150131b}, \eqref{e-gue13630IVabm2} and since 
$\mathcal{A}_k:H^s_{{\rm comp\,}}(D)\To\cC^\infty_0(D)$ for every $s\in\mathbb Z$, we conclude that 
\begin{equation}\label{e-gue150131bI}
\Pi_{k,s}\mathcal{A}_k\equiv\mathcal{S}_k\mathcal{A}_k\mod O(k^{-\infty}).
\end{equation}
From \eqref{e-gue150131bI} and Theorem~\ref{t-gue140826aI}, Theorem~\ref{t-main} follows. 
\end{proof}

\section{Kodaira Embedding theorem for Levi-flat CR manifolds}\label{s-gue140915}

In this section, we will prove Theorem~\ref{t-embleviflat}. 
Let $s$ be a local trivializing section of $L$ on an open set $D\subset X$. 
Fix $p\in D$ and let $x=(x_1,\ldots,x_{2n-1})$, $z_j=x_{2j-1}+ix_{2j}$, $j=1,\ldots,n-1$, 
be local coordinates of $X$ defined in some small neighbourhood of $p$ such that 
\eqref{e-gue140826m} hold. We may assume that the local coordinates $x$ defined on $D$. 
We write $x'=(x_1,\ldots,x_{2n-2})$. Let $M>1$ be a large constant so that 
\begin{equation}\label{e-gue140916}
\abs{-2{\rm Im\,}\ddbar_b\phi(x)+u\omega_0(x)}^2\leq\frac{M^2}{8},\ \ 
\forall x\in D, \abs{u}\leq1.
\end{equation}
Consider
\[
\tau\in\cC^\infty_0(\Real,[0,1]),\: \text{$\tau=1$ on $\big[\tfrac{1}{4},\tfrac{1}{2}\big]$, 
$\supp\tau\subset[0,1]$},
\]
\[
\chi\in\cC^\infty_0(\Real,[0,1]),\:\text{$\chi=1$ on $\big[-\tfrac{1}{2},\tfrac{1}{2}\big]$, 
$\supp\chi\subset[-1,1]$, $\chi(t)=\chi(-t)$, $t\in\Real$}.
\] 
Fix $0<\delta<1$. Put 
\begin{equation}\label{e-gue140916Ibf}
\alpha_\delta(x,\eta,k):=\tau\Big(\frac{\langle\,\eta\,|\,\omega_0(x)\,\rangle}{\delta}\Big)
\chi\Big(\frac{4\abs{\eta}^2}{M^2}\Big)\in S^0_{{\rm cl\,}}(1,T^*D)
\end{equation}
and let $\mathcal{A}_{k,\delta}$ be a properly supported classical semi-classical pseudodifferential operator on $D$ with 
\[\mathcal{A}_{k,\delta}(x,y)\equiv\frac{k^{2n-1}}{(2\pi)^{2n-1}}
\int e^{ik\langle x-y,\eta\rangle}\alpha_\delta(x,\eta,k)d\eta\mod O(k^{-\infty}).
\]
Fix $\ell\in\mathbb N$, $\ell\geq2$.  In view of Theorem~\ref{t-main}, we see that there is a $N_\ell>0$ 
such that for every $k\geq N_\ell$, 
$\Pi_{k,s}\mathcal{A}_{k,\delta}(x,y)\in \cC^\ell(D\times D)$ and
\begin{equation} \label{e-gue140916I}
(\Pi_{k,s}\mathcal{A}_{k,\delta})(x,y)\equiv\int e^{ik\psi(x,y,u)}a_\delta(x,y,u,k)du\mod O(k^{-\infty})\ \ \text{in $\cC^\ell(D\times D)$},
\end{equation}
where
\begin{equation}  \label{e-gue140916II}
\begin{split}
&a_\delta(x,y,u,k)\in \cC^\infty_0(D\times D\times(-M,M))\cap  S^{n}_{{\rm loc\,},{\rm cl\,}}(1;D\times D\times(-M,M)),\\
&a_\delta(x,y,u,k)\sim\sum^\infty_{j=0}a_{j,\delta}(x,y,u)k^{n-j}\text{ in }S^{n}_{{\rm loc\,}}
(1;D\times D\times(-M,M)).
%a_{j,\delta}(x,y,u)\in \cC^\infty_0(D\times D\times(-M,M)),\ \ \forall j.
\end{split}
\end{equation}
From \eqref{e-gue13630Vabm}, \eqref{e-gue140916} and \eqref{e-gue140916I}, we get 
\begin{equation}\label{e-gue140916III}
a_{0,\delta}(x,x,u)=\frac{1}{2}\pi^{-n}
\abs{\det R^L_x}\tau\Big(\frac{u}{\delta}\Big),\ \ 
\forall (x,x,u)\in D\times D\times(-M,M).
\end{equation}
From now on, we assume that $k\geq N_\ell$. 

We will use the following rescaling of the coordinates:
\[
F^*_k:\Real^{n-1}\to\Real^{n-1},\quad 
F^*_ky:=\Big(\frac{y_1}{\sqrt{k}},\frac{y_2}{\sqrt{k}},\ldots,
\frac{y_{2n-2}}{\sqrt{k}},\frac{y_{2n-1}}{k}\Big).
\]
We introduce the shorthand notations
\[
\boldsymbol{\chi}(y):=\chi(y_1)\ldots\chi(y_{2n-2})
\chi(y_{2n-1}),
\]
\[\boldsymbol{\chi}(k,y):=\chi(\sqrt{k}y_1)\ldots\chi(\sqrt{k}y_{2n-2})
\chi(ky_{2n-1}).\]
hence $\boldsymbol{\chi}(y)=\boldsymbol{\chi}(k,F^*_ky)$.

For $j=1,\ldots,n-1$, let $\lambda_j=\lambda_j(p)$ are the eigenvalues of 
$R^L_p$ with respect to $\langle\,\cdot\,|\,\cdot\,\rangle$
and let $\alpha_j\in\Complex$ 
be as in \eqref{e-gue140826m}. Set
\[
R(w)=\sum^{n-1}_{l=1}(\alpha_lw_l-\ol\alpha_l\ol w_l)
+iuy_{2n-1}+\frac{1}{2}\sum^{n-1}_{j=1}\lambda_j\abs{w_j}^2
\]
where $w_j=y_{2j-1}+iy_{2j}$.
Let 
\begin{equation}\label{e-gue140916IV}
u_{k,\delta,p}:=\Pi_kU_{k,s}\mathcal{A}_{k,\delta}
\Bigl(e^{kR(w)}
\boldsymbol{\chi}(k,y)\Bigr),
\end{equation}
so $u_{k,\delta,p}$ is a global $\cC^\ell$ CR section. We write 
$u_{k,\delta,p}=U_{k,s}\Td u_{k,\delta,p}$ on $D$, with $\Td u_{k,\delta,p}\in C^{\ell}(D)$. Then, 
$\abs{u_{k,\delta,p}(x)}_{h^{k}}=\abs{\Td u_{k,\delta,p}(x)}$, $x\in D$.
Put  
\[
\begin{split}
\psi_0(x,y,u)&:=\psi(x,y,u)-i\sum^{n-1}_{j=1}(\alpha_jw_j-\ol\alpha_j\ol w_j)
+uy_{2n-1}-\frac{i}{2}\sum^{n-1}_{j=1}\lambda_j\abs{w_j}^2\\
&=\psi(x,y,u)-iR(w).
\end{split}
\]
From \eqref{e-gue140916I}, we can check that we have 
$\operatorname{mod} O(k^{-\infty})$ in $\cC^\ell(D)$,
\begin{equation}\label{e-gue140916V}
\begin{split}
\Td u_{k,\delta,p}(x)
&\equiv\int e^{ik\psi_0(x,y,u)}a_\delta(x,y,u,k)\boldsymbol{\chi}(k,y)
\\
&\equiv\int e^{ik\psi_0(x,F^*_ky,u)}k^{-n}a_\delta(x,F^*_ky,u,k)
\boldsymbol{\chi}(y)dudy.
\end{split}
\end{equation}
Put
\begin{equation}\label{e-gue140917}
\widehat u_{k,\delta,p}:=\exp\Bigl(\!-k\sum\limits^{n-1}_{j=1}(\alpha_jz_j-\ol\alpha_j\ol z_j)\!\Bigr)
\Td u_{k,\delta,p}\in \cC^\ell(D).
\end{equation}
%-------------
\begin{lem}\label{l-gue140916}
With the notations above, there is a $k_0>0$ such for all $k\geq k_0$ and $p\in X$, 
\begin{equation}\label{e-gue140916VI}
\begin{split}
\frac{1}{8}\delta c_p\leq\abs{\widehat u_{k,\delta,p}(p)}\leq 2\delta c_p,\ \ \frac{1}{32}\delta^2 c_p\leq \abs{\frac{1}{k}\frac{\pr\widehat u_{k,\delta,p}}{\pr x_{2n-1}}(p)}\leq 2\delta^2 c_p,\ \ \abs{\frac{1}{k}\frac{\pr\widehat u_{k,\delta,p}}{\pr x_{j}}(p)}\leq\delta^4,
\end{split}
\end{equation}
where $j=1,2,\ldots,2n-2$, and $c_p=\frac{1}{2}\pi^{-n}\abs{\det R^L_p}\int\boldsymbol{\chi}(y)dy$.
\end{lem}

\begin{proof}
From \eqref{e-gue140916V}, \eqref{e-gue140916III}, \eqref{e-gue140826Im} 
and note that $\psi_0(0,0,u)=0$, $\forall u\in\Real$, we can check that 
\[\begin{split}
&\lim_{k\To\infty}\abs{\widehat u_{k,\delta,p}(p)}
=\frac{1}{2}\pi^{-n}\abs{\det R^L_p}\int\tau\Big(\frac{u}{\delta}\Big)\boldsymbol{\chi}(y)dydu,\\
&\lim_{k\To\infty}\abs{\frac{1}{k}\frac{\pr\widehat u_{k,\delta,p}}{\pr x_{2n-1}}(p)}=\frac{1}{2}\pi^{-n}\abs{\det R^L_p}\int u\tau\Big(\frac{u}{\delta}\Big)\boldsymbol{\chi}(y)dydu,\\
&\lim_{k\To\infty}\abs{\frac{1}{k}\frac{\pr\widehat u_{k,\delta,p}}{\pr x_{j}}(p)}=0,\ \ j=1,2,\ldots,2n-2.
\end{split}\]
Since $\frac{\delta}{4}\leq\int\tau\big(\frac{u}{\delta}\big)du\leq\delta$ and 
$\frac{\delta^2}{16}\leq\int u\tau\big(\frac{u}{\delta}\big)du\leq\delta^2$,  there is $k_0>0$ such that for every $k\geq k_0$, \eqref{e-gue140916VI} hold. Since $X$ is compact, $k_0$ can be taken to be independent of the point $p$. 
\end{proof} 

For every $j=1,2,\ldots,n-1$, let 
\begin{equation}\label{e-gue140916a}
\begin{split}
&u^j_{k,\delta,p}:=
\Pi_kU_{k,s}\mathcal{A}_{k,\delta}
\Bigl(e^{kR(w)}
\sqrt{k}(y_{2j-1}+iy_{2j})\boldsymbol{\chi}(k,y)\Bigr).
\end{split}
\end{equation}
Then, $u^j_{k,\delta,p}$ is a global $\cC^\ell$ CR section. On $D$, we write 
$u^j_{k,\delta,p}=U_{k,s}\Td u^j_{k,\delta,p}$, with
$\Td u^j_{k,\delta,p}\in \cC^\ell(D)$. 
From \eqref{e-gue140916I}, we can check that
\begin{equation}\label{e-gue140916aI}
\begin{split}
&\Td u^j_{k,\delta,p}(x)
\equiv\int e^{ik\psi_0(x,F^*_ky,u)}
k^{-n}a_\delta(x,F^*_ky,u,k)(y_{2j-1}+iy_{2j})\boldsymbol{\chi}(y)dudy,
\end{split}
\end{equation}
$\operatorname{mod} O(k^{-\infty})\ \ \mbox{in $\cC^\ell(D)$}$. 
Put
\begin{equation}\label{e-gue140917I}
\widehat u^j_{k,\delta,p}:=
\exp\!\Big(\!\!-k\sum\limits^{n-1}_{l=1}(\alpha_lz_l-\ol\alpha_l\ol z_l)\!\Big)
\Td u^j_{k,\delta,p}\in \cC^\ell(D),\ \ j=1,2,\ldots,n-1.
\end{equation}
%------------------
\begin{lem}\label{l-gue140916I}
With the notations above, there exists $k_0>0$ such that for all $p\in X$ and $k\geq k_0$\,, 
\begin{equation}\label{e-gue140916b}
\begin{split}
&\abs{\widehat u^j_{k,\delta,p}(p)}\leq\delta^4,\:\:
\abs{\frac{1}{k}\frac{\pr\widehat u^j_{k,\delta,p}}{\pr x_{2n-1}}(p)}\leq\delta^4,\ \ 
\abs{\frac{1}{k}\frac{\pr\widehat u^j_{k,\delta,p}}{\pr z_{j}}(p)}
\geq\frac{1}{8}\delta\lambda_jd_p,\ \ j=1,2,\ldots,n-1,\\
&\abs{\frac{1}{k}\frac{\pr\widehat u^j_{k,\delta,p}}{\pr\ol z_{s}}(p)}
\leq\delta^4,\ \ j,s=1,2,\ldots,n-1,\\
&\abs{\frac{1}{k}\frac{\pr\widehat u^j_{k,\delta,p}}{\pr z_{s}}(p)}
\leq\delta^4,\ \ j,s=1,2,\ldots,n-1,\ \ j\neq s,\\
%&\abs{\frac{1}{k}\frac{\pr\widehat u^j_{k,\delta,p}}{\pr z_{j}}(p)}\geq\frac{1}{8}\delta\lambda_jd_p,\ \ j=1,2,\ldots,n-1,
\end{split}
\end{equation}
where $\{\lambda_j\}_{j=1}^{n-1}$ 
are the eigenvalues of 
$R^L_p$ with respect to $\langle\,\cdot\,|\,\cdot\,\rangle$
and \[d_p=
\frac{1}{2\pi^{n}}\abs{\det R^L_p}\int\abs{y_1+iy_2}^2\boldsymbol{\chi}(y)dy.\]
\end{lem} 

\begin{proof}
From \eqref{e-gue140916aI}, \eqref{e-gue140916III}, \eqref{e-gue140826Im} 
and observing that $\psi_0(0,0,u)=0$ for all $u\in\Real$, it is straightforward to check that 
for every $j,s,t=1,\ldots,n-1$, $s\neq j$,
\[\begin{split}
\lim_{k\To\infty}\Big|\frac{1}{k}&\frac{\pr\widehat u^j_{k,\delta,p}}{\pr z_{j}}(p)\Big|\\
&=\frac{\lambda_j}{2\pi^{n}}\abs{\det R^L_p}\int\tau\!\left(\frac{u}{\delta}\right)
\abs{y_{2j-1}+iy_{2j}}^2\boldsymbol{\chi}(y)dydu,
\end{split}\]
\[
\lim_{k\To\infty}\abs{\widehat u^j_{k,\delta,p}(p)}=
\lim_{k\To\infty}\abs{\frac{1}{k}\frac{\pr\widehat u^j_{k,\delta}}{\pr x_{2n-1}}(p)}=
\lim_{k\To\infty}\abs{\frac{1}{k}\frac{\pr\widehat u^j_{k,\delta,p}}{\pr z_{s}}(p)}=
\lim_{k\To\infty}\abs{\frac{1}{k}\frac{\pr\widehat u^j_{k,\delta}}{\pr\ol z_{t}}(p)}=0.
\]
%where $s,t=1,\ldots,n-1$, $s\neq j$.
Since $\frac{\delta}{4}\leq\int\tau\Big(\frac{u}{\delta}\Big)du\leq\delta$, there is a constant $k_0>0$ 
such that \eqref{e-gue140916b} holds for every $k\geq k_0$. Since $X$ is compact, $k_0$ can be 
taken to be independent of the point $p$. The lemma follows.
\end{proof} 
%---------
Consider the $\cC^\ell$ map 
\begin{equation}\label{e-gue140916bII}
\Phi_{k,\delta,p}:D\to\Complex^{n},\:\: x\longmapsto \left(\frac{\Td u_{k,\delta,p}}{\Td u_{k,\delta^2,p}}(x),
\frac{\Td u^1_{k,\delta,p}}{\Td u_{k,\delta^2,p}}(x),\ldots,
\frac{\Td u^{n-1}_{k,\delta,p}}{\Td u_{k,\delta^2,p}}(x)\right).
\end{equation}

The following Lemma is a consequence of \eqref{e-gue140916b} and \eqref{e-gue140916VI} together with a 
straightforward computation and therefore we omit the details. 

\begin{lem}\label{l-gue140916II}
With the notations above, there are $k_0>0$ and $0<\delta_0<1$  
such that for all $k\geq k_0$, $0<\delta\leq\delta_0$ and $p\in X$, the differential of 
$\Phi_{k,\delta,p}$ is injective at $p$\,.
\end{lem} 

Let $\dist (\cdot,\cdot)$ denote the Riemannian distance on $X$ and for $x\in X$ and $r>0$, 
put $B(x,r):=\set{y\in X;\, \dist (x,y)<r}$. From now on, we fix $k>k_0$ and $0<\delta<\delta_0$, 
where $k_0>0$ and $0<\delta_0<1$ are as in Lemma~\ref{l-gue140916II}. 
Since $X$ is compact there exists $r_k>0$ 
such that for every $x_0\in X$, $\Td u_{k,\delta^2,x_0}(x)\neq0$ for every 
$x\in B(x_0,2r_k)$ and the maps $\Phi_{k,\delta,x_0}$ and $d\Phi_{k,\delta,x_0}$
are injective on $B(x_0,2r_k)$. We can find $x_1,x_2,\ldots,x_{d_k}\in X$ such that
\begin{equation}\label{e-gue140917II}
X=B(x_1,r_k)\cup B(x_2,r_k)\cup\ldots\cup B(x_{d_k},r_k).
\end{equation}
For every $j=1,2,\ldots,d_k$, let 
$u_{k,\delta^2,x_j}, u_{k,\delta,x_j},u^1_{k,\delta,x_j},\ldots,u^{n-1}_{k,\delta,x_j}\in \cC^\ell(X,L^k)$
be as in \eqref{e-gue140916IV} and \eqref{e-gue140916a}. Consider the map
$\Phi_{k,\delta}:X\To\Complex\mathbb P^{(n+1)d_k-1}$,
\begin{equation}\label{e-gue140917III}
\begin{split}
\Phi_{k,\delta}\!=\!\big[u_{k,\delta^2,x_1},u_{k,\delta,x_1},u^1_{k,\delta,x_1},\ldots,
u^{n-1}_{k,\delta,x_1},\ldots,u_{k,\delta^2,x_{d_k}},
u_{k,\delta,x_{d_k}},u^1_{k,\delta,x_{d_k}},
\ldots,u^{n-1}_{k,\delta,x_{d_k}}\big].
\end{split}
\end{equation}
Let $q\in X$. Then, $q\in B(x_j,r_k)$ for some $j=1,2,\ldots,d_k$. From the discussion before \eqref{e-gue140917II}, we see that $u_{k,\delta^2,x_j}(q)\neq0$. Thus, $\Phi_{k,\delta}$ is well-defined as a $\cC^\ell$ map. 

\begin{thm}\label{t-gue140917}
With the notations above, the differential of $\Phi_{k,\delta}$ is injective at every $x\in X$ and for every $x_0, y_0\in X$ with $\dist (x_0,y_0)\leq\frac{r_k}{2}$, we have $\Phi_{k,\delta}(x_0)\neq\Phi_{k,\delta}(y_0)$.
\end{thm}

\begin{proof}
Let $q\in X$. Assume that $q\in B(x_1,r_k)$. Then, $u_{k,\delta^2,x_1}(q)\neq0$. On $B(x_1,r_k)$, consider the map $\Psi:B(x_1,r_k)\To\Complex^{(n+1)d_k-1}$,
\begin{equation}\label{e-gue140917IV}
\begin{split}
\Psi=\Bigl(\frac{u_{k,\delta,x_1}}{u_{k,\delta^2,x_1}},
\frac{u^1_{k,\delta,x_1}}{u_{k,\delta^2,x_1}},\ldots,
\frac{u^{n-1}_{k,\delta,x_1}}{u_{k,\delta^2,x_1}},
\ldots,\frac{u_{k,\delta^2,x_{d_k}}}{u_{k,\delta^2,x_1}},
\frac{u_{k,\delta,x_{d_k}}}{u_{k,\delta^2,x_1}},
\frac{u^1_{k,\delta,x_{d_k}}}{u_{k,\delta^2,x_1}},
\ldots,\frac{u^{n-1}_{k,\delta,x_{d_k}}}{u_{k,\delta^2,x_1}}\Bigr).
\end{split}
\end{equation}
From the discussion before \eqref{e-gue140917II}, we see that $d\Phi_{k,\delta,x_1}$ is injective on $B(x_1,2r_k)$. 
Thus, $d\Psi$ is injective at $q$ and hence $d\Phi_{k,\delta}$ is injective at $q$. 

Let $x_0, y_0\in X$ with $\dist (x_0,y_0)\leq\frac{r_k}{2}$. We may assume that $x_0\in B(x_1,r_k)$. Thus, $x_0, y_0\in B(x_1,2r_k)$. From the discussion before \eqref{e-gue140917II}, we see that $\Phi_{k,\delta,x_1}$ is injective on $B(x_1,2r_k)$. Hence, 
\begin{equation}\label{e-gue140917V}
\Phi_{k,\delta,x_1}(x_0)\neq\Phi_{k,\delta,x_1}(y_0).
\end{equation}
By the definition \eqref{e-gue140916bII} of $\Phi_{k,\delta,x_1}$, relation
\eqref{e-gue140917V} implies that $\Phi_{k,\delta}(x_0)\neq\Phi_{k,\delta}(y_0)$. 
The lemma follows. 
\end{proof}

%By Theorem~\ref{t-gue140917}, $X$ can be $\cC^\ell$ CR immersed into projective space by 
%the map $\Phi_{k,\delta}$ and $\Phi_{k,\delta}$ separates points 
%$x_0, y_0\in X$ with $\dist (x_0,y_0)\leq\frac{r_k}{2}$. 

Let $s$ be a local trivializing section of $L$ on an open set $D\subset X$. As before, we fix 
$p\in D$ and let $x=(x_1,\ldots,x_{2n-1})$, $z_j=x_{2j-1}+ix_{2j}$, $j=1,\ldots,n-1$, 
be local coordinates of $X$ defined in some small neighbourhood of $p$ such that \eqref{e-gue140826m} hold. 
We may assume that the local coordinates $x$ defined on $D$. Take $m>N_\ell$ be a large constant and let 
$u_{m,\delta,p}$ be as in \eqref{e-gue140916IV}. On $D$, we write 
$u_{m,\delta,p}=U_{k,s}\Td u_{m,\delta,p}$, $\Td u_{m,\delta,p}\in \cC^\ell(D)$.
Put $D_{p,m}:=\set{x=(x_1,\ldots,x_{2n-1});\, \abs{x}<\frac{1}{m\log m}}$. We need the following.

\begin{lem}\label{l-gue140919}
With the notations above, there exists $m_0>0$ such that
$r_km_0^{1/3}>4$ and for all $m\geq m_0$ and $p\in X$, 
\begin{equation}\label{e-gue140919II}
\inf\set{\abs{u_{m,\delta,p}(x)}_{h^{m}};\, x\in D_{p,m}}\geq\frac{1}{8}\delta c_p,
\end{equation}
where $c_p=\frac{1}{2}\pi^{-n}\abs{\det R^L_p}\int\boldsymbol{\chi}(y)dy$, and 
for every $q\in X$ with $\dist (q,x)\geq\frac{r_k}{4}$, for all $x\in D_{p,m}$, we have
\begin{equation}\label{e-gue140919III}
\abs{u_{m,\delta,p}(q)}_{h^{m}}\leq\frac{1}{2}\inf\set{\abs{u_{m,\delta,p}(x)}_{h^{m}};\, x\in D_{p,m}},
\end{equation}
where $r_k>0$ is as in  Theorem~\ref{t-gue140917}. 
\end{lem}

\begin{proof}
Let $m>N_\ell$ be large enough so that
\begin{equation}\label{e-gue140919}
r_km^{1/3}>4.
\end{equation}
As in \eqref{e-gue140916V}, we have $\operatorname{mod} O(m^{-\infty})$ in $\cC^\ell(D)$
\begin{equation}\label{e-gue140919I}
\begin{split}
\Td u_{m,\delta,p}(x)\equiv\int e^{im\psi_0(x,F^*_my,u)}m^{-n}a_\delta(x,F^*_my,u,m)
\boldsymbol{\chi}(y)dudy. 
\end{split}
\end{equation}
From \eqref{e-gue140919I}, we can repeat the proof of the first formula of \eqref{e-gue140916VI} 
with minor changes and get \eqref{e-gue140919II}. We only need to prove \eqref{e-gue140919III}. 
Let $q\in X$ with $\dist (q,x)\geq\frac{r_k}{4}$, for all $x\in D_{p,m}$. 
If $q\notin D$, from (i) in Theorem~\ref{t-gue150129}, we can check that $\abs{u_{m,\delta,p}(q)}_{h^{m}}=O(m^{-\infty})$.

We may thus assume that $q\in D$. 
For simplicity, we may suppose that $\operatorname{dist}(x_1,x_2)=\abs{x_1-x_2}$ on $D$. 
We write $q=(q_1,\ldots,q_{2n-1})$. Since $\dist (q,x)\geq\frac{r_k}{4}$, for all $x\in D_{p,m}$, 
from \eqref{e-gue140919}, we have $\abs{q}\geq\frac{1}{4m^{1/3}}$ for $m$ large. 
Thus, $\abs{q'}\geq\frac{1}{8m^{1/3}\log m}$ or $\abs{q_{2n-1}}\geq\frac{1}{8m^{1/3}}$, 
where $q'=(q_1,\ldots,q_{2n-2})$. If $\abs{q'}\geq\frac{1}{8m^{1/3}\log m}$, by using the fact that 
$m\,{\rm Im\,}\psi_0(q,F^*_my,u)\geq cm^{1/3}\frac{1}{(\log m)^2}$, 
$\forall y\in\supp\boldsymbol{\chi}(y)$,
where $c>0$ is a constant independent of $m$, we conclude that
\begin{equation}\label{e-gue140919a}
\abs{\Td u_{m,\delta,p}(q)}=O(m^{-\infty})\,,\:\: \text{if $\abs{q'}\geq\frac{1}{8m^{1/3}\log m}\,\cdot$}
\end{equation}
If $\abs{q_{2n-1}}\geq\frac{1}{8m^{1/3}}$ and $\abs{q'}<\frac{1}{8m^{1/3}\log m}$, 
from \eqref{e-gue140826Im}, we can integrate by parts with respect to $u$ several times and conclude that 
\begin{equation}\label{e-gue140919aI}
\abs{\Td u_{m,\delta,p}(q)}=O(m^{-\infty})\,,\:\: \text{if 
$\abs{q_{2n-1}}\geq\frac{1}{8m^{1/3}\log m}$ and $\abs{q'}<\frac{1}{8m^{1/3}\log m}\,\cdot$}
\end{equation}
From \eqref{e-gue140919a} and \eqref{e-gue140919aI}, \eqref{e-gue140919III} follows. 
\end{proof} 

Now, we fix $m\geq N_\ell+m_0$, where $m_0$ is as Lemma~\ref{l-gue140919}. From Lemma~\ref{l-gue140919}, we see that we can find $x_1\in X, x_2\in X,\ldots,x_{d_m}\in X$ such that $X=\bigcup^{d_m}_{j=1}U_{x_j,m}$, where for each $j$, $U_{x_j,m}$ is an open neighbourhood of $x_j$ with $\sup\{\dist (q_1,q_2);\, q_1, q_2\in U_{x_j,m}\}<\frac{r_k}{4}$, and for each $j$, we can find a $\cC^\ell$ global CR section $u_{m,\delta,x_j}$ such that 
\begin{equation}\label{e-gue140919b}
\inf\set{\abs{u_{m,\delta,x_j}(x)}_{h^{m}};\, x\in U_{x_j,m}}>0,
\end{equation}
and 
for every $q\in X$ with $\dist (q,x)\geq\frac{r_k}{4}$, for all $x\in U_{x_j,m}$, we have
\begin{equation}\label{e-gue140919bI}
\abs{u_{m,\delta,x_j}(q)}_{h^{m}}\leq\frac{1}{2}\inf\set{\abs{u_{m,\delta,x_j}(x)}_{h^{m}};\, x\in U_{x_j,m}},
\end{equation}
where $r_k>0$ is as in Theorem~\ref{t-gue140917}. Consider the map:
\begin{equation}\label{e-gue140919bII}
\begin{split}
\Psi_{m,\delta}:X\To\Complex\mathbb P^{d_m-1},\:\:
x\longmapsto[u_{m,\delta,x_1},u_{m,\delta,x_2},\ldots,u_{m,\delta,x_{d_m}}](x).
\end{split}
\end{equation}
Let $q\in X$. Then, $q\in U_{x_j,m}$ for some $j=1,2,\ldots,d_m$. In view of \eqref{e-gue140919b}, we see that $u_{m,\delta,x_j}(q)\neq0$. Thus, $\Psi_{m,\delta}$ is well-defined as a smooth map.
%--------------
\begin{thm}\label{t-gue140919}
The map $(\Phi_{k,\delta},\Psi_{m,\delta}):X\To\Complex\mathbb P^{(n+1)d_k-1}\times\Complex\mathbb P^{d_m-1}$
is a $\cC^\ell$ CR embedding, where $\Phi_{k,\delta}$ is given by \eqref{e-gue140917III}
\end{thm}
%--------------
\begin{proof}
In view of Theorem~\ref{t-gue140917}, we only need to show that $(\Phi_{k,\delta},\Psi_{m,\delta})$ is injective. 
Let $q_1, q_2\in X$, $q_1\neq q_2$. Assume first that $\dist (q_1,q_2)\leq\frac{r_k}{4}$. 
From Theorem~\ref{t-gue140917}, we know that $\Phi_{k,\delta}(q_1)\neq\Phi_{k,\delta}(q_2)$ 
and hence $(\Phi_{k,\delta}(q_1),\Psi_{m,\delta}(q_1))\neq(\Phi_{k,\delta}(q_2),\Psi_{m,\delta}(q_2))$. 
We assume that $\dist (q_1,q_2)>\frac{r_k}{4}$. From \eqref{e-gue140919bI}, 
it is straightforward to check that $\Psi_{m,\delta}(q_1)\neq\Psi_{m,\delta}(q_2)$ and thus 
$(\Phi_{k,\delta}(q_1),\Psi_{m,\delta}(q_1))\neq(\Phi_{k,\delta}(q_2),\Psi_{m,\delta}(q_2))$. The theorem follows.
\end{proof}

Note that $\Phi_{k,\delta}$ are defined by collecting many local embedding CR maps and it is difficult to show that $\Phi_{k,\delta}$ is injective on $X$. 

\begin{proof}[Proof of Theorem~\ref{t-embleviflat}]
With the notations above, consider the Segre map
\begin{equation}\label{e:segre}
\begin{split}
\Upsilon:\Complex\mathbb P^{(n+1)d_k-1}\times\Complex\mathbb P^{d_m-1}
&\To\Complex\mathbb P^{(n+1)d_kd_m-1},\\
([z_1,\ldots,z_{(n+1)d_k}],[w_1,\ldots,w_{d_m}])
&\To[z_1w_1,z_1w_2,\ldots,z_1w_{d_m},z_2w_1,\ldots,z_{(n+1)d_k}w_{d_m}],
\end{split}
\end{equation}
which is a holomorphic embedding. By
Theorem~\ref{t-gue140919}, we deduce that
\[\Upsilon\circ(\Phi_{k,\delta},\Psi_{m,\delta}):X\To\Complex\mathbb P^{(n+1)d_kd_m-1},\]
is a $\cC^\ell$ CR embedding. We have proved that for every $M\geq k+N_\ell+m_0$,
we can find CR sections $s_0, s_1,\ldots,s_{d_M}\in \cC^\ell(X,L^M)$, 
such that the map $x\in X\To [s_0(x),s_1(x),\ldots,s_{d_M}(x)]\in\Complex\mathbb P^{d_M}$ is an embedding. 
Theorem~\ref{t-embleviflat} follows. 
\end{proof}

\section*{Acknowledgements}
We are grateful to Masanori Adachi and Xiaoshan Li for several useful conversations.
We would like to thank the referees for their insightful and stimulating reports.

The first-named author was partially supported by Taiwan Ministry of Science of Technology project 
103-2115-M-001-001, 104-2628-M-001-003-MY2 and the Golden-Jade fellowship of Kenda Foundation.
The second-named author was partially supported by the DFG project 
SFB TRR 191 and Universit\'e Paris 7.

%\bibliographystyle{mrl}
%\bibliography{mmbook}

\def\cprime{$'$}

\end{document}